\documentclass[11pt]{article}
\usepackage{graphicx} 

\usepackage[a4paper,margin=2cm]{geometry}
\usepackage{enumerate}

\usepackage{setspace}

\DeclareFontEncoding{LS1}{}{}
\DeclareFontSubstitution{LS1}{stix}{m}{n}
\DeclareSymbolFont{symbols2}{LS1}{stixfrak}{m}{n}
\DeclareMathSymbol{\typecolon}{\mathbin}{symbols2}{"25}

\usepackage{pdfpages}

\usepackage{eurosym}

\usepackage{amsmath,amsfonts,amssymb,graphics,amsthm, comment}
\usepackage{hyperref}
\hypersetup{
    colorlinks=true,
    linkcolor=blue,
    citecolor=red,
    urlcolor=blue,
    pdfborder={0 0 0}
}
\usepackage{enumitem}
\usepackage{overpic}

\definecolor{mygreen}{rgb}{0,0.2,0.8}

\usepackage{xcolor}
\usepackage{transparent}
\usepackage{bbm}
\usepackage{wrapfig} 
\usepackage{enumerate}

\usepackage[font=sf, labelfont={sf,bf}, margin=2cm]{caption}

\usepackage{cleveref}
  \crefname{theorem}{Theorem}{Theorems}
  \crefname{thm}{Theorem}{Theorems}
  \crefname{thm*}{Theorem*}{Theorems}
    \crefname{problem}{Problem}{Theorems}
  \crefname{lemma}{Lemma}{Lemmas}
  \crefname{lem}{Lemma}{Lemmas}
  \crefname{remark}{Remark}{Remarks}
  \crefname{prop}{Proposition}{Propositions}
\crefname{notation}{Notation}{Notations}
\crefname{claim}{Claim}{Claims}
  \crefname{defn}{Definition}{Definitions}
  \crefname{corollary}{Corollary}{Corollaries}
  \crefname{section}{Section}{Sections}
  \crefname{figure}{Figure}{Figures}
    \crefname{assumption}{Assumption}{Assumptions}

\newtheorem{thm}{Theorem}[section]
\newtheorem{theorem}[thm]{Theorem}

\newtheorem{thm*}{Theorem*}[section]

\newtheorem{lemma}[thm]{Lemma}
\newtheorem{lem}[thm]{Lemma}

\newtheorem{corollary}[thm]{Corollary}
\newtheorem{prop}[thm]{Proposition}
\newtheorem{defn}[thm]{Definition}

\newtheorem{assumption}[thm]{Assumption}
\numberwithin{equation}{section}

\theoremstyle{definition}
\newtheorem{remark}[thm]{Remark}
\newtheorem{rmk}[thm]{Remark}

\def\cW{\mathcal{W}}

\def\cT{\mathcal{T}}

\def\cO{\mathcal{O}}

\def\cM{\mathcal{M}}
\def\cL{\mathcal{L}}

\def\cH{\mathcal{H}}
\def\cG{\mathcal{G}}
\def\cF{\mathcal{F}}

\def\cD{\mathcal{D}}
\def\cC{\mathcal{C}}
\def\cB{\mathcal{B}}

\def\PP{\mathbb{P}}
\def\EE{\mathbb{E}}
\def\CC{\mathbb{C}}
\def\RR{\mathbb{R}}
\def\ZZ{\mathbb{Z}}

\def\DD{\mathbb{D}}

\def\HH{\mathbb{H}}

\def  \p- {p\textunderscore}

\def\eps{\varepsilon}

\def\ph{\varphi}

\usepackage{multicol}

\DeclareMathOperator{\dist}{dist}

\DeclareMathOperator{\Tr}{Tr}

\DeclareMathOperator\supp{supp}

\newcommand{\ii}{\mathbf{i}}

\DeclareMathOperator{\diam}{Diam}

\newcommand{\nb}[1]{{\color{red}{[#1]}}}
\newcommand{\lr}[1]{{\color{blue}{#1}}}

\newcommand{\Ebr}[3]{\mathbb{E}_{#1 \overset{#3}{\to} #2}}

\newcommand{\Pbr}[3]{\mathbb{P}_{#1 \overset{#3}{\to} #2}}

\newcommand{\Primal}{\Gamma}
\newcommand{\Primalbis}{\Lambda}
\newcommand{\infPrimal}{\Primal^{\infty}}
\newcommand{\closurePrimal}{\overline{\Primal}}

\newcommand{\Dual}{\Primal^{\star}}
\newcommand{\Dualbis}{\Primalbis^{\star}}
\newcommand{\infDual}{(\Primal^{\infty})^*}
\newcommand{\Black}{B}
\newcommand{\White}{W}
\newcommand{\infBlack}{B^{\infty}}
\newcommand{\infWhite}{W^{\infty}}

\newcommand{\Int}{\mathrm{Int}}

\newcommand{\Rot}{\mathrm{Rot}}

\usepackage{slashed}

\newcommand{\mass}{M}
\newcommand{\weight}{\mu}

\newcommand{\mGreen}{G^m}

\newcommand{\dGreen}{G_\#}
\newcommand{\dfreeGreen}{G_{\#,\infty}}
\newcommand{\dmGreen}{\dGreen^m}
\newcommand{\dDelta}{\Delta_\#}
\newcommand{\dDirichletDelta}{\Delta_\#^d}

\newcommand{\dpartial}{\partial_\#}
\newcommand{\dbarpartial}{\overline{\partial}_\#}
\newcommand{\Proj}{\mathrm{Proj}}
 
\newcommand{\dd}{d_\#}
\newcommand{\daverage}{A_\#}
\newcommand{\daveragetwo}{(A_2)_\#}
\newcommand{\daverageone}{(A_1)_\#}
\newcommand{\dpartialone}{(\partial_1)_\#}
\newcommand{\dpartialtwo}{(\partial_2)_\#}
\renewcommand{\path}{\gamma}
\newcommand{\dpath}{\gamma_\#}

\newcommand{\dCRx}{\mathrm{d}_x}
\newcommand{\dCRy}{\mathrm{d}_y}

\newcommand{\tx}{\tilde{x}}
\newcommand{\ty}{\tilde{y}}
\newcommand{\tw}{\tilde{w}}
\newcommand{\ttheta}{\tilde{\theta}}

\newcommand{\kappabar}{\overline{\kappa}}
\newcommand{\kappastarbar}{\overline{\kappa^*}}
\newcommand{\kappastar}{\kappa^*}

\newcommand{\diagonal}{\mathcal{D}}

\newcommand{\ReIm}{\mathfrak{P}}

\newcommand{\dirac}{\overline{\partial}}

\newcommand{\tH}{\tilde{H}}
\newcommand{\tcC}{\tilde{\cC}}

\newcommand{\sgn}{\mathrm{sgn}}

\newcommand{\dz}{\partial}
\newcommand{\dzbar}{\overline{\partial}}

\newcommand{\BV}{\mathrm{BV}}
\newcommand{\CSG}{\mathrm{CSG}}

\newcommand{\Cauchy}{\mathcal{K}}

\newcommand{\partialout}{\partial_{\mathrm{out}}}

\title{Massive holomorphicity of 
near-critical dimers\\ and sine-Gordon model}
\author{Nathana\"el Berestycki\thanks{Fakultät für Mathematik,University of Vienna, Austria, \texttt{nathanael.berestycki@univie.ac.at}} \and Scott Mason\thanks{Courant Institute, New York University, \texttt{sm12814@nyu.edu}} \and 
Lucas Rey
  \thanks{PSL University-Dauphine, CNRS, UMR 7534, CEREMADE, 75016 Paris, France,\texttt{ lucas.rey@inrae.fr}}
\date{\today}
}
\begin{document}

\maketitle

\begin{abstract}
In this paper, we consider the near-critical dimer model in the setup of isoradial superpositions with Temperleyan boundary conditions. We show that the centered height function converges as the mesh size tends to zero to a limiting field which agrees with the (electromagnetically tilted) sine-Gordon model, whose derivative correlations are described by Grassmann variables
(or equivalently determinants involving a massive Dirac operator). This answers a longstanding question in the field. A crucial part of the work is to develop a notion of discrete massive holomorphic functions and the tools to study such functions, in particular finding an exact discrete form of the massive Cauchy--Riemann equations, which is satisfied by the inverse Kasteleyn matrix. In comparison with previous studies, a key novelty of this part of our work is that the mass is not only allowed to be non-constant but can be complex-valued. 
\end{abstract}

\tableofcontents

\section{Introduction; main results}

\subsection{Background and overview}
The dimer model, introduced by Kasteleyn \cite{Kasteleyn1961TheLattice} and independently Fisher and Temperley \cite{FisherTemperley} in the 1960s, is one of the most classical models of statistical mechanics. Given a finite, planar bipartite graph $G = (V,E)$, a dimer configuration $\mathbf{m} \subset E$ is a perfect matching of $G$, i.e., a subset of edges such that each vertex $v \in V$ is incident to exactly one edge $e \in \mathbf{m}$. We let $\cD(G)$ denote the set of dimer configurations on $G$. If edges of the graph $G$ are furthermore weighted (let $c_e>0$ denote the weight or conductance of the undirected edge $e$), then we form a probability measure on the set of dimer configurations by setting
$$
\mathbb{P} ( \mathbf{m}) = \frac1Z \prod_{e\in \mathbf{m}} c_e,
$$
where the constant $Z = \sum_{\mathbf{m} \in \cD(G)} \prod_{e \in \mathbf{m}} c_e$, the \textbf{partition function}, is chosen so that $\mathbb{P}$ is a probability measure. We will always assume that $Z>0$ (i.e., the set of dimer configurations $\cD(G)$ is nonempty).

When $G$ is a portion of the square lattice, or more generally when it is an isoradial graph (or more precisely, an isoradial superposition, see definitions below) it is well known that there is a natural choice of weights where the dimer model is \textbf{critical}: e.g., dimer-dimer correlations decay polynomially with the distance, and the associated dimer height function becomes conformally invariant in the scaling limit (and in an appropriate gauge), in fact converges to a multiple of the Gaussian free field with appropriate boundary conditions in the scaling limit where the mesh size $ \eps \to 0$. This occurs e.g. on the square lattice when the weights are uniform ($c_e \equiv 1$) and the graph has \textbf{Temperleyan boundary conditions}, which will be defined below more precisely; this was the content of Kenyon's groundbreaking results \cite{Kenyon_confinv, KenyonGFF}. On isoradial superpositions when the weights are inherited from Kenyon's critical weights \cite{Ken02} this was obtained in the full plane in \cite{dT07} (generalised to domains with Temperleyan boundary conditions in \cite{Li17} under additional assumptions on the boundary, and substantially more generally in \cite{BLR_DimersGeometry}).

In this paper we are concerned with the \textbf{near-critical dimer model}, in which the edge weights deviate from the above critical weights by an amount which converges to 0 as $ \eps \to 0$ at a suitably chosen speed. The goal of this paper will be to prove that in this case it is still possible to obtain a nontrivial scaling limit. Characteristic features of this scaling limit are that on scales smaller than macroscopic it is indistinguishable from the above critical scaling limit: in particular, at such scales dimer-dimer correlations will still decay polynomially fast with the distance, and the height function is locally a Gaussian free field. However, at macroscopic scales and beyond, dimer-dimer correlations decay exponentially fast, and the height function is no longer Gaussian in the scaling limit. In other words it is given by a \textbf{massive field theory}. The main goal of this paper is to identify precisely the limiting height function, which we show is given by the so-called \textbf{sine-Gordon model with electromagnetic field}, a model from Quantum Field Theory (QFT) which is integrable but non-conformal. Questions of this nature go back in the physics literature at least to the work of Lukanov \cite{Lukyanov} who considered the six-vertex model close to the critical point and suggested that in a suitable limit the scaling is identical to that of the sine-Gordon model.


Let us briefly introduce the model we will be working with in the case of the square lattice (the model will be described in a more general framework of isoradial superpositions below). Given a bounded, simply connected domain $\Omega \subset \RR^2 \simeq \CC$ and a graph $G_ \eps \subset  \eps \ZZ^2$ which is a Temperleyan approximation of $\Omega$, the basic parameter of the near-critical deformation we consider will be a \textbf{vector field} $\alpha : \bar \Omega \to \RR^2$ which is assumed to be smooth up to and including the boundary. We will assume that this vector derives from a potential:
\begin{equation}\label{eq:potential}
    \alpha = \nabla V, \text{ where } V : \bar \Omega \to \RR
\end{equation}
Given such a potential $V$, we define near-critical weights in $G_ \eps$ in the following manner. Consider the standard black-white chessboard decomposition of $ \eps \ZZ^2$. Around every other black vertex $b$ (See Figure \ref{Fig:rhombusintro} below), if $w$ is a neighbouring white vertex (thus $|w- b| =  \eps$), set 
\begin{align} \label{eq:weightsZ2}
c_{e} &= e^{2(V(w) - V(b))} 
 = 1 +   2 \langle \nabla V, w-b\rangle + o( \eps)
\end{align} where $e$ is the undirected edge $e= (bw)$. 

In the case where $\alpha(z) \equiv \alpha_0$ is constant throughout $\Omega$ (and when $\Omega$ is actually equal to the whole plane) such a model was introduced by Chhita in \cite{Chhita2012}, who showed among other things that the limiting height function cannot be Gaussian since its moments do not satisfy the Wick relation. The same model was further considered in \cite{Mason} who observed that the two-point correlation of the derivative of the height function match those of the sine-Gordon model in the full plane. Roughly simultaneously and independently, the above model was introduced (in a domain $\Omega$ and with nonconstant vector field $\alpha$) in \cite{BHS}. There, this model was studied from the perspective of Temperley's bijection. In particular it was shown there that under the assumption that $V$ is \textbf{log-convex}, paths in the Temperleyan tree have a nontrivial scaling limit, which in the case  where $\alpha \equiv \alpha_0$ is just constant coincides with Makarov and Smirnov's \textbf{Massive SLE$_2$} (\cite{MakarovSmirnov, ChelkakWan} --  in fact this massive SLE had been introduced independently and slightly before by Bauer, Bernard and Kytöla in \cite{BBK}; see also \cite{BBC}). A precise conjecture was also formulated there for the scaling limit of the height function: namely, the authors conjectured that in the scaling limit, the height function is given by a multiple (namely, $\sqrt{2}$) times a field whose law is 
\begin{equation}\label{eq:conjBHS1}
    \mathbb{P}_{(\Omega, \alpha)} ( \mathrm{d}\phi) =  \exp \left( z_0 \int_\Omega  \left\langle e^{\ii  \frac1\chi \phi(x)}, \alpha(x) \right\rangle \mathrm{d}x \right) \mathbb{P}^{\text{GFF}} ( \mathrm{d}\phi),
\end{equation}
where $\ii = \sqrt{-1}$ (here and everywhere in the paper), $z_0\in \RR$ is an unspecified constant, $\chi = 1/\sqrt{2}$ is the \textbf{Imaginary Geometry} constant associated with $\kappa =2$, and we identify complex numbers with vectors in multiple ways (thus $\langle z, w \rangle = \Re (z\bar w)$ is the standard Euclidean inner product), and $\PP^{\mathrm{GFF}}$ is the law of a Gaussian free field (see \cite{BP} for a general introduction). We note for future reference that in \eqref{eq:conjBHS1} (as well as in \eqref{eq:conjBHS2} and \eqref{eq:conj3} below) the height function is implicitly defined with respect to a reference flow is the winding flow of, say, \cite{BLR_torus, BLRriemann2} and whose total mass is $2\pi$; in other words this height function is $2\pi$ larger than the one we will be working with in this paper. In fact, this conjecture forgot to take into account a factor of $\ii$ in $\alpha$ (or, equivalently, an additive factor of $\pi/2$ in the limiting height function), as we will explain.

In \eqref{eq:conjBHS1} it is also important to note that the authors were dealing with an uncentered height function (i.e., the height function was defined with respect to the so-called winding reference flow, see \cite{BLR_DimersGeometry, BLR_torus, BLRriemann2}) and  the height of the Gaussian free fied is normalised according to the conventions of imaginary geometry, i.e., $\EE^{\text{GFF}}[\phi(x) \phi(y) ] = - \log |x- y| + O(1)$. However it is more common in the literature concerning sine-Gordon to normalise the field differently: namely we use the law $\PP^{\text{GFF}\#}$ to denote a GFF in which $\EE^{\text{GFF}\#} [\phi(x) \phi(y) ] = - (2\pi)^{-1} \log |x-y | + O(1)$. Then the law \eqref{eq:conjBHS1} in the conjecture from \cite{BHS} can be rewritten (up to multiplication by a factor $\sqrt{2} \times \sqrt{2\pi} = 2\sqrt{\pi}$, and after correcting the conjecture to include the factor of $\ii$) as the law
\begin{equation}\label{eq:conjBHS2}
    \mathbb{P}_{\text{SG}(
    \alpha)} ( \mathrm{d}\phi) =  \exp \left( z_0 \int_\Omega  \left\langle e^{\ii  \sqrt{4\pi} \phi (x) }, {\ii \alpha(x)} \right\rangle \mathrm{d}x \right) \mathbb{P}^{\text{GFF}\#} ( \mathrm{d}\phi).
\end{equation}
We note that, formally, in the case when $\Omega = \RR^2$ is the whole plane and $\alpha$ is constant (which can then be taken to be purely imaginary by rotation invariance, so $\alpha \equiv -\ii \alpha_0$, say), the law \eqref{eq:conjBHS2} coincides with the law
\begin{equation}
\label{eq:conj3}
\mathbb{P}_{\text{SG}(\alpha)} ( \mathrm{d}\phi) =  \exp \left( z_0 \int_\Omega  \cos (\sqrt{4\pi} \phi (x) ) \mathrm{d}x \right) \mathbb{P}^{\text{GFF}\#} ( \mathrm{d}\phi),
\end{equation}
where the value of $z_0$ has changed to incorporate that of $\alpha_0$.
Once again, we point out that, although it is not explicitly written in \cite{BHS}, in the above conjecture it is important  to note that the GFF of the law $\PP^{\text{GFF}\#}$ does not have zero boundary conditions: instead, it has \emph{winding} boundary conditions (see (2.7) in \cite{BLR_DimersGeometry}).

This is, at least formally, the so-called \textbf{sine-Gordon field at the free fermion point}. In the general case where $\alpha$ is non-constant, the law \eqref{eq:conjBHS2} can be viewed intuitively as that of a perturbed sine-Gordon field, in which the electromagnetic field $\alpha$ (deriving from the potential $V$) tilts the field so that $e^{\ii \sqrt{4\pi} \phi}$ points in the same direction as $\alpha$. 

\medskip {The main goal of this paper will be to prove this conjecture under certain assumptions on the potential $V$}. To state it, let us first point out that actually defining the sine-Gordon model from \eqref{eq:conj3} is in itself highly nontrivial. In the full plane, this was achieved by the breakthrough work of Bauerschmidt and Webb \cite{BauerschmidtWebb}. In a bounded domain (in fact, in the unit disc) subject to certain boundary conditions, this is the subject of a work by Park, Virtanen and Webb \cite{PVW}. In both cases, a prominent feature (both a property of the field and a tool used in its construction) is the so-called \textbf{Coleman transform}, which expresses the fermionic nature of the field (and will thus ultimately yield a massive extension of the Boson-Fermion correspondence for the dimer model). This leads to concrete expressions for the moments of the sine-Gordon field, of the type \begin{equation}\label{eq:conjmoments}
    \EE_{\text{SG}(\alpha)}( (\bar \phi, f)^n) = m_n(f),
\end{equation}
where $\bar \phi = \phi - \EE_{\text{SG}(\alpha)}( \phi)$, $f$ is a test function, and $m_n$ is the sum of $n$-fold integrals of certain kernels $\chi^{(\mathbf{s})}(z_1, \ldots, z_n)$ (which can be interpreted as the weak holomorphic and antiholomorphic derivative correlations); here $z_1, \ldots, z_n$ are pairwise disjoint points of the domain $\Omega$, while $\mathbf{s} = (s_1, \ldots, s_n)\in \{0,1\}^n$ indicates whether we consider the holomorphic or antiholomorphic derivative. These kernels can be expressed in terms of
Grassmann variables, or,  equivalently, in terms of determinants involving a massive Dirac operator. (Concrete expressions for the correlation kernels $\chi^{(\mathbf{s})}$, and thus for $m_k(f)$, will be given below). 

We state our main theorem below, in a slightly informal manner for now.
Let $\EE_{(\Omega, \alpha)}$ denote the expectation for the dimer model in $\Omega$ with weights from~\eqref{eq:weightsZ2}. 

\begin{thm}\label{T:intro}
    Suppose $\Omega$ is a bounded simply connected domain. Let $h_\eps$ denote the centered height function associated to the dimer model on the Temperleyan graph $G_\eps \subset \eps \ZZ^2 \cap \Omega$, associated with weights as in \eqref{eq:weightsZ2}. Suppose that the potential $V$ is log-convex, and satisfies additional assumptions. Then for any smooth test function $f$ with compact support in $\Omega$, for any $k\ge 1$, $\EE_{(\Omega, \alpha)}((h_\eps, f)^n) \to m_n(f)$  as $\eps \to 0$, where $m_{n}(f)$ are the expressions in \eqref{eq:conjmoments}.
    Furthermore the moments $m_n(f)$ identify uniquely the corresponding law. 
\end{thm}

In the above result the height function is defined with respect to the reference flow $\omega$ such that $\omega(wb)$ is the probability that the edge $wb$ belongs to a random dimer configuration. 

See Theorem \ref{T:SG} for a more precise statement. 
In fact, we obtain this result in the more general setting of \emph{double} isoradial graphs (also known as isoradial superpositions) with appropriate edge weights, which will be detailed below. 

\medskip We emphasise that the work of Park, Virtanen and Webb \cite{PVW} gives a rigorous construction of the sine-Gordon model from \eqref{eq:conj3} and proves rigorously \eqref{eq:conjmoments} in the case of the unit disc. Thus, combining their work with our Theorem \ref{T:intro}, this identifies the (centred) sine-Gordon model at the free fermion point $\beta = 4\pi$ as the limit of the centered height function for near-critical dimers, thereby resolving Conjecture \eqref{eq:conj3}, under the assumptions in the statement on the potential $V$. On the other hand, our work can be read independently from \cite{PVW} and proves the convergence of the height function to a field whose moments match with the \emph{Coleman form} of the sine-Gordon model: concretely, Theorem \ref{thm:fermionic} up to \eqref{eq:fermionic} does not use any input from \cite{PVW}. It is only at the very end that 
see Section \ref{SS:identification:overview} for more details.  


\medskip Physically, the result in Theorem \ref{T:intro} corresponds to a massive extension of the \textbf{boson-fermion correspondence} embodied by Kenyon's convergence result \cite{Kenyon_confinv, KenyonGFF} (note that the GFF has Gaussian correlations which are thus bosonic by virtue of Wick's formula, whereas the dimer correlations are determinantal/Pfaffian, hence fermionic).

\medskip \paragraph{Heuristics and connection to imaginary geometry.} Although few details are given in \cite{BHS} about this conjecture, it is in fact very natural from an imaginary geometry perspective. We now briefly explain this heuristics. We consider the Radon-Nikodym derivative of the near-critical dimer model defined by \eqref{eq:weightsZ2} with respect to the critical case where $c_e \equiv 1$, and argue that it is approximately given by \eqref{eq:conjBHS1}. Let $\Gamma$ denote the class of black vertices around which the edge weights are not equal to one (as defined by the Temperleyan, biperiodic structure of the edge weights). Note that if $\mathbf{m}$ is a given dimer configuration, then
$$
\PP_{(\Omega, \alpha)}(\mathbf{m}) \propto \left(\prod_{e \in \mathbf{m}} c_e\right)\PP_{(\Omega, 0)}(\mathbf{m}) \approx \exp \left ( 2 \sum_{e = (b,w)\in \mathbf{m}, b \in \Gamma} \langle \nabla V, w - b \rangle \right) \PP_{(\Omega, 0)}(\mathbf{m})
$$
from \eqref{eq:weightsZ2}. (We used here the Temperleyan nature of the weights, in that edges $e = (b,w)$ with $b \not \in \Gamma$ have weight equal to one by definition). Let $\cT(\mathbf{m})$ denote the wired spanning tree on the graph induced by $\Gamma$ associated to $\mathbf{m}$ via Temperley's bijection: i.e., for each $b\in \Gamma$, by definition, $\cT(\mathbf{m})$ includes the edge $(bb')$, where $b'\in \Gamma$ is uniquely defined by the requirement that $(bw) \in \mathbf{m}$, with $w$ the white vertex midway through $b$ and $b'$. We write $b' = \sigma(b)$ for this unique vertex, which is the successor of $b$ in the arborescence $\cT$. We can then rewrite
$$
\PP_{(\Omega, \alpha)}( \mathbf{m}) \approx \exp \left(  \sum_{b \in \Gamma}  \langle \nabla V , \overrightarrow{(b, \sigma(b))}\rangle \right) \PP_{(\Omega, 0)}(\mathbf{m})
$$
where $\overrightarrow{(b, \sigma(b))}$ is the vector joining $b$ to $\sigma(b)$. Dividing by the norm of this vector (necessarily equal to $2\eps$) we get a unit vector which we write as $e^{ \mathbf{i} \theta_b}$. Thus
\begin{align}
\PP_{(\Omega, \alpha)}( \mathbf{m}) &\approx \exp \left( 2\eps \sum_{b \in \Gamma}  \langle \nabla V , e^{\mathbf{i} \theta_b} 
\rangle 
\right) \PP_{(\Omega, 0)}(\mathbf{m})
= \exp \left( \frac{2}{\eps}\eps^2 \sum_{b \in \Gamma}  \langle \nabla V , e^{\mathbf{i} \theta_b} 
\rangle 
\right) \PP_{(\Omega, 0)}(\mathbf{m})
\label{eq:approxweights}
\end{align}
It remains to observe that $e^{\mathbf{i} \theta_b}$, the direction of the unique outgoing arrow emanating from $b$ in the arborescence $\cT(\mathbf{m})$, is nothing else but $e^{ \mathbf{i} h_\eps (f)}$, where $h_\eps$ is the height function with respect to the winding reference flow (in units of $2\pi$), and $f$ is a face adjacent to $b$, chosen by convention so that $b$ is the top-left corner of $f$. This follows from the fact that in Temperley's bijection, the height function measures the winding of the branches in the spanning tree $\cT(\mathbf{m})$ (see in particular Proposition 4.4 in \cite{BLR_torus} for this formulation of Temperley's bijection). Approximating the Riemann sum in \eqref{eq:approxweights} by an integral, we deduce 
\begin{equation}\label{eq:approxRN}
\PP_{(\Omega, \alpha)}( \mathbf{m}) \approx \exp\left( \frac{2}{\eps} \int_\Omega  \langle \nabla V(x), e^{\ii h_\eps (x)} \rangle \mathrm{d} x \right) \PP_{(\Omega, 0)}(\mathbf{m}).    
\end{equation}
Recall that, under the uniform law $\mathbb{P}_{(\Omega, 0)}$, by Kenyon's result \cite{Kenyon_confinv} (see also \cite[Theorem 1.2]{BLR_DimersGeometry}, as here we are working with the non-centered field), $h_\eps \to \tfrac{1}{\chi}{\phi}$, where $\phi$ has the law of a Gaussian free field with winding boundary conditions, $+\pi/2$ (this extra factor of $\pi/2$ accounts for the fact that the height function measures the intrinsic winding of the path in the tree connecting the point to the reference boundary vertex; and this path has an extra turn of $\pi/2$ when it hits the boundary; see, e.g., Theorem 3.1 or Theorem 5.1 in \cite{BLR_DimersGeometry}, bearing in mind that on the square lattice, the local mean winding of the uniform spanning tree, $m^{\#\delta}$, tends to 0).

Furthermore, we have (by definition of the imaginary chaos, see e.g. \cite{JunnilaSaksmanWebb}) 
\begin{equation}\label{eq:IMC}
\eps^{-\tfrac{1}{2\chi^2}} \int_\Omega  \langle \nabla V(x), e^{\ii \tfrac{1}{\chi} \phi_\eps (x) + \ii \pi/2} \rangle \mathrm{d} x \to - \int_\Omega \langle \ii \nabla V(x), e^{\ii \tfrac{1}{\chi}\phi (x)} \rangle \mathrm{d} x.
\end{equation}
Since $\chi = 1/\sqrt{2}$ (equivalently, since we are at the free fermion point $\beta = 4\pi$), the power of $\eps$ in the factor in front of the integral above is simply $\eps^{-1}$, which matches exactly the power of $\eps$ in \eqref{eq:approxRN}. (Note that in reality, \eqref{eq:IMC} only holds strictly below the free fermion point. At the free fermion point the integral itself does not make literal sense; this is related to the fact that the sine-Gordon model is not expected to be absolutely continuous with respect to the Gaussian free field at the free fermion point.)

Making the substitution, we finally find:
\begin{equation}
\PP_{(\Omega, \alpha)}( \mathbf{m}) \approx \exp\left( -2 \int_\Omega  \left\langle \ii \nabla V(x), e^{\ii \tfrac{1}{\chi} \phi (x)} \right\rangle \mathrm{d} x \right) \mathbb{P}^{\text{GFF}} ( \mathrm{d}\phi)
    \label{eq:approxweights_2}
\end{equation}
This is indeed a discrete form of \eqref{eq:conjBHS1} with $z_0 = -2$. 

\medskip \paragraph{Discrete and continuous massive holomorphicity.} While the reasoning leading to \eqref{eq:approxweights_2} is very suggestive and appealing, it appears difficult to turn it into a proof at this stage using currently available technology. On the other hand, we will approach this conjecture via a totally different perspective. Our approach {generalizes the approach of~\cite{KenyonGFF} to the near-critical setting: we} develop a theory of \textbf{massive holomorphicity} and use it to compute exactly the limit of the so-called inverse Kasteleyn matrix, which governs dimer-dimer correlations. The fact that the inverse Kasteleyn matrix is not discrete holomorphic but is instead \emph{massive} holomorphic in the scaling limit is one of the main challenges in the analysis.  
In turn, after explicit calculations, we will show that the moments of the height function integrated against a given test function are given by  formulas which match those of the above sine-Gordon model, at least when the latter is expressed in terms of fermionic integrals (something which is possible due to the already mentioned Coleman transform, a feature of the fact that we consider the sine-Gordon model at its free fermion point).


Discrete holomorphicity is one of the most useful tools to study critical models. 
In~\cite{MakarovSmirnov}, Makarov and Smirnov introduced discrete massive martingale observables and advocated the use of a \textbf{discrete massive holomorphicity theory} to study near-critical models.
In the context of the Ising model, this was first partially implemented in~\cite{DC_Garban_Pete}, and a theory of convergence of discrete massive holomorphic observables towards their continuous counterpart was developed in~\cite{Park,Par22} and used to study the near-critical FK-Ising model. Because our near-critical perturbation depends on a vector field $\alpha$ rather than single real number (``mass'') we will need to formulate a new notion of massive holomorphicity (which reduces to previously studied notions when the vector field is constant and, viewed as a complex-valued function, is purely real-valued). 

In the context of the dimer model, a first form of discrete massive holomorphicity theory was developed in~\cite{dT21} with the introduction of a massive discrete Dirac operator on double isoradial graphs with $Z$-invariant weights. In~\cite{Rey} this theory is extended to double isoradial graphs with more general weights, and is applied in the near-critical setting to generalize to double isoradial graphs the results of~\cite{BHS} on the convergence of the near-critical dimer height function.


\subsection{Isoradial setup; discrete massive harmonic functions}
\label{SS:isoradial:setup}

Let $\infPrimal = \infPrimal_{\eps}$ denote an infinite isoradial graph with mesh $\eps$. This means that it can be embedded in the plane in such a way that every face is inscribed in a circle of radius $\eps$, and such that the circumcenter of each face lies in the closure of the face. Each pair of dual edges form the diagonals of a rhombus.

 We will everywhere assume that this graph has the \textbf{bounded angle property}: there exists $\eta >0$ such that all rhombi angles are in $[\eta, \pi/2-\eta]$. 
One classical consequence of the bounded angle assumption (see for example Assumption 1 of \cite{SchStuWar24}) is that for any $x, y \in \infPrimal$ we can choose a path $\gamma_{\#}: x \to y$ in $\infPrimal$ of length $\leq \frac{C|y - x|}{\eps}$ where the constant $C$ depends only on $\eta$.

 Let $\Primal = \Primal_\eps$ be a finite simply connected subgraph of $\infPrimal$, denote by $\closurePrimal$ the union of $\Primal$ and all adjacent vertices, by $\Primal^*$ the restricted planar dual of $\Primal$, that is the dual graph from which the dual vertex corresponding to the outer face and all incident dual edges are removed. The set $\Black = \Primal \cup \Primal^*$ corresponds to the black vertices of a Temperleyan superposition graph $G$ (also known as medial graph). That is, we remove one vertex on the boundary of $\Primal^*$, called $b^*$. 

 We denote by $\White = \Black^*$ the set of white vertices, corresponding to centres of rhombi in the isoradial graph $\Primal$. Note that in general (e.g. if $\Primal$ is not bipartite) there is only one class of white vertices. When $\Primal = \infPrimal$, we will write $\White = \infWhite$ and $\Black = \infBlack$.

 \begin{assumption}\label{assumption:domain}
    We always assume that in the limit $\eps \to 0$, the graph $\Primal$ with Temperleyan corner $b^*$ approximates a fixed open set $\Omega$ with boundary point $\beta \in \partial \Omega$, by which we mean 
    \begin{enumerate}[label=(\roman*)]
        \item $\Primal \subset \Omega$
        \item for every compact subset $\cC \subset \Omega$, $\Primal \cap \cC = \infPrimal \cap \cC$ 
        \item for all $\omega \in \partial \Omega$, there exists $x_\eps \in \Primal$ such that $x_\eps \to \omega$
        \item $b^* = b^*(\eps) \to \beta^* \in \partial \Omega$.
    \end{enumerate}
 \end{assumption}
 
In fact, assumption (iii) easily follows from assumption (ii).


\paragraph{Edge weights.} To define the edge weights on the superposition graph $G$ we proceed as follows.

We fix a vector field $\alpha : \RR^2 \to \RR^2$, smooth on $\bar \Omega$. 
As mentioned previously we will assume that $\alpha = \nabla V$ derives from a potential $V: \RR^2 \to \RR$ which is assumed $C^2$ in $\bar \Omega$.


Let $x^-, x^+$ be two adjacent vertices of $\infPrimal$, and let $y^-$ and $y^+$ be the adjacent faces so that $x^- y^- x^+ y^+$ forms a rhombus in the superposition whose center is called $w$. 
Let $\theta_w = \theta_{x^-x^+} = \theta$ denote the half angle of $\widehat{y^- x^- y^+}$. 
Let 
$$
c^0_{x^-, x^+} = \tan (\theta),
$$
be the \textbf{critical weights} on $\infPrimal$. We now introduce the \textbf{near-critical weights} which we will use throughout the paper. We set 
\begin{equation}\label{Eq:offcriticaledgeweights}
c^\alpha_{x^-, x^+} = \exp \Big(  V(x^+) - V(x^-) \Big) \tan(\theta). 
\end{equation}

\begin{remark}
    This differs slightly from the convention in \eqref{eq:weightsZ2}, where (in these notations) we considered $c^\alpha_{x^-, x^+} = \exp \Big( \Big (V(w) - V(x^+)\Big)\tan (\theta)$ instead of \eqref{Eq:offcriticaledgeweights}. However this difference is irrelevant: it can be checked from the results in \cite{BHS} that the height functions associated to both weight sequences have the same limit. This also explains why we included a factor of 2 in the exponent of \eqref{eq:weightsZ2}.
\end{remark}
Associated with these conductances on $\closurePrimal$ we define edge weights on the superposition graph $G$ of $\Primal$ and $\Primal^{\star}$, where $ c_{x^-, w} = c^\alpha_{x^-, x^+} $; $c_{x^+, w} = c^{\alpha}_{x^+, x^-}$, whereas $c_{w, y^-} = c_{w, y^+} = 1$.
See Figure \ref{Fig:rhombusintro}. 
On the boundary, that is when $x^- \sim x^+, x^- \in \Primal, x^+ \in \closurePrimal \setminus \Primal$, we simply remove the edge $wx^+$. 
Similarly, if $y \sim b^*, y \in \Primal^{\star}$, we simply remove the edge $wb^*$. 


\begin{figure}
    \centering
    \begin{overpic}[scale = .95]{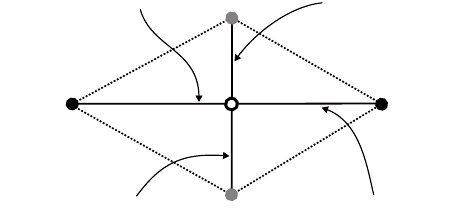}
        \put(10,25){$x^-$}
        \put(83,25){$x^+$}
        \put(50,0){\textcolor{gray}{$y^-$}}
        \put(50,45){\textcolor{gray}{$y^+$}}
        \put(15,46){$\frac{e^{V(x^+)}}{e^{V(x^-)}}c^0_{x-,x^+}$}
        \put(70,-2){$\frac{e^{V(x^-)}}{e^{V(x^+)}}c^0_{x^+,x^-}$}
        \put(25,4){$1$}
        \put(70,44){$1$}
    \end{overpic}
    \caption{Edge weights. (The weights indicated on this picture correspond to the orientation of a given edge from the black vertex towards the white vertex in the middle of the rhombus).
    }
    \label{Fig:rhombusintro}
\end{figure}
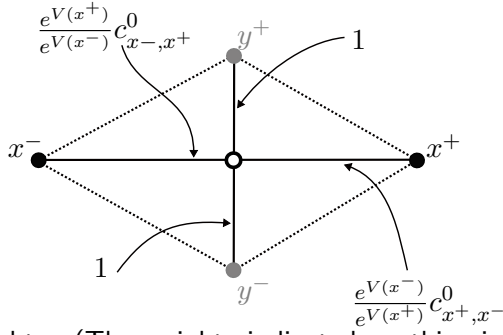

We consider the dimer model on $G$ with the above off-critical edge weights. Ultimately we aim to describe the scaling limit of the height function as the mesh size $\eps \to 0$.

\paragraph{Interpretation as drift.} We provide a simple lemma, which shows that the vector field $\alpha$ may be viewed as the limiting drift in the diffusion which describes the scaling limit of the random walk on $\bar \Gamma$.

\begin{lemma}
Let $\Omega \subset \RR^2$ be a domain and let $\alpha : \RR^2 \to \RR^2$ be a vector field which is smooth on $
\bar \Omega$. Then the random walk on $\closurePrimal$ with conductances $c^\alpha$, killed when exiting $\Primal$ converges to the solution of the above SDE, killed when exiting $\Omega$.
\end{lemma}

We also define a \textbf{mass function}
\begin{equation}\label{eq:mass_intro}
    M(x) := \Delta V(x)  + \| \nabla V(x) \|^2.
\end{equation}
This function will play an important role in the rest of this paper. Under the \textbf{drift-mass equivalence} observed in \cite{BHS} and generalised in \cite{Rey} to the isoradial setup, the function $M$ will play the role of a mass. This is the main reason for enforcing the following assumption (the log-convexity assumption of $V$ ensures precisely that $M(x) \ge 0$ for all $x \in \bar \Omega$):

\begin{assumption}[Log-convexity]\label{assumption:pos}
    $V$ is log-convex throughout $\Omega$: that is,
    \begin{equation}\label{eq:pos}
        M(x) = \Delta V(x)  + \| \nabla V(x) \|^2 > 0; \quad x \in \bar \Omega.
    \end{equation}
    If $V$ is positive then this assumption is indeed equivalent to stating that $\log V$ is convex on $\Omega$.
\end{assumption}

\subsection{Notion of massive holomorphicity in the continuum}

Our first goal will be to state a theorem on the scaling limit for the inverse Kasteleyn matrix $K^{-1}$ (see Section \ref{SS:Gauge} for the choice of gauge function and corresponding Kasteleyn matrix; see Theorem \ref{thm:isoradial} below for the statement on $K^{-1}$). This necessitates the introduction of a notion of massive holomorphic function and some basic notions related to this. 


\begin{defn} \label{D:massive_holo}
    Let $U \subset \CC$ be an open set. 
    
    A function $h: U \to \RR$ is \textbf{massive harmonic} (or $M$-massive harmonic) if for all $x \in \Omega$,
    \begin{equation}\label{eq:massive_holo_Laplacian}
        [\Delta h](x) = \mass(x)h(x)
    \end{equation}
    with the function $M$ from~\eqref{eq:mass_intro}.
    
    A function $f: U \to \CC$  is \textbf{$\alpha$-massive holomorphic}, resp. \textbf{$-\bar\alpha$-massive holomorphic}, if it satisfies the identity:
\begin{equation}\label{eq:massive_holo}
    \dzbar f = \frac12\alpha \overline{f} \quad , \text{resp.} \quad  \dzbar f = -\frac12\bar\alpha \overline{f} 
\end{equation}
where we identify here $\alpha = \nabla \ph$ with the complex number $\partial_x \ph + \ii \partial_y \ph$ (as usual) and the above product on the right hand side is the multiplication of two complex numbers. 
\end{defn}

\begin{rmk}
    This differs from previously considered notions of massive holomorphicity (e.g. \cite{MakarovSmirnov, Park}) where the corresponding equations were
    $$
        \Delta h = M^2 h \quad ; \quad \partial \bar f = iM f, 
    $$
    for a \textbf{constant} $M \in \RR$ (with $M^2$ often called the \textbf{mass}). 
    
    The notion introduced in Definition \ref{D:massive_holo} is therefore a novel and more general formulation of the notion of massive holomorphicity; note in particular that the ``mass'' term $M(x)$ in  \eqref{eq:massive_holo_Laplacian} is not simply the square (or even the squared norm) of $\alpha$ in \eqref{eq:massive_holo}. 
    Ultimately, the novelty comes from the fact that the near-criticality in the dimer model is described by a vector field $\alpha$ in the domain, instead of a single real number. 
    
    It is worth noting that, in the full plane and with the vector field $\alpha$ being constant, the two notions \emph{do} coincide, since (after rotation) we can assume without loss of generality that $\alpha$ points in the real positive direction. This will be analogous to the fact that the generalized sine-Gordon field coincides with the (constant-mass) sine-Gordon field in the full plane. 
\end{rmk}

Written explicitly in real coordinates \eqref{eq:massive_holo} becomes the \textbf{($\alpha$-), resp. ($-\bar\alpha$-), massive Cauchy-Riemann equation}
\begin{equation}\label{eq:a:CR}
\begin{cases}
\partial_x u - \alpha_x  u & = \partial_y v + \alpha_y v \\
\partial_y u - \alpha_y u & = - ( \partial_x v + \alpha_x v)
\end{cases}
\quad 
\text{, resp.}
\quad 
\begin{cases}
\partial_x u + \alpha_x  u & = \partial_y v + \alpha_y v \\
\partial_y u - \alpha_y u & = - ( \partial_x v - \alpha_x v)
\end{cases}
\end{equation}
Recalling that $\alpha = \nabla V$, this can be rewritten as
$$
\begin{cases}
     \partial_x ( ue^{-V}) e^V & = e^{-V} \partial_y ( v e^V) \\
      \partial_y (ue^{-V}) e^V& = - e^{-V} \partial_x ( ve^V) .
\end{cases}
\quad 
\text{, resp.}
\quad 
\begin{cases}
     \partial_x ( ue^{V}) & = \partial_y ( v e^V) \\
      \partial_y (ue^{-V}) & = -\partial_x ( ve^{-V}) .
\end{cases}
$$
The $\alpha$-massive Cauchy-Riemann equation can be rewritten more geometrically 
\begin{equation}\label{eq:CR:geometric}
e^V\nabla ( u e^{- V}) = e^{-V} (-\ii)\nabla ( v e^{V}),
\end{equation}
where the gradient is identified with a complex number. This relation is the massive analogue of the fact that in the critical case, the gradients of $u$ and $v$ can be obtained from one another by applying a 90 degrees rotation.

In the critical case, the real part of a holomorphic function is harmonic. This extends to the massive case, where the mass $M$ is precisely the function introduced in \eqref{eq:mass_intro}. Furthermore, given a real-valued, massive harmonic function $u$, there is at most one function $v$ such that $f = u+\ii v$, which by analogy with the critical case we call the \textbf{massive harmonic conjugate} of $u$. (More precisely, two massive holomorphic functions $f$ and $g$ with the same real part coincide). We state these two properties in the lemma below.

\begin{lem} Suppose $\Omega \subset \CC$ is open and connected (not necessarily simply connected).
If $f = u+\ii v$ is $\alpha$-massive holomorphic in $\Omega$, then $u$ is a $\mass$-massive harmonic function in $\Omega$. Furthermore, the $\alpha$-harmonic conjugate of $u$ is at most unique up to an additive constant. That is, suppose $f$ and $g$ are two massive holomorphic functions with $\Re(f) = \Re(g)$. Then $f-g$ is a constant. 
\end{lem}

\begin{proof} Note that
    $$
    \begin{aligned}
    \Delta u + \ii \Delta v = \Delta f 
    &= 4 \dz \dzbar f\\
    &= 2 \dz (\alpha \bar f)\\
    &= 2\left[(\dz \alpha) \bar f + \alpha (\dz \bar f)\right]\\
    &= 2\left[2(\dz\dzbar V) \bar f+ \frac{1}{2}\alpha \bar \alpha f \right]\\
    &= \Delta V \bar f+ |\nabla V|^2 f\\
    &= \Delta V (u-\ii v) + |\nabla V|^2 (u+\ii v)\\
    &= (\Delta V + |\nabla V|^2) u + \ii (-\Delta V + |\nabla V|^2)) v\\
    &= \mass u + \ii (-\Delta V + |\nabla V|^2)) v
    \end{aligned}
$$
Identifying the real and imaginary parts, we obtain that $u$ is always a $\mass$-massive harmonic function, while $v$ is a $\mass$-massive harmonic function if and only if $\Delta V = 0$ (which happens in particular when the drift $\alpha$ is constant). 

The fact that $v$ is uniquely defined is a direct consequence of \eqref{eq:CR:geometric} since the gradient of $u$ determines the gradient of $v$.
\end{proof}

\begin{remark}
Let $u:\Omega \to \RR$ be a massive harmonic function. Suppose that
    \begin{equation}\label{eq:lackofconjugate}
        w\mapsto e^{2V(w)} \nabla ( u(w) e^{-V(w)}) \text{ derives from a potential.}
    \end{equation} 
    Then, integrating this quantity from a designated marked point in $\Omega$ to an arbitrary $z\in \Omega$ produces a function $z\in \Omega \mapsto v(z)\in \RR$ (indeed, by assumption this does not depend on the choice of path). It is then straightforward to verify that $u+\ii v$ is massive holomorphic. 

    On the other hand, it is \emph{a priori} not clear at all why \eqref{eq:lackofconjugate} should hold. In the critical/massless case, this happens as soon as $\Omega$ is simply connected. Could this also hold in the massive case? We do not know the answer to this question. As a result we do not know any criterion (other than \eqref{eq:lackofconjugate} itself) for the existence of a harmonic conjugate. Fortunately, this is not needed in the rest of the paper: in Theorem \ref{thm:isoradial}, we appeal to a specific harmonic conjugate and need to prove as part of the argument that this function is well defined.
    \end{remark}
    
We will also need the following definition.
\begin{defn}\label{def:massive:harmonic:conjugate}
    Let $f = u+\ii v, f^* = u^*+\ii v^*: U \to \CC$, be two complex-valued function whose real and imaginary parts $u,u^*$ and $v,v^*$ are massive harmonic functions. We \textbf{do not} assume that $f$ and $f^*$ are massive holomorphic. We say that $f^*$ is the \textbf{$\alpha$-conjugate} of $f$ if $u^*$, resp. $v^*$, is the $\alpha$-conjugate of $u$, resp. $v$.
\end{defn}
This definition can be written conveniently in matrix form.
For a complex number $z = a+\ii b$ we denote by $z^\intercal$ the line vector $\begin{pmatrix} a & b \end{pmatrix}$.
We denote by $\Rot_{-\pi/2} $ the rotation matrix of angle $-\pi/2$.
Then, $f$ and $f^*$ are $\alpha$-conjugates if and only if the following equality between $2$ by $2$ matrices holds:
    \begin{equation}\label{eq:harmonic:conjugate}
        e^V \nabla (fe^{-V})^{\intercal} = e^{-V} \Rot_{-\pi/2} \cdot \nabla(f^*e^V)^{\intercal}. 
    \end{equation}


Note that given $f = u+\ii v: \Omega \to \CC$ with $u,v$ real massive harmonic, the existence of an $\alpha$-conjugate is not obvious. 
However, if $\Omega$ is connected and $x \in \partial U$ is fixed, there is at most one $\alpha$-conjugate $f^*$ with $f^*(x)=0$, and the values of $f^*$ can be recovered by integrating \eqref{eq:harmonic:conjugate}.

\subsection{Scaling limit for the inverse Kasteleyn matrix}

The scaling limit of the inverse Kasteleyn matrix will be given in terms of a suitable massive Green function, which we now introduce. (In fact, the scaling limit itself will involve a number of transformations of this Green function). Recall the mass function $M(x) = \|\nabla V(x)\|^2 + \Delta V(x)$ defined in \eqref{eq:mass_intro}.

\begin{defn}\label{def:mGreen}
    The \textbf{(continuous) massive Green function} in $\Omega$, with mass $\mass(\omega)$ at position $\omega \in \Omega$, and with Dirichlet boundary condition is  
    \begin{equation}\label{eq:def:mGreen}
        \mGreen (\omega_1,\omega_2) = \int_0^\infty p_t^\Omega (\omega_1,\omega_2) \Ebr{\omega_1}{\omega_2}{t}  \left( \exp \left( - \int_0^t \mass(B_s) \mathrm{d}s \right) \right) \mathrm{d}t
    \end{equation}
    where $p_t^\Omega (\omega_1,\omega_2)$ is the transition probability of ordinary Brownian motion killed upon reaching $\partial \Omega$, and $\Pbr{\omega_1}{\omega_2}{t}$ is the law of Brownian bridge from $\omega_1$ to $\omega_2$ of duration $t$, while $\Ebr{\omega_1}{\omega_2}{t}$ is the corresponding expectation. 
\end{defn}
Note that the expectation in the definition of the Green function ~\eqref{eq:def:mGreen} is well-defined due to the log-convexity assumption  (Assumption~\ref{assumption:pos}), which makes $M$ a nonnegative function.

The following properties of the massive Green function are easily established. 

\begin{lemma}
The massive Green function is symmetric: $\mGreen(\omega_1,\omega_2) = \mGreen(\omega_2,\omega_1)$ for all $\omega_1 \neq \omega_2 \in \Omega$.
Furthermore, $G^m$ can be characterised as follows: for each fixed $\omega_1 \in \Omega$, the massive Green function is the unique solution in $L^2(\Omega) \cap \cC^0(\overline{\Omega}\setminus \{\omega_1\},\RR)$ of the boundary value problem:    \begin{equation}\label{eq:green:distribution}
        \Delta_2 \mGreen(\omega_1,\cdot) = \mass(\cdot)\mGreen(\omega_1,\cdot) + \delta_{\omega_1} ( \cdot), 
    \end{equation}
    in the sense of distributions, and with boundary conditions $\mGreen(\omega_1,\omega_2) = 0$ for $\omega_2 \in \partial \Omega$.
    \end{lemma}

\begin{proof}
We include a short proof here for completeness. The symmetry follows at once from the time-reversibility of the law of Brownian bridge.
The uniqueness of the boundary value problem satisfied by $\mGreen$ follows from the maximum principle for the massive Laplacian $\Delta - M I$, which holds under the positivity Assumption~\ref{assumption:pos}.
Indeed, if $g_1, g_2 \in L^2(\Omega) \cap \cC^0(\overline{\Omega}\setminus \{\chi_1\},\RR)$ are two solutions of Equation \eqref{eq:green:distribution} in the sense of distributions, then $h = g_1-g_2 \in L^2(\Omega) \cap \cC^0(\overline{\Omega}\setminus \{\omega_1\},\RR)$ is a solution in the sense of distributions of 
    $$
        \Delta h = \mass h
    $$
    with zero boundary conditions. 
    By elliptic regularity (see \cite[Chapter 10.1 (5)]{LiebLoss}, $h \in \cC^{\infty}(\overline{\Omega},\RR)$, hence by the maximum principle $h=0$ so $g_1 = g_2$.
\end{proof}



We now introduce a function of two variables $\kappa$ which plays a prominent role in our results below, and for which we give the three equivalent definitions. For this it is useful to make use of \textbf{Wirtinger derivatives} (i.e., holomorphic and antiholomorphic derivatives), $$\dz = dx + \ii dy = \frac12 \overline{\nabla}, \text{\ and \ }\dzbar = dx - \ii dy = \frac12 \nabla.
$$
Thus in terms of Wirtinger derivatives, if we think of $\alpha$ as a complex-valued function rather than a vector field, since $\alpha = \nabla V$ we see that
\begin{equation}\label{eq:alphadzbar}
\alpha =  2 \dzbar V. 
\end{equation}
With this in hand we now introduce the function $\kappa$:
\begin{equation}\label{eq:def:kappa}
    \begin{aligned}
        \kappa(\omega_1,\omega_2) 
        &= \overline{\nabla_2} \mGreen(\omega_1,\omega_2) - \overline{\alpha}(\omega_2)\mGreen(\omega_1,\omega_2)\\
        &= \overline{\nabla_2} \mGreen(\omega_1,\omega_2) - \overline{\nabla} V(\omega_2)\mGreen(\omega_1,\omega_2)\\
        &= 2\left(\dz_2 \mGreen(\omega_1,\omega_2) - \dz V(\omega_2)\mGreen(\omega_1,\omega_2)\right)
    \end{aligned}
\end{equation}
(The expressions $\nabla_2, \dz_2, \ldots$ indicate that the differentiation is taking place with respect to the second variable). As the last expression suggests, $\kappa$ may be seen as a kind of \textbf{massive gradient} of the massive Green function.

Note that the real part and imaginary parts of $\kappa$ are massive harmonic in the first variable away from $\omega_2$, since derivatives in the first and second variables commute. As for the behavior with respect to the second variable, using that $\Delta = 4\dz \dzbar$, we obtain that $\kappa$ satisfies
\begin{equation}\label{eq:holomorphy:kappa}
    \begin{aligned}
        \dzbar_2 \kappa(\omega_1,\omega_2) 
        &= 2 \dzbar_2 \dz_2 \mGreen(\omega_1,\omega_2) -\overline{\alpha}(\omega_2)\dzbar_2\mGreen(\omega_1,\omega_2)-2\dzbar\dz V(\omega_2)\mGreen(\omega_1,\omega_2)\\
        &= \frac12 \Delta_2 \mGreen(\omega_1,\omega_2) -\overline{\alpha}(\omega_2)\dzbar_2\mGreen(\omega_1,\omega_2)-\frac12\Delta V(\omega_2)\mGreen(\omega_1,\omega_2)\\
        &\overset{\eqref{eq:green:distribution},\eqref{eq:pos}}{=} \frac12 \mass(\omega_2) \mGreen(\omega_1,\omega_2) + \frac12 \delta_{\omega_1}-\overline{\alpha}(\omega_2) \dzbar_2\mGreen(\omega_1,\omega_2) -\frac12\mGreen(\omega_1,\omega_2)(\mass(\omega_2)-|\alpha(\omega_2)|^2)\\
        &= \frac12 \delta_{\omega_1}- \frac{\overline{\alpha}(\omega_2)}{2}\kappabar(\omega_1,\omega_2)\\
    \end{aligned}
\end{equation}

Because of \eqref{eq:massive_holo}, this identity means that $\kappa(\omega_1, \cdot)$ is $(-\bar \alpha)$-massive holomorphic (in the second variable) away from $\omega_1$.

We are now ready to state our result concerning the scaling limit of the inverse Kasteleyn matrix associated to the near-critical weights defined in \eqref{Eq:offcriticaledgeweights} (we recall the choice of gauge function and thus Kasteleyn matrix is specified in Section \ref{SS:Gauge}.) This theorem holds for any (sequence of) isoradial superposition lattices $G_\eps$ (where $\eps>0$ denotes the common radius of all faces), requiring only the bounded angle assumption and straight boundary assumption that will be introduced formally in Section~\ref{subsec:straight}. We also introduce the following notation: for a set $S \subset \CC^2$, $\diagonal(S)$ denotes the set of its diagonal elements, that is $\diagonal(S) = \{(z,z) \in S, z \in \CC\}$. 
More generally, for $S \subset \CC^k$, $\diagonal(S)$ denotes the set of elements having at least two coinciding coordinates.

\begin{thm}\label{thm:isoradial}
    Assume that $\Omega$ and $\Primal$ have a straight boundary centered at the Temperleyan corner $\partial \Omega \cap B(b^*, \eta) = L$, as in Section~\ref{subsec:straight}.
    Let $\kappa$ be as in \eqref{eq:def:kappa}. Then there exists a unique $\alpha$-conjugate function $\kappastar(\cdot, \omega_2)$ in $\Omega \setminus \{\omega_2\}$ satisfying $\kappastar(\beta^*,\omega_2)=0$ at the Temperleyan corner.

    Furthermore, let $x_1 = x_1(\eps) \in \Primal$, $y_1 = y_1(\eps)$ and $w_2 = w_2(\eps) \in \White$, $w_2 = w_2(\eps)$. 
    Let $x_2^-x_2^+$ denote the edge of $\Primal$ corresponding to $w_2$, with $K_{w,x_2^+}>0$, $K_{wx_2^-}<0$. 
    Uniformly for $(x_1,w_2)$, resp. $(y_1,w_2)$, in compact subsets of $\Omega^2 \setminus \diagonal(\Omega^2)$,
        \begin{equation}\label{eq:thm:invK}
        \begin{aligned}
        K^{-1}(x_1,w_2) &=  \sqrt{\tan(\theta_{w_2})}\langle\kappabar(x_1,w_2) , x_2^+-x_2^-\rangle+o(\eps)\\
         K^{-1}(y_1,w_2)
	&=  \ii\sqrt{\tan(\theta_{w_2})} \langle\kappastarbar(y_1,w_2), x_2^+-x_2^-\rangle+o(\eps)
        \end{aligned}
        \end{equation}
    
    Moreover, we have the following bounds: uniformly for $(x_1,w_2)$ in compact subsets of $(\Omega \cup L)^2$ and $(y_1,w_2)$ in compact subsets of $(\Omega \cup L) \times (\Omega \cup L \setminus \{\beta^*\})$,
    \begin{equation}\label{eq:bound:isoradial}
        K^{-1}(x_1,w_2) = O\left(\frac{\eps}{|x_1-w_2|}\right) \quad ; \quad 
        K^{-1}(y_1,w_2) = O\left(\frac{\eps}{|y_1-w_2|}\right).
    \end{equation}
\end{thm}
    Note that the existence of the $\alpha$-conjugate $\kappastar$ is a part of the statement. The proof of this theorem will be given in Section \ref{subsec:proof:thm:b0}.

\subsection{Scaling limit for moments of height function}
Using the asymptotic estimate of the inverse Kasteleyn matrix obtained in Theorem~\ref{thm:isoradial}, we can compute the asymptotic moments of the height field.
This kind of computation is by now very well understood in the critical (non-massive) setting, see~\cite{Ken02} for the square lattice case, \cite{dT07} for the isoradial case (full plane), and~\cite{Li17} for the Temperleyan isoradial case.

Let $h_{\eps}$ be the height function associated with a random dimer configuration (on a graph $G_\eps$ subject to the same assumptions as in Theorem \ref{thm:isoradial}), using the reference flow $\omega$ defined by setting $\omega(wb)$ to be the probability that the edge $wb$ belongs to a random dimer configuration, and $\omega(wb) = - \omega(bw)$. Recall that $h_{\eps}$ is a random real-valued function defined on the faces of the superposition (medial) graph $G$. Recall the definition of $\kappa, \kappastar$ from Theorem~\ref{thm:isoradial}, and let
\begin{equation}\label{eq:def:F_i}
    F_0 = \frac{\kappastar - \ii\kappa}{2} \quad ; \quad F_1 = -\frac{\kappastar + \ii\kappa}{2}.
\end{equation}
For a complex number $z \in \CC$, we write $z^{(0)} = z, z^{(1)} = \overline{z}$. Then, the following holds.
\begin{thm}\label{thm:height}
    Assume that $\Omega$ and $\Primal$ have a straight boundary centered at the Temperleyan corner $\partial \Omega \cap B(b^*, \eta) = L$, as in Section~\ref{subsec:straight}.
    Let $n \geq 2$ and $\zeta_1, \zeta_2, \dots, \zeta_n$ denote faces of the superposition graph $G$. Then, uniformly for $(\zeta_1,\dots, \zeta_n)$ in compact subsets of $\Omega^n \setminus \diagonal(\Omega^n)$,
    \begin{equation}\label{eq:thm:height}
    \EE\left[\prod_{i=1}^n h_\eps(\zeta_i)\right]
    \overset{\eps \to 0}{\longrightarrow} 
    \sum_{s_i \in \{0,1\}} \int_{\path_1}\cdots \int_{\path_n}  \det_{i \neq j}\Big[F_{s_i+s_j}^{(s_j)}(z_i,z_j)\Big] \prod_{i=1}^n \mathrm{d}z_i^{(s_i)}.
    \end{equation}
    where $s_i + s_j$ is computed mod 2, $\path_1, \dots, \path_n$ are any disjoint smooth paths in $\Omega \cup L$, starting at $n$ distinct points on $L$ and ending at the $\zeta_i$.
    Moreover, uniformly for $(\zeta_1,\dots, \zeta_n)$ in compact subsets of $\Omega^n$,
    \begin{equation}\label{eq:bound:height}
        \EE\left[\prod_{i=1}^n h_\eps(\zeta_i)\right] = 
        O\left(\sum_{\sigma} \prod_{i=1}^n (1+|\log(|\zeta_i-\zeta_{\sigma(i)}|)|)\right),
    \end{equation}
    where the sum is over all permutations of $\{1,\dots,n\}$ with no fixed points.
\end{thm}
    We prove this theorem in Section~\ref{sec:Li}.

    \begin{remark}[Coherence with Kenyon's \cite{Kenyon_confinv, KenyonGFF} and Li's \cite{Li17} results]\label{rem:GFF}
    As a sanity check we now explain how these results are coherent with what is known in the critical case, since the translation is not entirely obvious.
    If we let $G_i(w,z) = \mathbf{i} F_i(w,z)$, by multi-linearity of the determinant and being careful with the conjugates, rewrites as~\eqref{eq:thm:height}
    \begin{equation}\label{eq:ken:li}
        \EE\left[\prod_{i=1}^n h_\eps(\zeta_i)\right]
    \overset{\eps \to 0}{\longrightarrow} 
        (-\mathbf{i})^n\sum_{s_i \in \{0,1\}} \int_{\path_1}\cdots \int_{\path_n}  \det_{i \neq j}\big[G_{s_i+s_j}^{(s_j)}(z_i,z_j)\big] \prod_{i=1}^n (-1)^{s_i}\mathrm{d}z_i^{(s_i)}.
    \end{equation}
    Note that the variables are exchanged between our setting and Kenyon and Li's setting: in our setting, $K^{-1}(w,b) = \kappa(w,b)$ while in their setting $K^{-1}(w,b) = \kappa(b,w)$ (see for example Lemma 6 of Li's paper). 
    Hence, if we let $H_0(w,z) = G_0(z,w)$, $H_1(w,z) = G_1(z,w)$, re-indexing Equation~\eqref{eq:ken:li} we obtain 
    $$
        \EE\left[\prod_{i=1}^n h_\eps(\zeta_i)\right]
    \overset{\eps \to 0}{\longrightarrow} 
        (-\mathbf{i})^n\sum_{s_i \in \{0,1\}} \int_{\path_1}\cdots \int_{\path_n}  \det_{i \neq j}\big[H_{s_i+s_j}^{(s_i)}(z_i,z_j)\big] \prod_{i=1}^n (-1)^{s_i}\mathrm{d}z_i^{(s_i)}.
    $$
    In other words, if we define $F_+(w,z) = H_0(w,z) = \mathbf{i}F_0(z,w)$ and $F_-(w,z) = \overline{H_1}(w,z) = -\mathbf{i}\bar F_1(z,w)$, then
    $$
        \EE\left[\prod_{i=1}^n h_\eps(\zeta_i)\right]
    \overset{\eps \to 0}{\longrightarrow} 
        (-\mathbf{i})^n\sum_{s_i \in \{0,1\}} \int_{\path_1}\cdots \int_{\path_n}  \det_{i \neq j}\big[F_{s_i,s_j}(z_i,z_j)\big] \prod_{i=1}^n (-1)^{s_i}\mathrm{d}z_i^{(s_i)}.
    $$
    with the convention that 
    $$
        F_{0,0} = F_+ \quad ; \quad F_{1,0} = F_- \quad ; \quad F_{0,1} = \bar F_- \quad ; \quad F_{1,1} = \bar F_+.
    $$
    This matches the result of~\cite{Li17}, see the expression before Lemma 15 (this is not so straightforward to check since her expression is very explicit).
    This also matches Proposition 20 in Kenyon~\cite{Kenyon_confinv}, up to what seems to be a typo: the $(-\mathbf{i})^n$ is present at the last line of the proof, but not in the statement of Proposition 20.
    This also matches Proposition 2.2 in \cite{KenyonGFF}. The $(-\mathbf{i})^n$ does not appear there, but it is compensated by the fact that $F_-$ is changed into $-F_-$ compared to \cite{Kenyon_confinv}, accounting for the $(-\mathbf{i})^n$ in the determinant by multi-linearity (note that the odd terms vanish [for example because an expectation of height functions is always real], and for even terms changing $F_1$ into $-F_1$ indeed produces an overall factor $(-1)^n$).
    \end{remark}

In the work of~\cite{Li17} (which concerns 
 the critical, i.e. non-massive case), the analogue of this theorem is proved in two steps: the first step is to compute the asymptotic of $K^{-1}$, as we did in Theorem~\ref{thm:isoradial}. 
The second step is to recover the height function from $K^{-1}$ by summing along paths and to recover integrals. 
This second step can be applied with little modification to our context, see Section~\ref{sec:Li}.

To prove that the limit in Theorem~\ref{thm:height} is the sine-Gordon model with external electromagnetic field, we will need to identify $F_0, F_1$ as the solution of a certain boundary value problem. 
This is the content of the following proposition.

\begin{prop}\label{lem:F0F1} \label{lem:bc}
    The functions $F_0$ and $F_1$ satisfy
        \begin{align}\label{eq:F0F1_bulk}
            \bar \partial_w F_0(w,z) 
        &= \frac{\alpha(w)}{2} F_1(w,z) + \frac{\ii}{2}\delta_w(z) && \bar\partial_z F_0(w,z) = \frac{\bar\alpha(z)}{2}\bar F_1(w,z) - \frac{\ii}{2}\delta_w(z) 
        \\
        \partial_w F_1(w,z) 
        &= \frac{\bar\alpha(w)}{2}F_0(w,z) && \bar\partial_z F_1(w,z) = \frac{\bar\alpha(z)}{2}\bar F_0(w,z).
        \end{align}
     Their diagonal behavior is given by
     \begin{equation}\label{eq:singularity:F0F1}
         F_0(w,z) = \frac{\ii}{4\pi(w-z)}+O(\log|w-z|) \quad ; \quad F_1(w,z) = O(\log|w-z|).
    \end{equation}
            Furthermore, for $z \in \partial \Omega$, let us denote by $\tau(w)$ the tangent vector to $\Omega$ at $w$ such that $\Omega$ is on the left.
    The functions $F_0$ and $F_1$ satisfy the following boundary conditions:
    \begin{align}
        F_1(w,z) &= -F_0(w,z) \quad ; \quad w \in \partial \Omega\\
        F_1(w,z) &= \bar\tau(z)^2 \bar F_0(w,z) \quad ; \quad z \in \partial \Omega.
    \end{align}
\end{prop} 
This result is proved in Section \ref{SS:corr_bvp}.

\begin{remark}
    As a sanity check, let us compare once again with the massless case of Kenyon~\cite{Kenyon_confinv} and Li~\cite{Li17}.
    Recall the definition of $F_+, F_-$ from Remark~\ref{rem:GFF}.
    According to Proposition~15 in \cite{KenyonGFF} or Lemma~15 in \cite{Li17} in the massless case, if $\ph$ is a conformal map from $\Omega$ to the upper half-plane,
    $$
        F_+(w,z) = -\frac{\ph'(w)}{4\pi(\ph(w)-\ph(z))} \quad ; \quad F_-(w,z) = \frac{\overline{\ph'}(w)}{4\pi(\bar\ph(w)-\ph(z))}.
    $$
    (note that there is a factor $8$ between the normalization of Kenyon and the normalization of Li, which is the same as ours). By our Remark~\ref{rem:GFF}, this implies that in the massless case,
    $$
        F_0(w,z) = -\ii F_+(z,w) = -\frac{\ii \ph'(z)}{4\pi(\ph(w)-\ph(z))} \quad ; \quad F_1(w,z) = -\ii\bar F_-(z,w) = \frac{\ii\ph'(z)}{4\pi(\bar\ph(w)- \ph(z))}
    $$
    For $w \in \partial \Omega$, $\ph(w) \in \RR$ so $F_1(w,z) = -F_0(w,z)$.
    For $z \in \partial \Omega$, $\ph(z) \in \RR$ so
    $$
        \bar F_0(w,z) = \frac{\overline{\ph'}(z)}{\ph'(z)} F_1(w,z).
    $$
    and since $\arg(\ph'(z)) = \arg(\bar\tau(z))$ at a boundary point, this is consistent with the boundary condition of Proposition~\ref{lem:bc}.
\end{remark}

The boundary value problem of Proposition \ref{lem:F0F1} will play a crucial role in the identification with the sine Gordon field, which will be discussed below. 
The Dirac boundary value problem of Proposition \ref{lem:F0F1} is to find a pair of functions $F = (F_0, F_1): \Omega^2 \to \CC^2$ such that
\begin{equation}\label{eq:BV:alpha}
\begin{cases}
\bar \partial_w F_0(w,z) 
        = \frac{\alpha(w)}{2} F_1(w,z) + \frac{\ii}{2}\delta_w(z) 
        \quad ; \quad 
        \partial_w F_1(w,z) 
        = \frac{\bar\alpha(w)}{2}F_0(w,z)
        &\text{ on }\Omega^2
    \\
F_1(w,z) = -F_0(w,z) ; & \text{ if } w \in \partial \Omega\\
\end{cases}
\end{equation}
We call this boundary value problem $\BV_1(\Omega, \alpha)$.
By Proposition \ref{lem:F0F1}, the functions $(F_0,F_1)$ provided by Theorem~\ref{thm:height} solve $\BV_1(\Omega, \alpha)$.
Given a vector field $\alpha: \Omega \to \RR^2 \simeq \CC$ in $\Omega$, this problem can be conveniently rewritten in terms of the Dirac operator with complex mass 
$$
\slashed\partial_{\alpha}:=\begin{pmatrix}  \alpha & 2\bar\partial_w \\ 2\partial_w & \overline{\alpha}\end{pmatrix}.
$$
The first two lines can equivalently be rewritten as
$$
\slashed{\partial}_\alpha 
\begin{pmatrix}
    \ii F_1 & \ii \overline{F_0}\\
    -\ii F_0 & -\ii \overline{F_1}  
\end{pmatrix} 
= I \text{ on }  \Omega^2
$$

\subsection{Identification with sine-Gordon model}

\label{SS:identification:overview}

Recently, \cite{PVW} have constructed the sine-Gordon correlation functions in a setting which intersects with our own. In particular, they prove formulas for the mixed moments of Wirtinger derivatives ($\partial \phi, \bar\partial\phi$) in terms of Pfaffians of two-point massive Ising fermion correlations. 
 Here we introduce some of their key results and 
 state a theorem which shows that the correlations of the (Wirtinger derivatives) of the limiting height function agree with those of the sine-Gordon, as constructed in \cite{PVW}.

\cite{PVW} consider the case where $\Omega$ is the unit disc ${\mathbb{D}}$ and regularise the underlying Dirichlet GFF via a heat kernel regularisation. That is, they take the covariance of the Gaussian process on $\Omega$ to be
\begin{align}
C_\eps(x,y):=\int_{\eps^2}^\infty p_{\Omega}(s,x,y)\mathrm{d}s
\end{align}
where $p_{\Omega}(s,x,y)$ is the Dirichlet heat kernel on $\Omega$, i.e., the transition density of a Brownian motion (with speed two, so its infinitesimal generator is the Laplacian $\Delta$ instead of $(1/2) \Delta$) killed upon exiting $\Omega$. We let $\mathbb{P}_{\mathrm{GFF}^\#(\epsilon)}$ denote the law of this Gaussian field $\phi_\epsilon$, and following their notation, use $\langle \cdot \rangle_{\mathrm{GFF}^\#(\epsilon)}$ to denote expectation under it. They define Wick-ordered trigonometric functions $: \! \cos(\gamma \phi_\epsilon)\!:$ and $:\!\sin(\gamma\phi_\epsilon)\!:$ which contains an explicit additional normalising constant $c_\alpha$ (defined in terms of the Euler--Mascheroni constant) compared to the standard Wick-ordered notions; see  \cite[(1.3.6) and (2.2.2)]{PVW}. This normalising constant $c_\alpha$ is non-universal and depends on the choice of regularisation. 
They then introduce 
a further regularisation $\typecolon \cos (\sqrt{4\pi} \phi_\eps)\typecolon, \typecolon \sin (\sqrt{4\pi}\phi_\epsilon)\typecolon$ for the case $\gamma=\sqrt{4\pi}$, see \cite[(1.3.7)]{PVW}. Let $\rho\in C_c^\infty(\Omega)$ (this $\rho$ plays the same role as our $\alpha$, up to a constant multiple, but is in particular assumed here to be \emph{real-valued}).
The authors of \cite{PVW} define the regularised sine-Gordon expectation by 
\begin{align}
    \langle  F(\phi_\epsilon)\rangle_{\mathrm{SG}(\rho|\epsilon,\Omega)}:=\frac{\Big\langle F(\phi_\epsilon) \exp\Big(\int_\Omega \mathrm{d}A(v) \rho(v)\typecolon\cos (\sqrt{4\pi}\phi_\epsilon(v))\typecolon\Big)\Big\rangle_{\mathrm{GFF}^\#(\epsilon)}}{\Big\langle\exp\Big(\int_\Omega \mathrm{d}A(v) \rho(v)\typecolon\cos (\sqrt{4\pi}\phi_\epsilon(v))\typecolon\Big)\Big\rangle_{\mathrm{GFF}^\#(\epsilon)}},
\end{align}
where $\mathrm{d}A(v)$ is the area (Lebesgue) measure. That is, for a fixed $\epsilon$, they define a probability measure $\mathbb{P}_{\mathrm{SG}(\rho|\epsilon, \Omega)}$ on fields (in fact, smooth fields in $C^1(\Omega)$) via its Radon-Nikodym derivative with respect to $\mathbb{P}_{\mathrm{GFF}^\#(\epsilon)}$.

Fix some test functions $f_1,..,f_p, g_1,...,g_{q}\in C^\infty_c(\Omega)$ with 
pairwise disjoint supports. \cite{PVW} define their correlation functional by taking $\epsilon \rightarrow 0$ limits of the regularised versions. A key aspect of their result (contained within Theorem 1.1 of \cite{PVW}) is to show that these limits exist, thereby showing the existence of a measure associated to the sine-Gordon model. 

\begin{thm}[Part of Theorem 1.1 in \cite{PVW}] \label{thm:PVWintro}
\begin{align}
& \Big\langle \prod_{j=1}^p \partial \phi(f_j)\prod_{j'=1}^q \bar\partial \phi(g_{j'})\Big\rangle_{\mathrm{SG}(\rho|\Omega)}\nonumber \\
:=& \lim_{\epsilon\rightarrow 0} \int (\prod_{j=1}^p(-\partial f_j(z_j)\prod_{j'=1}^q(-\bar\partial g_{j'}(w_{j'})) \Big\langle \prod_{j=1}^p \phi_\epsilon(z_j)\prod_{j'=1}^q\phi_\epsilon(w_{j'})\Big\rangle_{\mathrm{SG}(\rho|\epsilon, \Omega)} \mathrm{d}z_j \mathrm{d}\bar w_{j'}.
\end{align}
\end{thm}

A more precise version of this result (still restricted to the setting we need) will be stated in Theorem \ref{thm:PVW}. An equivalent and more compact way of formulating the above result is as follows. Let $n\ge 1$ and pick ``signs'' $s_1, \ldots, s_n \in \{0,1\}$, as well as functions $f_1, \ldots, f_n \in C^\infty(\Omega)$ such that $\partial^{(s_j)} f_j$ have pairwise disjoint support. Then
\begin{align}
 \Big\langle \prod_{j=1}^n \partial^{(s_j)} \phi(f_j)\Big\rangle_{\mathrm{SG}(\rho|\Omega)}\nonumber 
 &=\lim_{\epsilon\rightarrow 0} \int \prod_{j=1}^n(-\partial^{(s_j)} f_j(z_j)) \Big\langle \prod_{j=1}^n \phi_\epsilon(z_j)
\Big\rangle_{\mathrm{SG}(\rho|\epsilon, \Omega)} \mathrm{d} z_j^{(s_j)}. 
\end{align}
Here, $\partial^{(0)}$ and $\partial^{(1)}$ refer to the Wirtinger derivatives $\partial_z$ and $\bar\partial_z$ respectively. This limit can be written in the form 
$$
 \Big\langle \prod_{j=1}^n \partial^{(s_j)} \phi(f_j)\Big\rangle_{\mathrm{SG}(\rho|\Omega)} = \int \chi^{(\mathbf{s})}_{\mathrm{SG}(\rho)}(z_1, \ldots, z_n) \prod_{j=1}^n f_j(z_j) \mathrm{d}z_j^{(s_j)},
 $$
where ${\mathbf{s}} = (s_1, \ldots, s_n) \in \{0,1\}^n$. The kernel $\chi^{(\mathbf{s})}_{\mathrm{SG}(\rho)}(z_1, \ldots,z_n)$ above can thus be viewed (somewhat informally) as the correlation kernel of the sine-Gordon mixed Wirtinger derivatives, i.e.,
\begin{equation}\label{eq:def:correlations:SG}
\chi^{(\mathbf{s})}_{\mathrm{SG}(\rho)}(z_1, \ldots, z_n) = \Big\langle \prod_{j=1}^n \partial^{(s_j)}_{z_j} \phi(z_j) \Big\rangle_{\mathrm{SG}(\rho|\Omega)} .
\end{equation}
\cite{PVW} give an explicit expression for $\chi^{(\mathbf{s})}_{\mathrm{SG}(\rho)}$ in terms of Grassmann variables or, equivalently, Pfaffians of certain operators, which we will also introduce and describe later in Section \ref{SS:identificationPVW}.

We now explain how to proceed with the identification of the limiting height field as the sine-Gordon model at the free fermion point with suitable electromagnetic field, as conjectured in \cite{BHS} and as briefly explained in the overview of this introduction. Roughly, the idea will be to say that the expression of the moments of the height function in Theorem \ref{thm:height} gives us access to the mixed Wirtinger derivative correlations and that these correlations are \emph{exactly} equal to $\chi^{(\mathbf{s})}_{\mathrm{SG}(\rho)}$ for a suitable choice of $\rho$, up to a scaling factor $\lambda^n$ which accounts for different choices of normalisations. 

In order to state this rigorously, we must however extend the preceding results in different directions:
\begin{itemize}
    
    \item We must remove the condition that $\rho$ is compactly supported in Theorem \ref{thm:PVWintro}, since ultimately $\rho$ is a multiple of $\alpha  = \nabla V$ where $V$ is assumed log-convex (this assumption is incompatible with the fact that $\alpha$ has compact support).

    \item We must extend the validity of Theorem \ref{thm:height} to domains not having a straight edge, in particular to the unit disc, so as to match the setting of Theorem \ref{thm:PVWintro}. 
    
    \item We must impose conditions on $V$ which ensure assumptions needed both for Theorem \ref{thm:height} and Theorem \ref{thm:PVWintro} to hold. 
\end{itemize}

To deal with the first issue, we consider a smooth truncation $\rho_N$ so that $\rho_N$ has compact support in $\Omega$ for each finite $N\ge 1$, and $\rho_N $ coincides with $\rho$ on compact $K_n \subset \Omega$ with $K_n \uparrow \Omega$. We then explain in Corollary \ref{Cor:PVW_notcompact} that 
\begin{equation}\label{eq:rho_notcompact}
\chi^{(\mathbf{s})}_{\mathrm{SG}(\rho)}(z_1, \ldots, z_n):= \lim_{N\to \infty} \chi^{(\mathbf{s})}_{\mathrm{SG}(\rho_N)} (z_1, \ldots, z_n)
\end{equation}
exists pointwise for each pairwise distinct $z_1, \ldots, z_n \in \Omega$ and each $\mathbf{s} \in \{0,1\}^n$. 

Concerning the second issue, we will show in Corollary \ref{cor:BV:alpha:conformal} that the validity of Theorem \ref{thm:height} can be extended to domains $\Omega$ that are ``smooth'' in the sense that they are the image of a bounded domain $U$ with a straight edge by a bijective map $T: U \to \Omega$ such that $T$ is analytic in a neighbourhood of $U$. This includes the case of $\Omega = \mathbb{D}$ being the unit disc. Therefore, for such a domain, if we let $\mathbb{P}_{(\Omega, \alpha)}$ denote the law of the limiting dimer height function in $\Omega$ with electromagnetic field $\alpha = \nabla V$ (which is already known to exist thanks to \cite{BHS}), then we have
\begin{align}
    \mathbb{E}_{(\Omega, \alpha)}\left[ (\phi, f)^n \right] = \sum_{\mathbf{s} \in \{0,1\}^n} \int_{\gamma_1} \ldots \int_{\gamma_n} v^{(\mathbf{s})}_{(\Omega, \alpha)} (z_1, \ldots ,z_n) \prod_{i=1}^n \mathrm{d}z_i^{(s_i)}
\end{align}
for a test function $f$, for some explicitly characterized mixed Wirtinger derivative correlation function 
\begin{equation}\label{eq:height_corr}
v^{(\mathbf{s})}_{(\Omega, \alpha)}(z_1, \ldots,z_n) = \mathbb{E}_{(\Omega, \alpha)} [\prod_{i=1}^n \partial_{z_i}^{(s_i)} \phi(z_i) ].
\end{equation}

Finally, concerning the third issue we will show in Section \ref{SS:identificationPVW} that when $\Omega = \DD$, for some very explicit conditions on $\alpha$, all required assumptions are satisfied. A concrete example where these conditions are satisfied is as follows. Fix a function $f\in C^\infty((0,\infty))$ such that $f'<0 $ and $(e^f)''>0$ (e.g., $f(x) $ could be of the form $f(x) = ax+b$ with $a<0$ or $f(x) = c(x+1)^2+b$ with $c<-1/2$). 

Let $V_{\HH} (z) = f ( \Im (z))$, and let $\alpha_{\HH} = \nabla V_{\HH}$. Finally, set 
\begin{equation}\label{eq:Vcond}
\alpha_{\DD} (w) = \bar g'(w) \alpha_{\HH}(g(w)),
\end{equation}
 where $g: \DD \to \HH$ is a conformal isomorphism. 


We can now state the theorem relating the height function to the sine-Gordon model. 
\begin{thm}\label{T:SG}
Suppose $\Omega = \mathbb{D}$, $\alpha = \alpha_\DD$ is as in \eqref{eq:Vcond}. For $\mathbf{s} \in \{0,1\}^n$, let $v^{(\mathbf{s})}_{(\Omega, \alpha)}$ denote the mixed holomorphic and antiholomorphic derivative correlations of the limiting height function in $(\Omega, \alpha)$ defined in \eqref{eq:height_corr}, and let $\chi^{(\mathbf{s})}_{\mathrm{SG}(\rho)}$ be the analogous sine-Gordon correlation kernel associated with $\rho \in C^\infty(\Omega)$, defined in \eqref{eq:rho_notcompact}. Then 
\begin{equation}
    v^{(\mathbf{s})}_{(\Omega, \alpha)} =  \lambda^n \chi^{(\mathbf{s})}_{\mathrm{SG}(\rho)}
\end{equation}
with 
\begin{equation}
\lambda = - (4\sqrt{\pi})^{-1},\quad     \rho = -\frac{\tilde \alpha}{8\ii \pi}, \text{ where } \tilde \alpha(z) = e^{\ii \arg g'(z)}\alpha(z), 
\end{equation}
and $g:\mathbb{H} \to \DD $ is a conformal isomorphism, $\arg g'$ being any continuous determination of the argument. 
\end{thm}

Therefore, from the point of view of correlations (and thus also of moments), the limiting height function $\phi$ with law $\mathbb{P}_{(\Omega, \alpha)}$ agrees with $\lambda$ times a sine-Gordon field SG$(\rho|\Omega)$ with $\rho $ as above. This factor $\lambda$ accounts for different choices of normalisations.
(Note that, due to the presence of a normalising constant $c_\alpha$ involving the Euler--Mascheroni constant in the definition of $\typecolon \cos (\sqrt{4\pi}) \phi \typecolon$, and the fact that $\rho = - \tilde \alpha / 8\ii \pi$, the sine-Gordon penalisation is merely proportional to $\ii \tilde \alpha$.)

\begin{remark}\label{R:alternate_alpha}
    Another possible choice of $\alpha$ so that Theorem \ref{T:SG} holds (instead of \eqref{eq:Vcond}) would be to take $V_\DD(z) = f ( \Im(g(z))$ where $f$ is as above, and $\alpha(z) = \nabla V_{\DD}(z), z \in \DD$. 
\end{remark}

\medskip In light of Theorem \ref{T:SG} it is natural to wonder if the moments identify uniquely the distribution of the height functions. To this end we state a theorem which shows that this is indeed the case as soon as the domain satisfies some very mild regularity conditions. While it may be more common to consider the moment generating function associated to a given moment sequence and verify that this has positive radius of convergence, we find it more convenient to consider the log-Laplace transform, whose generating function is given in terms of the cumulants (rather than the moments). More precisely, if $X$ is a random variable and $\kappa_n  = \langle X, \ldots, X\rangle$ denotes its $n$th cumulant, then $\log \EE (e^{\ii tX} ) = \sum_n (\ii t)^n \kappa_n / n!$. 

We find that the log-Laplace transform of the height function integrated against a test function $f$ has positive radius of convergence, hence the cumulants (and thus the moments) uniquely identify the law. As a byproduct of this approach, we obtain a Fredholm regularised determinant (of fourth order) formula for the log-Laplace transform of the sine-Gordon model, which may be of independent interest. See Section \ref{SS:Fredholm} for definitions and more details.

\begin{thm}\label{thm:moment:characterize:field}
    \label{thm:schatten} Assume that $\Omega$ is a Jordan domain and that $\partial \Omega$ is Dini-smooth, and assume that the conclusion \eqref{eq:thm:height} of Theorem \ref{thm:height} holds in this domain (thus, $\Omega$ could be a domain with a straight edge as in Theorem \ref{thm:height} on top of being Dini-smooth, or it could be smooth in the sense that it is the image of a domain with a straight edge by a conformal isomorphism that is analytic in a neighbourhood of this domain). 

        There exists an integral operator  
    \begin{align}
  K: L^{2}(\Omega\times[0,1])\oplus L^{2}(\Omega\times[0,1])
  \to L^{2}(\Omega\times[0,1])\oplus L^{2}(\Omega\times[0,1])
\end{align}
which lies in the fourth Schatten class $S_4$, such that 
$$
\langle (\phi, f); \ldots; (\phi, f)\rangle_{(\Omega, \alpha)} = (n-1)! \Tr ((fK)^n).
$$
As a result, the power series defined by the cumulants has a positive radius of convergence, so that the moments of $(\phi, f)^n$ uniquely determine the law of $(\phi,f)$. Furthermore, for $\mu\in \mathbb{C}$ with $|\mu|>0$ sufficiently small, we have the following regularized determinant formula for the moment generating function
\begin{align}
\EE_{(\Omega, \alpha)}[e^{\mu(\phi, f)}]&=\det_4(I-\mu fK)e^{\sum_{n=1}^3 \frac{(-\mu)^n}{n!}\big\langle (\phi,f);...;(\phi, f)\big\rangle_{(\Omega, \alpha)}}.
\end{align}
\end{thm}

See \cite[Chapter 3.1]{Pommerenke1992BoundaryMaps} for the definition of a Dini smooth Jordan domain. In particular, any Jordan domain whose boundary has a $C^{1,\alpha}$ parametrisation for some $0<\alpha<1$ is Dini-smooth (see \cite[(9), Chapter 3.3]{Pommerenke1992BoundaryMaps}). 
This theorem will be proved in Section \ref{SS:Fredholm}.

\paragraph{Acknowledgements.}
This project started when the first and third authors were participating in the program ``Geometry, Statistical Mechanics, and Integrability" hosted by the Institute for Pure and Applied Mathematics (IPAM) in Los Angeles, California, US, in Spring 2024,  
supported by the National Science Foundation (Grant No.~DMS-1925919).

N.B.  acknowledges the support from the Austrian Science Fund (FWF) grants 10.55776/F1002 on ``Discrete random structures: enumeration and scaling limits" and 10.55776/PAT1878824 on ``Random Conformal Fields''.

S.M. was supported by the Knut and Alice Wallenberg Grant KAW 2022.0295.

L.R. was
partially supported by the DIMERS project ANR-18-CE40-0033 funded by the French National Research Agency.

Finally, we are  indebted to Christian Webb for many highly illuminating discussions concerning the connections between this work and \cite{PVW}.

 \section{Analysis of massive discrete differentials}
 The goal of this section is to develop the tools we will need in order to prove Theorems~\ref{thm:isoradial} and~\ref{thm:height}. These can be viewed as a toolbox for discrete massive harmonic functions and discrete massive differentials which is interesting in itself, generalizing some of the tools of~\cite{ChelkakSmirnov} to the non-critical setting.
 This section is structured as follows:
 \begin{itemize}
     \item we first give an equivalent formulation of the dimer model to relate the Kasteleyn matrix to a discrete massive Laplacian, in the spirit of~\cite{BHS, Rey}. 
     This enables us to express the inverse Kasteleyn matrix in terms of derivatives of the discrete massive Green function, which are discrete massive harmonic in one variable and discrete massive differentials in the other.
     \item We develop a general theory of discrete massive harmonic functions and a discrete holomorphicity theory for discrete massive differentials, in the spirit of~\cite{ChelkakSmirnov} : we find analogues of the Harnack and Beurling estimates for discrete massive harmonic functions, and we also give a discrete Cauchy formula and regularity estimates for discrete massive differentials
     \item We apply those results to obtain asymptotic estimates of the derivatives of the discrete massive Green function.
 \end{itemize}
 In the next section, we will use these results to:
 \begin{itemize}
     \item Obtain asymptotic estimates for the inverse Kasteleyn matrix, that is prove Theorem~\ref{thm:isoradial}.
     \item Prove how Theorem~\ref{thm:isoradial} implies Theorem~\ref{thm:height}, that is how to recover the asymptotic of the heights from the asymptotic of the inverse Kasteleyn matrix.
     \end{itemize}

\subsection{Discrete massive Laplacian} 
For $x \in \infPrimal$, define its weight
\begin{equation}\label{eq:def:weight}
    \weight(x) = \frac{1}{2}\sum_{x \sim x'}\sin(2\theta_{xx'}).
\end{equation}
In comparison with Equation (2.1) of \cite{ChelkakSmirnov}, we do not include the $\eps^2$ factor in the renormalization. 
Define the discrete (critical) Laplacian $\dDelta$ on $\infPrimal$ by setting for $f: \infPrimal \to \RR$
\begin{equation}
    \dDelta f(x) = \sum_{x \sim x', x' \in \infPrimal}  c^0_{x,x'} (f(x') - f(x))  ; x \in \infPrimal.
\end{equation} 
Note the sign convention: $\dDelta$ is a negative operator which converges in the scaling limit to the continuum Laplacian after diffusive scaling. Also note that contrary to \cite{ChelkakSmirnov}, we do not renormalize the discrete Laplacian by $\eps^2\weight(x)$.

Define a \textbf{discrete mass function} 
\begin{equation}\label{eq:mass}
m(x) := e^{-V(x)}[\dDelta e^V](x), \quad x \in \Primal.
\end{equation}
By definition, $e^V$ is then a \textbf{discrete massive harmonic function} on $\Primal$:
\begin{equation}\label{eq:lambdaharmonic}
[(\dDelta - mI)e^V](x) = 0; \quad x \in \Primal.
\end{equation}
Now, note that on $\infPrimal$, since the discrete Laplacian approximates the continuum Laplacian as $\eps>0$ (see for example Lemma 2.2 of~\cite{ChelkakSmirnov}),
\begin{align}
\dDelta e^V &= \eps^2 \weight \Delta e^V + O ( \eps^3)\nonumber \\
& = \eps^2 \weight ( \Delta V + \| \nabla V \|^2 ) e^V + O( \eps^3).\label{asympLaplacian}
\end{align}
Hence, as $\eps \to 0$, recalling the definition of the continuous mass function $M$ from~\eqref{eq:mass_intro},
\begin{equation}\label{eq:mass:asymptotic}
m (x) = \eps^2 \weight(x)\mass(x) + O( \eps^3).
\end{equation}
Observe that using the positivity Assumption~\ref{assumption:pos}, this ensures that $m$ is non-negative for $\eps$ small enough.
\begin{remark}\label{rem:positivity}
    This is the only place where we need that $M$ is positive and not only non-negative.
    This is a technical assumption and can probably be removed.
\end{remark}

We will also need to discuss the boundary conditions of the Laplacian, when restricted to $\Primal$. For $f: \Primal \to \RR$, extends $f$ to a function $f_0: \infPrimal \to \RR$ by letting $f_0$ be $0$ outside $\Primal$, that is:
\begin{equation*}
    f_0(x) = f(x)~\forall x \in \Primal \quad;\quad f_0(x) = 0~\forall x \in \infPrimal \setminus \Primal.
\end{equation*} 
Then, the discrete Laplacian with \textbf{Dirichlet boundary conditions} $\dDirichletDelta$ on $\Primal$ is defined by setting, for $f:\Primal \to \RR$,
\begin{equation}\label{eq:dDirichletDelta}
    \dDirichletDelta f(x) := \dDelta f_0(x) ; x\in \Primal.
\end{equation}
We emphasise that $\dDelta$ is defined on the entire graph $\infPrimal$, whereas $\dDirichletDelta$ will always mean the discrete Laplacian on $\Primal$, acting on functions with zero boundary conditions in the above sense.

\subsection{Gauge equivalence, Kasteleyn matrix.} 

\label{SS:Gauge}
We set up an appropriate equivalent formulation of the off-critical dimer model using the massive Laplacian, analogous to \cite{BHS} and to \cite{Rey}.
As noted in \cite{Rey}, the dimer weights of Figure~\ref{Fig:rhombusintro} on the superposition graph $G$ are actually gauge equivalent to the family of weights specified by Figure~\ref{Fig:rhombusgaugechange}.


\begin{figure}
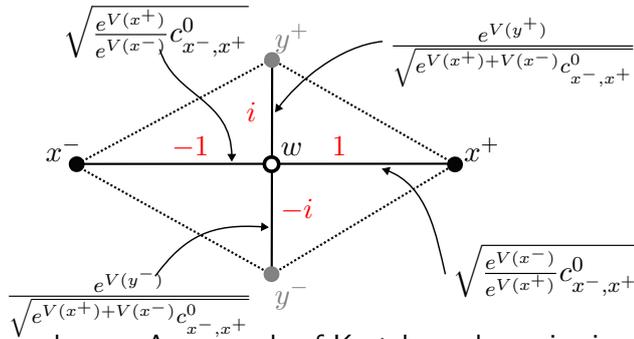

    \centering
    \begin{overpic}[scale = .95]{dimer_weight.pdf}
        \put(51,26){$w$}
        \put(10,25){$x^-$}
        \put(83,25){$x^+$}
        \put(50,0){\textcolor{gray}{$y^-$}}
        \put(50,45){\textcolor{gray}{$y^+$}}
        \put(13,46){$\sqrt{\frac{e^{V(x^+)}}{e^{V(x^-)}}c^0_{x^-,x^+}}$}
        \put(81,4){$\sqrt{\frac{e^{V(x^-)}}{e^{V(x^+)}}c^0_{x^-,x^+}}$}
        \put(3,0){$\frac{e^{V(y^-)}}{\sqrt{e^{V(x^+)+V(x^-)}c^0_{x^-,x^+}}}$}
        \put(70,44){$\frac{e^{V(y^+)}}{\sqrt{e^{V(x^+)+V(x^-)}c^0_{x^-,x^+}}}$}
        \put(32,26){\textcolor{red}{$-1$}}
        \put(60,26){\textcolor{red}{$1$}}
        \put(51,15){\textcolor{red}{$-i$}}
        \put(45,32){\textcolor{red}{$i$}}
    \end{overpic}
    \caption{Gauge change. An example of Kasteleyn phases is given in red.}
    \label{Fig:rhombusgaugechange}
\end{figure}

In reality, compared to \cite{Rey} we have performed an additional gauge change on the $B_1$ vertices (i.e.,  $y^-, y^+$ on the figure). Namely, the weights on $w y^-$ and $wy^+$ were both $1/ \sqrt{c^0_{x^-, x^+}e^{V(x^+)+V(x^-)}}$, and we have now multiplied those respectively by $e^{V(y^-)}$ and $e^{V(y^-)}$. 

\begin{remark}
    There might be other interesting choices of gauge change. What we want here is that the Kasteleyn matrix below corresponds to a drift in the vertical direction.
\end{remark}

We now specify a collection of Kasteleyn phases $\zeta$. 
Around each white vertex $w \in \Primal^*$ we pick one primal vertex $x^-$ (there are two possible choices, and we pick one among these two arbitrarily), and associate a phase to the edge $(wx^-) $ given by $\zeta_{w,x^-}=1$. Then rotating cyclically (counterclockwise) all four edges are given phases $1, \ii, -1, -\ii$. 
These phases have the property that around each face $x,w_1,y,w_2$ of $G$,
\begin{equation}
    \zeta_{w_1,x}\overline{\zeta_{w_1,y}}\zeta_{w_2,y}\overline{\zeta_{w_2,x}} = -1.
\end{equation}
Recall that $B = \Primal \cup \Dual$.
We define the \textbf{Kasteleyn operator/matrix} $K: W \times B$ by setting $K (w, b) = \tilde c_{w,b} \times$ corresponding phase, where $\tilde c_{w,b}$ is the above gauge-changed weights. 
Note that $K^* K : B\times B \to \CC$. In fact,

\begin{lemma}\label{lem:block}
$K^*K$ is naturally block diagonal where the blocks correspond respectively to $B_0 = \Primal$ and $B_1 = \Primal^*$:
\begin{equation}
\label{eq:blockdiagonal}
K^*K = 
\left(
\begin{array}{cc}
-\dDirichletDelta + m I & 0 \\
0 & \star
\end{array}
\right),
\end{equation}
where $m(x)\ge 0$ is the discrete mass function defined in \eqref{eq:mass}.
\end{lemma} 
\begin{proof}
This is a special case of Proposition 4.7 in \cite{Rey}, but we include the proof here since it is very short. We first check that $K^*K$ is block diagonal. Let $x \in \Primal, y \in \Primal^*$. If $x$ and $y$ are not adjacent in $G$ to a common vertex in $\Black^*$, then for all $w \in \Black^*$ either $K^*_{x,w} = 0$ or $K_{w,y} = 0$ so $(K^*K)_{x,y} = 0.$ Otherwise, $x$ and $y$ are the opposite black vertices of a rhombus $xw_1yw_2$, and
\begin{equation}
    (K^*K)_{x,y} = K^*_{x,w_1}K_{w_1,y}+K^*_{x,w_2}K_{w_2,x} = \sum_{i \in \{1,2\}}\tilde{c}_{w_i,x}\tilde{c}_{w_i,y}\overline{\zeta_{w_i,x}}\zeta_{w_i,y} = 0
\end{equation}
because by definition of the Kasteleyn phases $\overline{\zeta_{w_1,x}}\zeta_{w_1,y} = -\overline{\zeta_{w_2,x}}\zeta_{w_2,y}$ and by definition of the weights 
\begin{equation}
    \forall i \in \{1,2\}, \quad \tilde{c}_{w_i,x}\tilde{c}_{w_i,y} = e^{2V(y)-V(x)}
\end{equation}
Hence $K^*K$ vanishes on $\Primal \times \Primal^*$, and for the same reason it vanishes on $\Primal^* \times \Primal$.

We now check that $K^*K = I$ on $\Primal \times \Primal$. If $x_1 \neq x_2 \in \Primal$ are not adjacent in $\Primal$, then for all $w \in \Black^*$ either $K^*_{x_1,w} = 0$ or $K_{w,x_2} = 0$ so $(K^*K)_{x_1,x_2} = 0$. If $x_1 \sim x_2$, denoting by $w \in \Black^{\star}$ the middle of the edge $x_1x_2$, 
\begin{equation}
    (K^*K)_{x_1x_2} = \overline{\zeta_{w,x_1}}\zeta_{w,x_2}\tilde{c}_{w,x_1}\tilde{c}_{w,x_2} = - c_{x_1,x_2}.
\end{equation}
If $x \in \Primal$, denoting by $x_1, \dots, x_n$ its neighbours in $\Primal$ and by $w_1, \dots, w_n$ the middle of the corresponding edges, recalling the definition of edge weights $\tilde c_{x,w}$ in Figure \ref{Fig:rhombusgaugechange},
\begin{align*}
    (K^*K)_{xx} &= \sum_{i=1}^n \overline{\zeta_{w_i,x}}\zeta_{w_i,x} \tilde{c}_{w_i,x}^2\\
    & = \sum_{i=1}^n c^0_{x,x_i}\frac{e^{V(x_i)}}{e^{V(x)}} \\
&   = \sum_{i=1}^n c^0_{x,x_i}\frac{e^{V(x_i)} - e^{V(x)}}{e^{V(x)}} + \sum_{i=1}^n c^0_{x, x_i}\\
&     =   \frac{\dDelta e^{V(x)}}{e^{V(x)}}  + \sum_{i=1}^n c^0_{x,x_i}.\\
& =  m (x) - (\dDelta)_{x,x}.
\end{align*}
If one of the $x_i \in \closurePrimal \setminus \Primal$, this is equivalent to setting $\tilde{c}_{x,x_i} = 0$ that is $\dDirichletDelta$ has Dirichlet boundary conditions.
\end{proof}

\subsection{The “straight boundary” assumption.}\label{subsec:straight} 
To avoid boundary issues, we will need an additional assumption on $\Omega$ and $\Primal$. Let $\partialout \Primal$ denote the vertices of $y \in \infPrimal \setminus \Primal$ having a neighbor in $\Primal$. 

\begin{defn}
    We say that $\Omega$ and $\Primal$ \emph{have a straight boundary} (near $\beta^*\in \partial \Omega$) if there exists $\eta > 0$ (independent of $\eps$) such that 
    \begin{itemize}
        \item $\partial \Omega \cap B(\beta^*,\eta) = L$ is a straight segment in $\CC$
        \item $\partialout \Primal \subset L$
        \item $\infPrimal \cap B(\beta^*,\eta)$ is symmetric by reflection along $L$. 

    \end{itemize}\end{defn}
This is illustrated on Figure~\ref{Fig:straight:boundary}.
\begin{figure}
    \centering
    \includegraphics[scale=3]{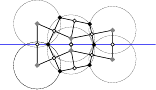}
    \caption{Straight boundary}
    \label{Fig:straight:boundary}
\end{figure}
This assumption is classical when dealing with Temperleyan graphs: see for example Section~5.3 of \cite{Ken02} or Section~3.4 of~\cite{ChelkakSmirnov}.
Note that when we make this assumption, $\eta > 0$ is a fixed macroscopic length, in particular $\eta \gg \eps$. For future reference, we denote 
by $\Gamma_\ph$ the graph $\Gamma_\ph := \Gamma \cup \ph( \Gamma \cap B)$ and by $\Omega_\ph = \Omega \cup \ph( \Omega \cap B)$ with $B = B(\beta^* , \eta)$. Equivalently under our assumptions, $\Gamma_\ph = \Gamma \cup (\Gamma^\infty \cap  B)$ and $\Omega_\ph = \Omega \cup B$. 

\begin{rmk}\label{rmk:twice:the:size}
    For convenience, we will also assume that $\partial \Omega \cap B(z, 2\eta)$ is a straight line. Note that this is not a stronger assumption, upon replacing $L$ by a segment half the size. 
\end{rmk}

When $\Omega$ and $\Primal$ have a straight boundary, $L = \partial \Omega \cap B(\beta^*,\eta)$, we denote by $\Primal \cup B(\beta^*,\eta)$ the set obtained by adding to $\Primal$ the edges and vertices contained in $\infPrimal \cap B(\beta^*,\eta)$. Let $\ph$ denote the reflection along $L$. The key observation is that if $H$ is massive harmonic with Dirichlet boundary conditions on $\Primal \cap B(\beta^*,\eta)$, that is $(\dDirichletDelta -m I)H = 0$ on $\Primal \cap B(\beta^*,\eta)$, we can extend $H$ to $\Primal \cup B(\beta^*,\eta)$ as follows: if $x \in \Primal \cap B(\beta^*,\eta)$, then 
$$
    H(\ph(x)) := -H(x).
$$
Note that for $x \in \partialout \Primal \cap B(\beta^*,\eta) \subset L$, $H(\ph(x)) = H(x) = 0$. We abusively denote by $H$ the extended function. It is easy to see that if the mass function is also symmetric in $B(\beta^*,\eta)$, that is $m(x) = m(\ph(x))$ for all $x \in \Primal \cap B(\beta^*,\eta)$, then the extended function $H$ is massive harmonic in $\infPrimal \cap B(\beta^*,\eta)$ (and vanishes on $L$). This is what we will call the \textbf{discrete Schwarz reflection principle}.

\subsection{Discrete massive Green function, resolvent identity}\label{subsec:def:green:function}
\begin{defn}[\cite{ChelkakSmirnov}, Definition 2.3]\label{def:dfreeGreen}
    The \textbf{non-massive free discrete Green function} $\dfreeGreen:(\infPrimal)^2 \to \RR$ is the unique function satisfying for all fixed $x_1 \in \infPrimal$,
    \begin{enumerate}
    \item  
    $
    \dDelta \dfreeGreen(x_1, \cdot) = \delta_{x_1=\cdot}
    $
    \item 
    $
    \dfreeGreen(x_1,x_1) = (\log(\eps)-\gamma_{Euler}-\log2)/(2\pi)
    $
    \item
    $
    \dfreeGreen(x_1,x_2) = o(|x_1-x_2|), |x_1-x_2|\to \infty.
    $
    \end{enumerate}
\end{defn}
Its asymptotic was obtained in Section~7 of~\cite{Ken02} and recalled in Theorem 2.5 of \cite{ChelkakSmirnov}:
    \begin{equation}\label{eq:green:asymptotic}
        \dfreeGreen(x_1,x_2) = \frac{1}{2\pi}\log|x_1-x_2| + O\left(\frac{\eps^2}{|x_1-x_2|^2}\right); x_1 \neq x_2.
    \end{equation}
    uniformly in $x_1,x_2$, $\eps$ and the isoradial graph (under the bounded angle assumption). Another important property of the non-massive free discrete Green function is its \emph{locality}: for $x_1, x_2 \in \infPrimal$, and any path of primal vertices $\gamma$ from $x_1$ to $x_2$, $\dfreeGreen(x_1,x_2)$ can be expressed as a function only of the angles of the rhombi adjacent to $\gamma$. This can be seen on the explicit expression of the Green function given in Theorem~7.1 of~\cite{Ken02}.
    \begin{rmk}\label{rem:locality}
        This locality of the discrete Green function also holds for the $Z$-invariant massive Laplacian on isoradial graphs, which corresponds to the constant mass case $m(x) = constant$ in our setting, see~\cite{BTR17}. This is due to the specific combinatorial structure of the weights and to the existence of an explicit multiplicative discrete exponential function. This is not known to hold in our more general setting.
    \end{rmk}
    \begin{defn}\label{def:discrete:massive:green}
        The \textbf{discrete massive Green function with Dirichlet boundary condition} $\dmGreen: \Primal^2 \to \RR$ is the unique function satisfying, for all fixed $x_1 \in \Primal$, $(\dDirichletDelta-m I) \dmGreen(x_1,\cdot) = \delta_{x_1}(\cdot)$.
        Equivalently, in matrix notation, $\dmGreen = (\dDirichletDelta-m I)^{-1}$ (both sides should be interpreted as $\Primal \times \Primal$ matrices).
    \end{defn}
    Note that by definition of the discrete Laplacian with Dirichlet boundary conditions (see Equation \eqref{eq:dDirichletDelta}), $\dmGreen(x_1, \cdot)$ has ``Dirichlet boundary conditions'': i.e. $\dmGreen (x_1, \cdot)$ is the unique function $g$ on $\Primal$ such that, extending it to be zero $\infPrimal \setminus \Primal$, $\dDelta g (x) = \delta_{x_1} (x)$ for all $x \in \Gamma$. The uniqueness in this  definition follows from the maximum principle for the discrete massive Laplacian $\dDelta - mI$, which holds under the positivity Assumption~\ref{assumption:pos} (see Remark~\ref{rem:positivity}).
    In the case $m=0$, $\dGreen := \dGreen^0$ is the \textbf{non-massive discrete Green function with Dirichlet boundary conditions} and corresponds to Definition 2.6 of \cite{ChelkakSmirnov}. 
 
    The massive and non-massive discrete Green functions have a logarithmic singularity near the diagonal. In Appendix~\ref{app:green:bound}, we recall several classical bounds, which will be useful everywhere in what follows.

    Since the Laplacian $\dDirichletDelta$ is a symmetric matrix, a direct consequence of Definition~\ref{def:discrete:massive:green} is the symmetry of the discrete Green function with Dirichlet boundary conditions:
    $$
        \forall x_1,x_2 \in \Primal,~\dmGreen(x_1,x_2) = \dmGreen(x_2,x_1). 
    $$
    This will be used without mention in what follows.

The massive and non-massive Green functions are related by the following identity, see Lemma~2.1 of~\cite{ChelkakWan} (and Proposition~4.10 of~\cite{BHS} for a continuous counterpart):
    \begin{prop}[Resolvent identity]\label{prop:resolvent}
        For all $x_1,x_2 \in \Primal$, 
        $$
            \dmGreen(x_1,x_2) = \dGreen(x_1,x_2)+\sum_{x \in \Primal}m(x)\dmGreen(x_1,x)\dGreen(x,x_2).
        $$
    \end{prop}
    \begin{proof}
        Let $x_1 \in \Primal$ be fixed and denote the right-hand side by $x_2 \to H(x_1,x_2)$. 
        It satisfies 
        \begin{align*}
            (\dDirichletDelta-m I)H(x_1,\cdot) 
            = \delta_{x_1}(\cdot) -m (\cdot)\dGreen(x_1,\cdot) + \sum_{x \in \Primal}m(x)\dGreen(x_1,x)\delta_{x}(\cdot) = \delta_{x_1}(\cdot),
        \end{align*}
        hence by definition $H$ is the massive Green function on $\Primal$. 
    \end{proof}

Let $X^m$ denote the killed random walk associated with the massive Laplacian $\dDelta-m I$. When $m=0$, we simply write $X=X^m$. Using the optional stopping theorem, we obtain the following probabilistic interpretation of massive harmonic functions. For any finite subset $\Primal' \subset \Primal_{\infty}$, let $\tau$ denote the first exit time of $\Primal'$ by $X^m$. Then, if $H: \overline{\Primal'} \to \RR$ is massive harmonic on $\Primal'$, 
\begin{equation}\label{eq:probabilistic:representation}
    H(x) = \EE_x[H(X^m_\tau)],~x \in \Primal'.
\end{equation}
In this equality, when $X^m$ dies before leaving $\Primal'$, the right-hand side should be interpreted as $H(X^m_\tau) = 0$.
    Define
    $$
        \pi^m(x) = m(x)+\sum_{y \sim x}c^0_{x,y},~x \in \Primal.
    $$
    The discrete massive Green function with Dirichlet boundary conditions can alternatively be defined as follows. Let $\tau$ denote the first hitting time of $(\closurePrimal\setminus \Primal) \cup \{x_1\}$. Then,
        \begin{equation}\label{eq:green:proba}
            \dmGreen(x_1,x_2) = \frac{\PP_{x_2}[X^m_\tau = x_1]}{\pi^m(x_1)}.
        \end{equation}
    Indeed, computing the Laplacian of the right-hand side gives Definition~\ref{def:discrete:massive:green}.
    As a consequence of this random walk representation, we obtain the following result which will be used everywhere in the sequel without mention, and relies on the positivity Assumption~\ref{assumption:pos} as well as \eqref{eq:mass:asymptotic}.
    \begin{lem}\label{lem:dmGreen<dGreen}
        $$
        \forall x_1, x_2 \in \Primal,~0 \leq \dmGreen(x_1,x_2) \leq \dGreen(x_1,x_2).
        $$
    \end{lem}
    This will often be used in combination with the bound~\ref{lem:green:bounded}.

\subsection{Massive harmonic analysis on isoradial graphs}
We now show how to adapt to the massive setting some of the basic tools in harmonic analysis on isoradial graphs developed in \cite{ChelkakSmirnov}. 
For any $\Primal \subset \infPrimal$ a subset of vertices, we denote by $\partial_E \Primal'$ its edge boundary, that is
\begin{equation}
    \partial_E \Primal' := \{x^-x^+, x^- \sim x^+, x^- \in \Primal', x^+ \in \infPrimal \setminus \Primal'\}. 
\end{equation}
The following \textbf{discrete Green formula} holds: it is not specific to isoradial graphs, the proof is a simple computation involving only the reversibility of the conductances.
\begin{lemma}[Equation (2.4) of \cite{ChelkakSmirnov}]
    For $\Primal' \subset \infPrimal$ and $F,G: \infPrimal \to \RR$,
\begin{equation}\label{eq:discrete:green:formula}
    \sum_{x \in \Primal'}[G(x)(\dDelta F)(x)-(\dDelta G)(x)F(x)] = \sum_{x^-x^+ \in \partial_E \Primal'}c^0_{x^-,x^+}[G(x^-)F(x^+)-G(x^+)F(x^-)].
\end{equation}
\end{lemma}
Note that the factor $\weight(x)$ does not appear on the left-hand side, contrary to Equation (2.4) of \cite{ChelkakSmirnov}, since we do not use the same renormalization for the Laplacian. 

We now state a lemma which implies that massive harmonic functions are globally Lipschitz: it is a generalization to the massive setting of Corollary 2.9 in \cite{ChelkakSmirnov}. For $x \in \RR^2$, $r >0$, let $B(x,r)$ denote the Euclidean ball of radius $r$ centered at $x$. Let $B_\#(x,r)$ denote its \emph{discretization}, that is the largest connected component of $B(x,r) \cap \infPrimal$. 

Let us also denote by $\cM := \sup_{x \in \Omega}|M(x)| < \infty$.

\begin{lemma}\label{lem:lipschitz}
    Let $x_0 \in \RR^2$, $R,r>0$ be such that $B(x_0,R+r) \subset \Omega$ and $H: \infPrimal \to \RR$ is massive harmonic on $B_\#(x_0,R+r)$ (i.e., $\dDelta H(x) = m(x) H(x)$ for all $x \in B_\#(x_0,R+r)$). 
    Then $H$ is globally Lipschitz on $B_\#(x_0,R)$, in the following sense: for all $w \in B_\#(x_0,R)$ corresponding to the middle of the edge $x^-x^+$,
    $$
    |H(x^+) - H(x^-) | \leq C\eps\left(r+\frac{1}{r}\right)\max_{B_\#(x_0,R+r)}|H|,
    $$
    where $C = C(\cM)$ depends only on $\cM$ and the bounded angles assumption, but not on $\eps$, $r$ and $x$. 
    
    We will often use the following weaker form: if $\cO \subset \RR^2$ is an open set, $\cC \subset \cO$ is a compact set, $H: \infPrimal \to \RR^2$ is harmonic and bounded (by a constant $C$) on $\cO$, then $H$ is Lipschitz on $\cC$ (with Lipschitz constant depending only on $C$ and $\dist(\cC, \cO^c)$).
\end{lemma}

Recall that because of the bounded angles assumption, there exist paths in the graph $\Gamma^\infty$ between points $x$ and $y$ of total length comparable to the Euclidean distance $|x-y|$, see the beginning of Section~\ref{SS:isoradial:setup}. Hence the second (weaker) part of the statement indeed follows from the first.


\begin{proof}
    The proof is an adaptation of the proof of Proposition 2.7 in Appendix 2 of \cite{ChelkakSmirnov}.
    We first prove an analogue of the mean property for massive harmonic functions. 
    Recall that $\dfreeGreen$ denotes the free discrete Green function on $\infPrimal$, see Definition \ref{def:dfreeGreen}. 
    Let $w \in B_\#(x_0,R)$ corresponding to the middle of $x^-x^+$. Define $F: \infPrimal \times \infPrimal \to \RR$ by
    \begin{equation}
        \forall x_1,x_2 \in \infPrimal,~F(x_1,x_2) = \dfreeGreen(x_1,x_2) - \frac{\log(r)}{2\pi}+\frac{r^2-|x_1-x_2|^2}{4\pi r^2}.
    \end{equation}
    as in the proof of Proposition A.2 in \cite{ChelkakSmirnov}. 
    It can be checked that this satisfies $$\dDelta F(x_1,\cdot) = \delta_{x_1}(\cdot) -\eps^2\weight(\cdot)/(\pi r^2),$$ 
    see \cite[Lemma 2.2 (i)]{ChelkakSmirnov}. Using Equation \eqref{eq:green:asymptotic}, it is also seen to satisfy 
    $F(x_1,x_2) = O(\eps^2/r^2)$ if $|x_1-x_2|=r+O(\eps)$.
    For $x_1 \in \{x^-,x^+\}$, the discrete Green formula Equation \eqref{eq:discrete:green:formula} applied on $B_{\#}(x',r)$ with the two functions $H$, $F(x',\cdot)$ gives
    \begin{align*}
        &H(x_1)-\sum_{x_2 \in B_{\#}(x_1,r)}\left(\frac{\eps^2\weight(x_2)}{\pi r^2}+m(x_2)F(x_1,x_2)\right)H(x_2) \\
        &\quad = \sum_{x_2^-x_2^+ \in \partial B_{\#}(x_1,r)}\tan(\theta_{x_2^-x_2^+})(H(x_2^-)F(x_1,x_2^+)-H(x_2^+)F(x_1,x_2^-)).
    \end{align*}
    Since $F(x_1,\cdot) = O(\eps^2/r^2)$ near $\partial B_{\#}(x_1,r)$, which contains $O(r/\eps)$ elements, and since $B(x_1,r) \subset B(x,R+r) \subset \Omega$,
    \begin{equation}\label{eq:massive:mean:value}
        \left|H(x_1)-\sum_{x_2 \in B_{\#}(x_1,r)}\left(\frac{\eps^2\weight(x_2)}{\pi r^2}+m(x_2)F(x_1,x_2)\right)H(x_2)\right| \leq \frac{C\eps}{r}\max_{x_2 \in B_\#(x_0,R+r)}|H(x_2)|
    \end{equation}
    for some absolute constant $C = C(\cM)$. Subtracting Equation \eqref{eq:massive:mean:value} with $x_1= x^-, x_1=x^+$ and using the triangle inequality, we obtain
    \begin{equation}\label{eq:bound:lipschitz}
        \begin{aligned}
            |H(x^+)-H(x^-)| 
            & \leq \sum_{x_2 \in B_{\#}(x^-,r) \cap B_{\#}(x^+,r)} m(x_2)|F(x^-,x_2)-F(x^+,x_2)|\cdot |H(x_2)|\\
            &\quad + \sum_{x_2 \in B_{\#}(x^-,r) \Delta B_{\#}(x^+,r)}\left| \left(\frac{\eps^2\weight(x_2)}{\pi r^2}+m(x_2)F(x_1,x_2)\right)H(x_2)\right|+ \frac{2C\eps\max_{B_\#(x_0,R+r)}|H|}{r},
        \end{aligned}
    \end{equation}
    where $A \Delta B$ denotes the symmetric difference $(A \setminus B) \cup (B \setminus A)$, and where in the last sum $F(x_1,x_2)$ is $F(x^-,x_2)$ if $x \in B(x^-,r) \setminus B(x^+,r)$, and $F(x^+,x_2)$ if $x \in B(x^+,r) \setminus B(x^-,r)$.
    On the one hand, since $m(x_2) = \eps^2 \weight(x_2)\mass(x_2)+O(\eps^3)$ by Equation \eqref{eq:mass:asymptotic}, and since on $B_\#(x^-,r) \Delta B_{\#}(x^+,r)$ (which contains $O(r/\eps)$ elements, $F(x_1,x_2) = O(\eps^2/r^2)$, we obtain
    \begin{equation}\label{eq:bound:1}
        \sum_{x_2 \in B_{\#}(x^-,r) \Delta B_{\#}(x^+,r)}\left| \left(\frac{\eps^2\weight(x_2)}{\pi r^2}+m(x_2)F(x_1,x_2)\right)H(x_2)\right| \leq \frac{C(\cM)\eps\max_{B_\#(x_0,R+r)}|H|}{r}
    \end{equation}
    for some constant $C = C(\cM)$ depending only on $\cM$ and the constant in the bounded angle assumption. 
    On the other hand, for all $x_2 \in B_{\#}(x^-,r) \cap B_{\#}(x^+,r)$, using that 
    \begin{align*}
        ||x_2-x^-|-|x_2-x^+|| \leq |x^--x^+| = \eps = O(\eps),
    \end{align*}
    we obtain for all $x_2 \neq x^-,x^+$
    \begin{align*}
        |F(x^-,x_2)-F(x^+,x_2)| 
        &\leq |\dfreeGreen(x^-,x_2)-\dfreeGreen(x^+,x_2)|+\frac{\left||x_2-x^-|^2-|x_2-x^+|^2\right|}{4\pi r^2}\\
        &\leq \frac{1}{2\pi}\log\left(\frac{|x_2-x^-|}{|x_2-x^+|}\right) +O\left(\frac{\eps^2}{|x_2-x^+|^2}\right)+\frac{\left||x_2-x^-|^2-|x_2-x^+|^2\right|}{4\pi r^2}\\
        &= \frac{1}{2\pi}\log\left(1+O\left(\frac{\eps}{|x_2-x^+|}\right)\right) +O\left(\frac{\eps^2}{|x_2-x^+|^2}\right)\\
        &\qquad +\frac{\left|(|x_2-x^+|+O(\eps))^2-|x_2-x^+|^2\right|}{4\pi r^2}\\
        &= O\left(\frac{\eps}{|x_2-x^+|}\right)+O\left(\frac{\eps|x_2-x^+|}{r^2}\right)
    \end{align*}
    where the $O$ is uniform in $x_2,r$. Note that for $x_2 = x_1 \in \{x^-,x^+\}$,
    \begin{equation*}
        |\dfreeGreen(x^-,x_2)-\dfreeGreen(x^+,x_2)|= \left|\frac{\log(\eps)-c}{2\pi}-\frac{1}{2\pi}\log|x^--x^+|+O\left(\frac{\eps^2}{|x^--x^+|^2}\right)\right| = O(1).
    \end{equation*}
    Hence, we can write for all $x_2 \in \infPrimal$,
    \begin{equation}
        |F(x^-,x_2)-F(x^+,x_2)| = O\left(1\wedge\frac{\eps}{|x_2-x^+|}\right)+O\left(\frac{\eps|x_2-x^+|}{r^2}\right),
    \end{equation}
    and the first term on the right-hand side of Equation \eqref{eq:bound:lipschitz} can be bounded as follows:
    \begin{equation}\label{eq:bound:2}
        \begin{aligned}
        &\sum_{x_2 \in B_{\#}(x^-,r) \cap B_{\#}(x^+,r)} m(x_2)|H(x_2)|\cdot|F(x^-,x_2)-F(x^+,x_2)| \\
        &\quad \leq C\eps^2\max_{B_\#(x_0,R+r)}|H|\sup_{B(x_0,R+r)}|\mass|\sum_{x_2 \in B_{\#}(x^-,r) \cap B_{\#}(x^+,r)} \left[O\left(1\wedge\frac{\eps}{|x_2-x^+|}\right)+O\left(\frac{\eps|x_2-x^+|}{r^2}\right)\right]\\
        &\quad \leq C'r\eps\max_{B_\#(x_0,R+r)}|H|\sup_{B(x_0,R+r)}|\mass|
    \end{aligned}
    \end{equation}
    for some universal constant $C,C' > 0$. 
    Inserting Equations \eqref{eq:bound:1} and \eqref{eq:bound:2} in Equation \eqref{eq:bound:lipschitz} gives
    \begin{equation}
        |H(x^+)-H(x^-)| \leq C\eps\left(r+\frac{1}{r}\right)\max_{B_\#(x_0,R+r)}|H|.
    \end{equation}
    for some constant $C = C(\cM)$ depending only on $\cM$ and the bounded angle assumption, which is the desired result.
\end{proof}

We now state a lemma to say that massive harmonic functions with Dirichlet boundary conditions are small near the boundary: this is a \textbf{Beurling estimate}, and is very close in spirit and proof to Proposition 2.11 of \cite{ChelkakSmirnov}. 
\begin{lemma}\label{lem:beurling}
    Assume that for some $r>0$ and $x \in \Primal$, $H$ satisfies $(-\Delta^d_\#+mI)H = 0$ on $B_\#(x,r) \cap \Primal$. 
    Then,
    \begin{equation}
        |H(x)| \leq C\left(\frac{\dist(x, \closurePrimal\setminus \Primal)}{r}\right)^{\beta} \cdot \sup_{B_\#(x,r)}|H|.
    \end{equation}
    for some universal constants $C=C(\mass),\beta > 0$ (depending only on the bounded angle assumption and the diameter of $\Omega$).
    We emphasise that $\Delta^d_{\#}$ denotes the Laplacian with Dirichlet boundary conditions on $\Primal$, hence (implicitly) $H$ is assumed to vanish outside $\Primal$.
\end{lemma}

The proof of this lemma heavily relies on the positivity Assumption~\ref{assumption:pos}.
\begin{proof}
    This proof is standard. 
    Recall that $X^m$ denotes the killed random walk on $\infPrimal$, $X$ denotes the random walk with no killing (i.e. $m=0$). 
    Using Lemma 2.10 of \cite{ChelkakSmirnov}, see their Figure 3B, there exists a constant $p > 0$ such that for each $r > 0$, $x \in \infPrimal$, the probability starting from any $x' \in B_{\#}(x,2r) \setminus B_{\#}(x,r)$ that the random walk $X$ does a full turn before leaving the annulus is bounded from below by some absolute constant $p > 0$. 
    We want to show that the same holds for the $X^m$. 
    The real part $\Re(X)$ is a martingale with quadratic variation bounded from below (see Equation (1.2) of \cite{ChelkakSmirnov} and around): $[X]_t \geq ct\eps^2$ for some constant $c>0$ depending on the bounded angle assumption. 
    Hence, if $\tau$ denotes the first time when $\Re(X)$ is at distance at least $4r$ from its starting point, applying the optional stopping theorem gives $\EE[\tau] \leq 16r^2/(c\eps^2)$ and by the Markov inequality, for all $T >0$, $\PP[\tau \geq Tr^2/\eps^2] \leq 16/(cT)$.
    In particular, for $T = 32/(cp)$, this probability is bounded from above by $p/2$. 
    Hence, by union bound, the probability that the random walk $X$ starting from any $x' \in B_{\#}(x,2r) \setminus B_{\#}(x,r)$ does a full turn in time less than $Tr^2/\eps^2$ is at least $p/2$. 
    Since the killing rate is $m = \eps^2\mu\mass + O(\eps^3) = O(\eps^2)$, the probability that $X^m$ is killed before time $Tr^2/\eps^2$ is bounded from below by some positive constant depending only on the diameter of $\Omega$.
    Hence we obtain a new constant $p' >0$ depending only on $\sup_{B(x,r) \cap \Omega} |\mass|$ and the diameter of $\Omega$, such that the probability that the killed random wall $X^m$ does a full turn before leaving the annulus $B_{\#}(x,2r) \setminus B_{\#}(x,r)$ or dying is bounded from below by $p'$. 

    Let $H$, $r$ be as in the statement of the lemma. 
    Recall that $H_0$ denotes the function $H$ extended by zero outside $\Primal$. 
    If $\sigma$ denotes the first escape time of $B_\#(x,r) \cap \Primal$, then using Equation \eqref{eq:probabilistic:representation},
    \begin{equation}
        |H(x)| = |\EE_x[H_0(X_{\sigma})]| \leq \PP[X_{\sigma} \notin \Primal] \cdot \sup_{B_\#(x,r)} |H|.
    \end{equation}
    Recall that in this equation, if $X^m$ dies before escaping $B_\#(x,r) \cap \Primal$, then $H_0(X_{\sigma}) = 0$.
    Denoting $d = \dist(x, \closurePrimal\setminus \Primal)$, $x$ is surrounded by $ n = \lfloor \log_2(r/d) \rfloor$ annuli $B(x, 2^{i+1}d) \setminus B(x, 2^{i}d) \subset \Primal \cap B(x,r)$.
    In each of these annuli, $X^m$ has a positive probability to make a full turn before dying or escaping the annulus: when this happens, $X^m$ hits $\closurePrimal\setminus \Primal$ before dying or escaping $B(x,r)$. 
    Hence 
    \begin{equation}
        \PP[X_{\sigma} \notin \Primal] \leq (1-p')^n \leq (1-p')^{\log_2(r/d)-1} \leq (1-p')^{-1} \left(\frac{d}{r}\right)^{-\log_2(1-p')},
    \end{equation}
    which concludes the proof.
\end{proof}

\subsection{Discrete holomorphicity and discrete massive differentials}
In addition to the discrete massive harmonic analysis tools developed in the beginning of the section, we must also develop a discrete holomorphicity theory to analyse discrete massive differentials. 
This theory extends to the massive case the discrete holomorphicity theory of \cite{ChelkakSmirnov}. 
Our first definitions and results are reminders of \cite{ChelkakSmirnov}. 
For a rhombus $x^-y^-x^+y^+$ with center $w \in \infWhite$, and angle $\theta$, we let
\begin{equation}\label{eq:def:mu}
\begin{aligned}
        \mu_{\White}(w) &= \text{Area}(x^-y^-x^+y^+) = \eps^2 \sin(2\theta)\\
	\mu_{wx^{\pm}} &= 2 \tan(\theta) \cdot (x^\pm-w) = \ii \cdot (y^\mp-y^\pm)\\
	\mu_{wy^\pm} &= 2 \cot(\theta) \cdot (y^\pm-w) = \ii \cdot (x^\pm-x^\mp).
\end{aligned}
\end{equation}
For $b \in \infBlack$, we also set
\begin{equation}
        \mu_{\Black}(b) = \frac{1}{4}\sum_{w \sim b}\mu_{\White}(w)
\end{equation}
where the sum is over the white vertices neighbouring $b$ in the superposition (medial) graph. Note that for $x \in \Primal$, $\mu_{\Black}(x) = \eps^2\mu(x)/2$ with $\mu$ the weight defined in Equation \eqref{eq:def:weight}.
\begin{defn}[Definition 2.12 of \cite{ChelkakSmirnov}]\label{def:holomorphy:black}
	Let $w \in \infWhite$ and $x^-,y^-,x^+,y^+ \in \infBlack$ the vertices of the corresponding rhombus in clockwise order. 
    If $H$ is defined on some subset of $\infBlack$ including $x^-,y^-,x^+,y^+$, let
        \begin{equation}
	\begin{aligned}
		[\dpartial H](w) &= \frac{1}{2}\left[\frac{H(x^+)-H(x^-)}{x^+-x^-}+\frac{H(y^+)-H(y^-)}{y^+-y^-}\right]\\
            [\dbarpartial H](w) &= \frac{1}{2}\left[\frac{H(x^+)-H(x^-)}{\overline{x^+-x^-}}+\frac{H(y^+)-H(y^-)}{\overline{y^+-y^-}}\right]
        \end{aligned}
        \end{equation}
        When $H$ is defined only on $\Primal$, we also use this definition with the convention $H_{|\Primal^*} = 0$.
\end{defn}
This definition can be rephrased as follows (see page 12 of \cite{ChelkakSmirnov}):
        \begin{equation}\label{eq:alternative:def:holomorphy}
            [\dpartial H](w) = \frac{1}{4\mu_{\White}(w)}\sum_{b \sim w}\overline{\mu_{wb}}H(b) \quad ; \quad [\dbarpartial H](w) = \frac{1}{4\mu_{\White}(w)}\sum_{b \sim w}\mu_{wb}H(b)
        \end{equation}
where the sum is taken over $b \in \{x^-,y^-,x^+,y^+\}$ that is the neighbours of $w$ in the superposition graph. We also need to define a dual notion of holomorphicity for functions defined on white vertices. 
\begin{defn}[Definition 2.13 of \cite{ChelkakSmirnov}]\label{def:dirac:dual}
	Let $b \in \infBlack$ and $w_i \in \infWhite$ be the neighbors of $b$ in the superposition (medial) graph $G$ with corresponding rhombi angles $\theta_i$. If a function $F$ is defined on some subset of $\infWhite$ including the $w_i$, define
	$$
		[\dbarpartial F](b) = - \frac{1}{4\mu_{\Black}(b)}\sum_i \mu_{w_ib}F(w_i) \quad ; \quad [\dpartial F](b) = - \frac{1}{4\mu_{\Black}(b)}\sum_i \overline{\mu_{w_ib}}F(w_i).
	$$
	We call $F$ \emph{discrete holomorphic} at $b$ if $[\dbarpartial F](b)=0$.
\end{defn}
Note that we slightly abuse notation by using the same notation as in Definition \ref{def:holomorphy:black}, but since the definitions apply to different objects there is no risk of confusion.
\begin{prop}[Proposition 2.14 of \cite{ChelkakSmirnov}]\label{prop:laplacian:holomorphy}
	Let $H$ be defined on a subset of $\infBlack$. Then
	$$
		[\dDelta H](b) = 8\mu_B(b)[\dbarpartial\dpartial H](b)
	$$
	at all $b \in \infBlack$ where the right-hand side makes sense. 
    When $b \in \infDual$, the right-hand side is the (dual) discrete Laplacian on the graph $\infDual$ with conductances $c^0_{y^-y^+} = 1/c^0_{x^-x^+}$ (for all rhombi $x^-y^-x^+y^+$).
\end{prop}
\begin{proof}
    The proof is a straightforward computation, see the proof of Proposition 2.14 in \cite{ChelkakSmirnov}. The factor $\mu_{\Black}(b)$ comes from our different normalization of the Laplacian.
\end{proof}
\begin{defn}
    We say that a function $F: \infWhite \to \CC$ is a \emph{massive discrete differential} (on a part of $\infWhite$) if it is the discrete derivative of a massive harmonic function $F = \dpartial H$ with $H: \infPrimal \to \RR$. In this case, by the proposition above and Equation~\eqref{eq:mass:asymptotic},
    $$
        [\dbarpartial F](x) = \frac{m(x)}{4\eps^2 \mu(x)} H(x) = \frac{\mass(x)+O(\eps)}{4}H(x).
    $$
\end{defn}
Note that this definition is more restrictive than the corresponding definition of discrete holomorphic functions in \cite{ChelkakSmirnov}, but this will be enough for our purposes.

\begin{rmk}\label{rem:def:discrete:massive:differential}
    In the massless case, this notion coincides with that of a discrete holomorphic function, as used for instance in \cite{ChelkakSmirnov}. This is natural from a continuous perspective: a continuous (non-massive) holomorphic function on a simply connected domain is always the gradient of a continuous (non-massive) harmonic function, which makes this definition natural. 
    
    Although the definition above seems like a natural generalization of this notion to the massive case, we will see that in fact discrete massive differentials do not in general converge in the scaling limit to massive holomorphic functions in the sense of Definition \ref{D:massive_holo}. This is why we refer to functions satisfying the above definition as \emph{massive discrete differentials}.

    In our massive setting, the correct definition of a discrete massive holomorphic function would be any function in the kernel of the Kasteleyn matrix. We expect these functions to converge in the scaling limit towards continuous massive holomorphic functions in the sense of Definition \ref{D:massive_holo}: for example we prove this for the inverse Kasteleyn matrix in Theorem~\ref{thm:isoradial}. 
\end{rmk}

The basic building block of the discrete holomorphicity theory is the \emph{discrete Cauchy kernel}, which is the discrete analogue of the Cauchy kernel $k(z,w) = 1/[2\ii \pi (z-w)]$ in complex analysis. It was introduced by Kenyon in \cite{Ken02} for periodic isoradial graphs, but his arguments extend to the non-periodic isoradial case, see Appendix A.1 of \cite{ChelkakSmirnov}. For complex numbers $z, \zeta \in \CC$, let
\begin{equation}\label{eq:def:proj}
    \Proj[z,\zeta] = \Re\left(z \frac{\overline{\zeta}}{|\zeta|}\right)\frac{\zeta}{|\zeta|}
\end{equation}
denote the orthogonal projection of $z$ onto the line $\zeta \RR$.
\begin{thm}[Theorems 4.1, 4.2, 4.3 of \cite{Ken02}, Theorem 2.21 of \cite{ChelkakSmirnov}]\label{thm:cauchy:kernel}
    Let $w_2 \in \infWhite$. There exists a unique function $k = \Cauchy(\cdot, w_2): \infBlack \to \CC$ called the Cauchy kernel such that
    \begin{itemize}
        \item $[\dbarpartial k](w_1) = 0$ for all $w_1 \neq w_2$ and $[\dbarpartial k](w_2) = 1/\mu_{W}(w_2)$
        \item $|k(b_1)| \to 0$ as $|b_1-w_2| \to \infty$.
    \end{itemize}
    The Cauchy kernel has the following asymptotics:
    $$
        \begin{aligned}
        \Cauchy(x_1,w_2) &= \frac{2}{\pi}\Proj\left[\frac{1}{x_1-w_2};\overline{x_2^+-x_2^-}\right] + O\left(\frac{\eps}{|x_1-w_2|^2}\right) \quad ; \quad x_1 \in \infPrimal\\ 
        \Cauchy(y_1,w_2) &= \frac{2}{\pi}\Proj\left[\frac{1}{y_1-w_2};\overline{y_2^+-y_2^-}\right] + O\left(\frac{\eps}{|y_1-w_2|^2}\right) \quad ; \quad y_1 \in \infDual\\ 
        \end{aligned}
    $$
    where $x_2^+y_2^+x_2^-y_2^-$ denotes the rhombus with center $w_2$ and the $O$ holds uniformly in $\eps, x_1, y_1, w_2$ and the isoradial graph $\infPrimal$.
\end{thm}

The discrete Cauchy formula of Proposition 2.22 of \cite{ChelkakSmirnov} can be adapted to the massive setting as follows. Let $\Primalbis \subset \infPrimal$ denote a simply connected subgraph and $\Dualbis$ denote its restricted dual (see Figure~\ref{fig:chelkak-smirnov} and below).
Let $x_0w_0x_1w_1 \dots x_nw_n$ denote the closed counterclockwise polyline boundary of the simply connected domain obtained by gluing all faces of $\Primalbis$ (see Figure~\ref{fig:chelkak-smirnov}), with vertices alternating between $\Primalbis$ and $\White$. 
Denote
$$
    \partial \Primalbis = \{x_0, \dots, x_n\} \quad ; \quad \partial W(\Primalbis) = \{w_0, w_1, \dots, w_n\},
$$
respectively the black and white vertices on the thick red outer line of Figure~\ref{fig:chelkak-smirnov}.

\begin{figure}
    \centering
    \begin{overpic}[scale = 3]{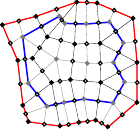}
        \put(6,48){$x_i$}
        \put(4,16){$x_{i+1}$}
        \put(7,31){$w_i$}
        \put(22,32){$y(w_i)$}
        \put(77,33){$y_{k+1}$}
        \put(76,21){$y_k$}
        \put(87,26){$\tilde{w}_k$}
        \put(100,19){$x(\tilde{w}_k)$}
    \end{overpic}
    \caption{A simply connected subgraph $\Primalbis$ of $\infPrimal$, in black. Its restricted dual $\Dualbis$ in grey.
The polylines $x_0w_0\dots x_nw_n$ and $y_0\tilde{w}_0\dots y_m \tilde{w}_m$ are respectively the outer thick red line and inner thick blue line.} 
    \label{fig:chelkak-smirnov}
\end{figure}
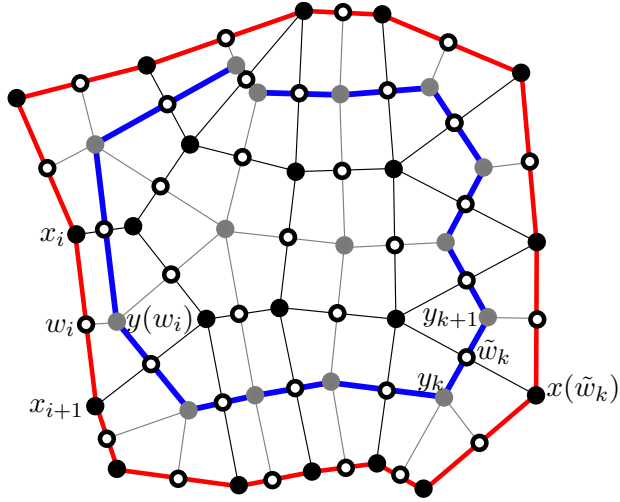

Note the difference with the notation $\partialout \Primalbis$ for the outer vertex boundary, which we introduced in Section~\ref{subsec:straight}. 
For all $i$, $x_i, x_{i+1}$ are the opposite vertices of a rhombus with center $w_i$. 
Let us denote by $y(w_i) \in \Dualbis$ the unique vertex of this rhombus lying strictly inside the polyline (see Figure~\ref{fig:chelkak-smirnov}).
For $F: W \to \CC, G: \Dualbis \to \CC$, we can define a contour integral
\begin{equation}
    \begin{aligned}
        \oint_{\partial \Primalbis}F(w)G(y)\dd w &:= \sum_{i=0}^{n}F(w_i)G(y(w_i))\cdot (x_{i+1}-x_i),\\
    \end{aligned}
\end{equation}
with $x_{n+1} = x_0$. Similarly, let $y_0\tilde w_0 y_1 \tilde w_1 \dots y_m \tilde w_m$ denote the closed polyline path, enumerated in counter-clockwise order, passing through the centers of all faces touching $x_0 x_1\dots x_n$ from inside (see also Figure~\ref{fig:chelkak-smirnov}). 
Denote
$$
    \partial \Dualbis = \{y_0, \dots, y_m\} \quad ; \quad \partial W(\Dualbis) = \{\tilde w_0, \tilde w_1, \dots, \tilde w_m\},
$$
respectively the grey and white vertices on the thick blue inner line of Figure~\ref{fig:chelkak-smirnov}.
For all $k$, $y_k, y_{k+1}$ are the opposite vertices of a rhombus with center $\tilde k_i$. 
Let us denote by $x(\tilde w_k) \in \partial \Lambda$ the unique black vertex of the rhombus containing $\tilde w_k$ lying on the polyline $x_0w_0\dots x_nw_n$.
For $F: W \to \CC, G: \Primalbis \to \CC$, we define a contour integral
\begin{equation}
    \begin{aligned}
        \oint_{\partial \Dualbis}F(w)G(x)\dd w &:= \sum_{i=0}^{m}F(\tilde w_i)G(x(\tilde w_i))\cdot (y_{i+1}-y_i).
    \end{aligned}
\end{equation}
Finally, let $\White(\Primalbis)$, resp. $\Black(\Primalbis)$, denote the set of white, resp. black, vertices inside (or on the boundary of) the closed polyline $x_0  \dots w_n$ (note that $B(\Lambda)$ contains both primal and dual black vertices). Define
\begin{equation}
    \Int \Primalbis = \Primalbis \setminus \partial \Primalbis 
    \quad ; \quad 
    \Int \White(\Primalbis) = \White(\Primalbis) \setminus \partial \White(\Primalbis),
    \quad ; \quad 
    \Int \Black(\Primalbis) = \Black(\Primalbis) \setminus \partial \Primalbis.
\end{equation}
Note that $\Int(\Lambda) \subset \infPrimal$ contains all \emph{primal} black vertices (strictly) inside the polyline $x_0, \ldots, w_n$, whereas $\Int B(\Lambda) \subset \infPrimal \cup \infDual$ contains \emph{all} black vertices (whether primal or dual) strictly inside that polyline. Thus, $\Int B(\Lambda) = \Int \Primalbis \cup \Dualbis$; where recall that $\Dualbis$ denotes the restricted dual.

Up to now everything was as in \cite{ChelkakSmirnov}, but we now start discussing how this theory needs to be adapted in our massive context. We are ready to adapt the discrete Cauchy formula Proposition 2.22 of \cite{ChelkakSmirnov} to the massive setting.
\begin{prop}\label{prop:cauchy:formula}
    Let $\Primalbis \subset \infPrimal$ be a simply connected subgraph and $H: \infPrimal \to \RR$ be a real discrete massive harmonic function on $\Int \Primalbis$. 
    Denote by $F := \dpartial H : \infWhite \to \CC$. 
    Recall that this is what we call a discrete massive differential. 
    Then, for any $w_* \in 
    \Int \White(\Primalbis)$, 
    $$
        F(w_*) = \frac{1}{4\ii}\left[\oint_{\partial \Primalbis} F(w) \Cauchy(y,w_*) dw +\oint_{\partial \Dualbis} F(w ) \Cauchy(x,w_*) dw\right] -\frac{1}{2} \sum_{x \in \Int(\Lambda)}\Cauchy(x,w_*)m(x)H(x).
    $$
\end{prop}
\begin{proof}
    The proof is the same as in \cite{ChelkakSmirnov}, taking into account the massive terms.  By the definitions of the discrete Cauchy kernel and of $\dbarpartial$ (see the alternative definition Equation \eqref{eq:alternative:def:holomorphy}),
    $$
        \begin{aligned}
        4F(w_*) &= \underset{w \in \Int\White(\Primalbis), b \sim w, b \in \Black(\Primalbis)}{\sum\sum}F(w)\mu_{wb}\Cauchy(b,w_*)\\
        &= \underset{b  \in \Black(\Primalbis),w \sim b, w \in \Int\White(\Primalbis)}{\sum\sum}\Cauchy(b,w_*)\mu_{wb}F(w)\\
        &= \sum_{x \in \partial \Primalbis, w \sim x, w \in \partial W(\Dualbis)}\Cauchy(x,w_0)\mu_{wx}F(w)-\sum_{y \in \partial \Dualbis, w \sim y, w \in \partial W(\Primalbis)}\Cauchy(y,w_*)\mu_{wy}F(w)\\
        &\quad +\underset{b  \in \Int\Black(\Primalbis),w \sim b, w \in \White(\Primalbis)}{\sum\sum}\Cauchy(b,w_0)\mu_{wb}F(w)\\
        &= \sum_{x \in \partial \Primalbis, w \sim x, w \in \partial W(\Dualbis)}\Cauchy(x,w_*)\mu_{wx}F(w)-\sum_{y \in \partial \Dualbis, w \sim y, w \in \partial W(\Primalbis)}\Cauchy(y,w_*)\mu_{wy}F(w)\\
        &\quad -4\sum_{b \in \Int\Black(\Primalbis)}\mu_B(b)\Cauchy(b,w_*)[\dbarpartial F](b).
        \end{aligned}
    $$
    Note that the first two sums are exactly the integrals appearing in the statement of the proposition by Equation \eqref{eq:def:mu}, up to a factor $\ii$. For instance, 
    \begin{align*}
        \sum_{x \in \partial \Primalbis, w \sim x, w \in \partial W(\Dualbis)}\Cauchy(x,w_*)\mu_{wx}F(w) & = \sum_{w \in \partial W(\Lambda^*)} \Cauchy ( x(w), w_*) \mu_{w, x(w)} F(w)\\
        & = -\ii \sum_{i=0}^m F(\tilde w_i) \Cauchy ( x(\tilde w_i), w_*) (y_{i+1} - y_i)\\
        & = - \ii \oint_{\partial \Dualbis} F(w ) \Cauchy(x,w_*) dw.
    \end{align*}
Likewise, 
\begin{align*}
    \sum_{y \in \partial \Dualbis, w \sim y, w \in \partial W(\Primalbis)}\Cauchy(y,w_*)\mu_{wy}F(w) & = \sum_{w \in \partial W(\Lambda)} F(w) \Cauchy (y(w), w_*) \mu_{w, y(w)}\\
    & = \sum_{i=0}^n F(w_i) \Cauchy ( y(w_i), w_*)  \ii (x_{i+1} - x_i)\\
    & = \ii \oint_{\partial \Primalbis} F(w) \Cauchy(y,w_*) dw 
\end{align*}

    We identify the last term. For $x \in \Int \Primalbis$, 
    $$
       4\mu_B(x)[\dbarpartial F](x) = 4\mu_B(x)[\dbarpartial \dpartial H](x) = \frac12(\dDelta H)(x) = \frac12m(x)H(x)
    $$
    since $H$ is massive harmonic on $\Int \Primalbis$. For $y \in \Dualbis$, by definition 
    $$
        [\dbarpartial F](y) = [\dbarpartial \dpartial H](x) = \frac{1}{8\mu_B(y)}  [\dDelta H](y) = 0
    $$
    since in the definition of $F = \dpartial H$ we used the convention $H_{|\infDual} = 0$ (see Definition \ref{def:holomorphy:black}). 
\end{proof}



As in Corollary 2.23 of \cite{ChelkakSmirnov}, we provide an asymptotic form of the discrete Cauchy formula using Theorem \ref{thm:cauchy:kernel}.

\begin{corollary}\label{cor:asymptotic:cauchy:formula}
    Let $\Primalbis \subset \infPrimal$ be a simply connected subgraph and $H: \infPrimal \to \CC$ be a discrete massive harmonic function on $\Int \Primalbis$. Denote by $F := \dpartial H : \infWhite \to \CC$. For any $w_0 \in \White(\Primalbis)$, denote by $x_0^+y_0^+x_0^-y_0^-$ the corresponding rhombus, in counterclockwise order. Then,
    $$
        \begin{aligned}
        F(w_0) &= \Proj\left[\frac{1}{2\ii \pi}\left(\oint_{\partial \Primal} \frac{F(w)}{w-w_0}\dd w + \oint_{\partial \Dualbis} \frac{F(w)}{w-w_0}\dd w\right)-\sum_{x \in \Int(\Primalbis)}\frac{\eps^2\mu(x)\mass(x)H(x)}{\pi(x-w_0)};\overline{x_0^+-x_0^-}\right]\\
        &\quad + O\left(\frac{\eps \cF \cL}{d^2}\right) + O\left(\cH \cM \eps \log\left(\frac{\cD}{\eps}\right)\right)
        \end{aligned}
    $$
    where $d = \dist(w_0, \partial \Dualbis)$, $\cH = \max_{x \in \Primalbis} |H(x)|$, $\cF = \max_{w \in \partial \White(\Primalbis) \cup \partial \Black(\Primalbis)}|F(w)|$ and $\cL = O(\eps(n+m))$ is the sum of the length of the polylines $\partial \Lambda$ and $\partial \Lambda^*$, $\cM = \sup_{x \in \Primalbis}|\mass(x)|$, $\cD = \diam(\Primalbis)$ (the diameter of the simply connected set obtained by gluing all faces).
\end{corollary}
\begin{rmk}
    The analogous result of \cite{ChelkakSmirnov} establishes separate asymptotics for the two different projections $\cB F$ and $\cW F$ (which we did not define). In our case, since $F = \dpartial H$ and $H$ is real on $\Primalbis$ and vanishes on $\Dualbis$, $F$ coincides with $\cB F$ by Remark 2.18 of \cite{ChelkakSmirnov}.
\end{rmk}
\begin{proof}
    Again, the proof is exactly the same as the proof of Corollary 2.23 of \cite{ChelkakSmirnov}, with an additional mass term. Let $w \in \partial \White(\Dualbis)$ correspond to the rhombus $x_1y_1x_2y_2$. Note that the edge $y_1y_2$ of the dual graph $\Dualbis$ belongs to the polyline $y_0\dots y_m$, hence $\dd w = y_2-y_1$, and since furthermore $H|_{\Gamma^*} = 0$,
    $$
        \frac{F(w)\dd w}{4\ii}= \frac{(H(x_2)-H(x_1))(y_2-y_1)}{8\ii(x_2-x_1)} \in \RR
    $$
    because $y_2-y_1$ and $x_2-x_1$ are orthogonal. Hence by the asymptotics of Theorem \ref{thm:cauchy:kernel},
    $$
        \Cauchy(y(w),w_0)\cdot \frac{F(w)\dd w}{4\ii} = \Proj \left[\frac{F(w)\dd w}{2\ii\pi(y(w)-w_0)}; \overline{x_0^+-x_0^-}\right]+O\left(\frac{\eps \cF|\dd w|}{d^2}\right).
    $$
    Similarly, let $w \in \partial \Black(\Primalbis)$ correspond to the rhombus $x_1y_1x_2y_2$. Note that the edge $x_1x_2$ of the primal graph belongs to the polyline $x_0\dots x_n$: hence $\dd w = x_2-x_1$, and
    $$
        \frac{F(w)\dd w}{4\ii} = \frac{H(x_2)-H(x_1)}{8\ii} \in \ii\RR.
    $$
    Hence by the asymptotics of Theorem \ref{thm:cauchy:kernel}, and using the fact that $\ii \Proj (z; \zeta) = \Proj ( \ii z, \ii \zeta)$, 
    $$
        \begin{aligned}
        \Cauchy(x(w),w_0) \cdot \frac{F(w)\dd w}{4\ii} 
        &= \Proj\left[\frac{F(w)\dd w}{2\ii\pi (x(w)-w_0)}; \ii \overline{y_0^+-y_0^-}\right] + O\left(\frac{\eps \cF |\dd w|}{d^2}\right) \\
        &= \Proj\left[\frac{F(w)\dd w}{2\ii\pi (x(w)-w_0)}; \overline{x_0^+-x_0^-}\right] + O\left(\frac{\eps \cF |\dd w|}{d^2}\right)
        \end{aligned}
    $$
    since $y_0^+ - y_0^-$ is orthogonal to $x_0^+-x_0^-$. Finally, for $x \in \Int(\Primalbis)$, using Equation \eqref{eq:mass:asymptotic},
    $$
        \Cauchy(x,w_0)m(x)H(x) = \Proj\left[\frac{2\eps^2\mu(x)\mass(x)H(x)}{\pi(x-w_0)};\overline{x_0^+-x_0^-}\right]+O\left(\frac{\eps^3 \cH \cM}{|x-w_0|^2}\right).
    $$
    Since 
    $$
    \sum_{x \in \Primalbis}\frac{\eps^3}{|x-w_0|^2} = O\left(\eps \log\left(\frac{\cD}{\eps}\right)\right),
    $$
    we finally obtain the corollary.
\end{proof}
For $F: \infWhite \to \CC$, we define an averaging operator
$$
    [\daverage F](z) := \frac{1}{4\mu_\Black(z)}\sum_{w \sim z, w \in \infWhite}\mu_\White(w)F(w) \quad; \quad z \in \infBlack.
$$
The point is that, when $F$ is a massive discrete differential, it is actually not that well behaved at the discrete level itself; rather, it should be viewed as the projection onto a direction $x_0^-, x_0^+$ -- which is lattice-dependent -- of $A_\# F$, which \emph{is} well behaved (e.g., it is $\beta$-Hölder for all $\beta <1$, as we will see below).

We now provide an analogue of Proposition 2.24 of \cite{ChelkakSmirnov}: for a massive harmonic function $H$ on $\Gamma^\infty$ we are able to define a \emph{discrete gradient} of $H$ (namely, $A_\# F$ where $F = \partial_\# H$ below) such that the discrete derivative in a given direction is close to the projection of the {discrete} gradient in this direction. We are also able to prove that this discrete gradient is well-behaved. 

\begin{prop}\label{prop:discrete:gradient}
    Let $u \in \CC$, $R> r >0$ and recall that $B_\#(u,R) \subset \infPrimal$ denotes the  discretized euclidean ball. Assume that $H$ is massive harmonic in $\Int B_\#(u,R)$, and denote by $F = \dpartial H$. Then for all $x_0 \in B(u,r)$, for all $x_0 \sim w_0 \in \infWhite$,
    $$
        \left|F(w_0)-\Proj\left[2[\daverage F](x_0);\overline{w_0-x_0}\right]\right| = O\left(\frac{\cF \eps R}{(R-r)^2}\right) + O\left(\cM \cH \eps \log\left(\frac{2R}{\eps}\right)\right).
    $$
    Moreover, for all $r>0$, $x_0, x_0' \in B_\#(u,r)$,
    $$
        \left|[\daverage F](x_0')-[\daverage F](x_0)\right| = O\left(\frac{|x_0'-x_0|R\cF}{(R-r)^2}\right) + O\left(\cM \cH |x_0'-x_0|\left(1+\log\left(\frac{R}{|x_0'-x_0|}\right)\right)\right)
    $$
\end{prop}
In fact this phenomenon is already apparent in the Cauchy formula above. If we had alternatively defined $\daverage F(x_0)$ as the term inside the projection in the right-hand side of the Cauchy formula, that is
$$
    \frac{1}{2\ii\pi}\left(\oint_{\partial \Primal} \frac{F(w)}{w-x_0}\dd w + \oint_{\partial \Dualbis} \frac{F(w)}{w-x_0}\dd w\right)-\sum_{x \in \Int(\Primalbis)\setminus \{x_0\}}\frac{\eps^2\mu(x)\mass(x)H(x)}{\pi(x-x_0)},
$$
this proposition would be a direct corollary of the discrete massive Cauchy formula. For later purposes, and following \cite{ChelkakSmirnov}, we prefer to use the definition of $\daverage$ as the local average.
Also note that our result in the massive setting is weaker than the corresponding result in the non-massive case: the discrete gradient is not Lipschitz (because of the $\log \eps$ factor). 
\begin{proof}
    Once again, the proof is similar to the proof of Proposition 2.24 in \cite{ChelkakSmirnov}. We apply to $\Lambda = B_\#(u,R)$ the asymptotic Cauchy formula of Corollary \ref{cor:asymptotic:cauchy:formula}. Note that 
    $$ 
        \cL \leq CR
    $$
    for some absolute constant $C$, and that $\cD = 2R$. For all $w_0 \sim x_0$, $|w_0-x_0| \leq \eps$ so for all $w \in \partial \White(\Lambda)$,
    $$
    \left|\frac{1}{w-w_0}-\frac{1}{w-x_0}\right| \leq \frac{C\eps}{(R-r)^2}.
    $$
    for some absolute constant $C$. Moreover, for all $x \in \Int(\Lambda)$, $|x_0-w_0| \leq \eps$ hence for all $x \in \Int(\Lambda) \setminus \{x_0\}$,
    $$
        \left|\frac{1}{x-w_0}-\frac{1}{x-x_0}\right| \leq \frac{C\eps}{|x-x_0|^2}.
    $$
    and since $|x_0-w_0| \geq c \eps$ for some absolute constant $c>0$, we obtain
    $$
        \begin{aligned}
        \left|\sum_{x \in \Int(\Lambda)}\frac{\eps^2}{x-w_0}-\sum_{x \in \Int(\Lambda)\setminus \{x_0\}}\frac{\eps^2}{x-x_0}\right| &\leq c\eps + \sum_{x \in \Int(\Lambda))\setminus \{x_0\}} \frac{C\eps^2}{|x-x_0|^2}\\
        &= O\left(\cM \cH \eps \log\left(\frac{2R}{\eps}\right)\right).
        \end{aligned}
    $$
    The asymptotic Cauchy formula implies 
    \begin{equation}\label{eq:F:averaged}
        F(w_0) = \Proj[S(x_0);\overline{w_0-x_0}]+ {O\left(\frac{\cF \eps R}{(R-r)^2}\right) + }O\left(\cM \cH \eps \log\left(\frac{2R}{\eps}\right)\right)
    \end{equation}
    with
    $$
        S(x_0) = \frac{1}{2\ii\pi}\left(\oint_{\partial \Lambda} \frac{F(w)}{w-x_0}\dd w + \oint_{\partial \Lambda^*} \frac{F(w)}{w-x_0}\dd w\right)-\sum_{x \in \Int(\Lambda) \setminus \{x_0\}}\frac{\eps^2\mu(x)\mass(x)H(x)}{\pi(x-x_0)}.
    $$
    Note that $S(x_0)$ does not depend on $w_0$, hence summing over $w_0 \sim x_0$ and using the identity
    $$
        \frac{1}{4\mu_\Black(x_0)}\sum_{w_0 \sim x_0}\mu_\White(w_0)\Proj[S;\overline{w_0-x_0}] = \frac{S}{2}
    $$
    valid for all $S \in \CC$ (see the proof of Proposition 2.24 of \cite{ChelkakSmirnov} for details), we obtain
    \begin{equation}\label{eq:daverage:S}
        [\daverage F](x_0) = \frac{S(x_0)}{2}+ {O\left(\frac{\cF \eps R}{(R-r)^2}\right) +} O\left(\cM \cH \eps \log\left(\frac{2R}{\eps}\right)\right)
    \end{equation}
    so in particular for $w_0 \sim x_0$, Equation \eqref{eq:F:averaged} implies that 
    $$
        \left|F(w_0)-\Proj[2[\daverage F](x_0);\overline{x_0-w_0}] \right| = O\left(\frac{\cF \eps}{R}\right) + O\left(\cM \cH \eps \log\left(\frac{2R}{\eps}\right)\right).
    $$
    This proves the first point of the lemma. 
    To prove the second point, let $x_0,x_0' \in B_\#(u,r)$. We can apply the first point in the ball $B_\#(u,R)$ to $x_0$ and $x_0'$ to obtain, for $x \in \{x_0, x_0'\}$,
    \begin{equation}\label{eq:daverage:S2}
        [\daverage F](x) = \frac{S(x)}{2}+ O\left(\frac{\cF \eps R}{(R-r)^2}\right) + O\left(\cM \cH \eps \log\left(\frac{2R}{\eps}\right)\right)
    \end{equation}
    Let $d = |x_0'-x_0|$. On the one hand, for $w \in \partial \White(B_\#(u,R))$,
    $$
        \left|\frac{1}{w-x_0}-\frac{1}{w-x_0'}\right| \leq \frac{d}{(R-r)^2}.
    $$
    Hence,
    $$
        \left|\oint_{\partial B_\#(u,R)} \frac{F(w)}{w-x_0}\dd w - \oint_{\partial B_\#(u,R)} \frac{F(w)}{w-x_0'}\dd w\right| = O\left(\frac{dR\cF}{(R-r)^2}\right),
    $$
    and the same holds for the second integral. 
    On the other hand, for $x \in \Int(B_\#(u,R)) \setminus \{x_0,x_0'\}$,
    $$
        \left|\frac{1}{x-x_0}-\frac{1}{x-x_0'}\right| = \frac{d}{|x-x_0||x-x_0'|}\leq
        \left\{
        \begin{array}{cc}
            2/|x-x_0| & \text{ when }x \in B(x_0,d/2) \\
            2/|x-x_0'| & \text{ when }x \in B(x_0',d/2) \\
            3d/|x-x_1|^2 & \text{ otherwise }.
        \end{array}
        \right. 
    $$
    Hence, splitting the sum in three, we obtain
    $$
        \left|\sum_{x \in \Int(B_\#(u,R)) \setminus \{x_0,x_0'\}}\frac{\eps^2}{x-x_0}-\frac{\eps^2}{x-x_0'}\right| \leq Cd\left(1 + \log\left(\frac{R}{d}\right)\right) 
    $$
    for some absolute constant $C$. Finally, we get
    $$
        \left|S(x_0)-S(x_0')\right| = O\left(\frac{dR\cF}{(R-r)^2}\right) + O\left(\cM \cH d\left(1+\log\left(\frac{R}{d}\right)\right)\right).
    $$
    Combining this with Equation \eqref{eq:daverage:S2}, since $\eps \leq d$, this concludes the proof of the lemma.
\end{proof}

\subsection{Convergence of the discrete massive Green function}
We  state a lemma to prove convergence of the discrete massive Green function with Dirichlet boundary conditions towards its continuous analog.
This lemma heavily relies on the Assumption~\ref{assumption:domain} that the graph $\Gamma$ approximates the domain $\Omega$.
\begin{lemma}\label{lem:cv:green}
 Suppose $x_1 = x_1(\eps), x_2 = x_2(\eps) \in \Primal$. 
 Then, uniformly for $(x_1,x_2)$ in compact subsets of $\Omega^2 \setminus \cD(\Omega^2)$,
	$$
	    \dmGreen(x_1,x_2) = \mGreen(x_1,x_2) + o(1).
	$$
    If $\Omega$ and $\Primal$ have a straight boundary near $\beta^*$ (say on the ball $B = B(\beta^*, \eta)$, so that $L = \partial\Omega \cap B$ is straight, recall that $\Gamma_\ph = \Gamma \cup \ph( \Gamma \cap B)$ and $\Omega_\ph = \Omega \cup B$, where $\ph$ is the reflection across the straight line $L$ defined in Section \ref{subsec:straight}.
    Recall that in this case, $\dmGreen$ is naturally extended to $\Primal_\ph$.
    Then, uniformly for $x_1 \in \Gamma_\ph$ over compact  subsets of $\Omega_\ph$, and $x_2 \in \Gamma_\ph$,
    $$
      \dmGreen(x_1,x_2) = O\left(1+|\log(|x_1-x_2| \wedge |\ph(x_1)-x_2|)|\right)
    $$ 
    where we set $\ph (x_1) = \infty$ if $x_1 \notin B$. 
\end{lemma}
\begin{proof}
    The second point of the lemma is a direct consequence of Lemmas~\ref{lem:dmGreen<dGreen} and~\ref{lem:green:bounded}.
    
    We now prove the first point of the lemma.
    By Lemma~\ref{lem:green:bounded}, Lemma~\ref{lem:lipschitz} and (ii) of Assumption~\ref{assumption:domain}, $\dmGreen$ is a (sequence of) Lipschitz function in each variable, uniformly bounded over $\eps > 0$ and compact subsets of $\Omega^2 \setminus \cD(\Omega^2)$, with uniform Lipschitz constants over $\eps$ and over compact subsets of $\Omega^2 \setminus \cD(\Omega^2)$.
    By Lemma~\ref{lem:lip:two:var}, $\dmGreen$ is hence Lipschitz as a function of two variables, with uniform Lipschitz constant over $\eps$ and over compact subsets of $\Omega^2 \setminus \cD(\Omega^2)$.
    Hence by the Arzela-Ascoli theorem we can find a subsequence along which $\dmGreen \to g \in \cC^0(\Omega^2 \setminus \cD(\Omega^2),\RR)$ uniformly on compact subsets of $\Omega^2 \setminus \cD(\Omega^2)$. 

    To conclude the proof, it only remains to identify the limit $g$. 
    Let $\chi_1 \in \Omega$. The rest of the proof consists in identifying $g(\chi_1,\cdot)$ uniquely.
    
    Let us identify first the boundary conditions for $g(\chi_1,\cdot)$.
    Let $\chi_2 \in \Omega$, $|\chi_1-\chi_2| =: 2d > 0$.
    Let $x_1 = x_1(\eps), x_2 = x_2(\eps) \in \Primal$ be such that $x_1 \to \chi_1, x_2 \to \chi_2$ as $\eps \to 0$.
    By Lemma \ref{lem:green:bounded} and Lemma \ref{lem:beurling}, we obtain that for all $x_2 \in \Primal \cap B(\chi_2,d)$,
    $$
        |\dmGreen(x_1, x_2)|\leq C(1+|\log d|) \left(\frac{\dist(x_2, \closurePrimal\setminus \Primal)}{d}\right)^{\beta}.
    $$
    By (i) and (ii) 
    of Assumption~\ref{assumption:domain}, $\dist(x_2, \closurePrimal\setminus \Primal) \overset{\eps \to 0}{\longrightarrow} \dist(\chi_2, \partial \Omega)$. Thus, in the $\eps \to 0$ limit, we obtain
    \begin{equation}
   \label{eq:control_g}     |g(\chi_1, \chi_2)|\leq C(1+|\log|\chi_1-\chi_2|) \left(\frac{\dist(\chi_2, \partial \Omega)}{|\chi_1-\chi_2|}\right)^{\beta}.
     \end{equation}
    hence $g(\chi_1, \cdot)$ can be extended by $0$ on $\partial \Omega$. 
    More precisely, if $\overline{g}(\chi_1,\cdot)$ coincides with $g(\chi_1,\cdot)$ on $\Omega$ and $\overline{g}(\chi_1,\chi_2) = 0$ for $\chi_2 \in \partial \Omega$, then $ \overline{g}(\chi_1,\cdot) \in \cC^0(\overline{\Omega}\setminus \{\chi_1\},\RR)$.

    
    
Let $f \in \cC^{\infty}(\RR^2,\RR)$ be a smooth test function compactly supported in $\Omega$. 
    We apply the discrete Green formula Equation \eqref{eq:discrete:green:formula} on $\Primal$ with the two functions $\dmGreen(x_1,\cdot)$ and $f$ to obtain
	\begin{equation}
		f(x_1) + \sum_{x_2 \in \Primal \cap \supp f}\dmGreen(x_1,x_2)\left(m(x_2)-\dDelta f(x_2)\right) = 0.
	\end{equation}
    Recall that $m(x_2) = \eps^2 \weight(x_2) \mass(x_2) + O(\eps^3)$ and note that the discrete Laplacian approximates the continuum Laplacian as $\eps \to 0$, more precisely $$\dDelta f(x_2) = \eps^2 \weight(x_2) \Delta f(x_2)+O(\eps^3)$$
    (see Lemma 2.2 of~\cite{ChelkakSmirnov}).
    Moreover, by the law of sines, $\weight(x_2)$ is half the sum of the areas of the rhombi around $x_2$.
    Hence the sum is a Riemann sum for the function $g(\chi_1,\cdot)(\mass-\Delta f)$ on $\Omega$.
    Since $g(\chi_1,\cdot)$ has a singularity at $x_1$, we must control the singularity to obtain convergence of the Riemann sum towards an integral.
    Since 
    \begin{equation}\label{eq:G:bound:log}
        |\dmGreen(x_1,x_2)| \leq |\dGreen(x_1,x_2)| = O(1+\log|x_1-x_2|)
    \end{equation}
 and since $\log$ is integrable around the origin in two dimensions (or rather, using crude bounds for the sum below),
we deduce that for all $s >0$ there exists $c(s) >0$ independent of $\eps$ such that
    $$
        \left|\sum_{x_2 \in \Primal \cap \supp f \cap B(x_1,s)} \dmGreen(x_1,x_2)\left(m(x_2)-\dDelta f(x_2)\right)\right| \leq c(s)
    $$
    with 
    $$
        c(s) \overset{s \to 0}{\longrightarrow} 0.
    $$
         This implies convergence of the Riemann sums. Taking the limit along the subsequence considered previously, we obtain
	\begin{equation}
		f + \int_{\Omega} g(\chi_1,\cdot)(\mass-\Delta f) = 0.
	\end{equation} 
   In other words, $g(\chi_1, \cdot)$ is a weak solution of the equation
    \begin{equation}\label{eq:distribution}
        \Delta g = \mass g + \delta_{x}
    \end{equation}
    with zero boundary condition on $\partial \Omega$. 
    By the uniqueness in Definition~\ref{def:mGreen} (recall the elliptic regularity arguments such as \cite[Chapter 10.1, (5)]{LiebLoss}), this implies that $g(\chi_1,\cdot) = \mGreen(\chi_1,\cdot)$, which concludes the proof.
\end{proof}

\section{Scaling limit of inverse Kasteleyn matrix and height function}
In this section, we establish the convergence of $\dpartialtwo \dmGreen$ and $\dpartialone \dpartialtwo \dmGreen$ and identify the limit, where we write $\dpartialone$ and $\dpartialtwo$ for the discrete derivatives with respect to the first or second variable respectively. In the non-massive case $m = 0$, these results were obtained by Li in \cite{Li17}. Note that an alternative proof of our convergence statements could be obtained by plugging her results in the resolvent identity. However, we preferred to develop a discrete holomorphicity theory for discrete massive differentials, which is interesting in itself and provides a more robust and self-contained proof.

\subsection{Convergence of first derivative of discrete massive Green function}\label{subsec:cv:green:first:derivative}
We state a lemma to prove convergence of the first order discrete derivative of the discrete massive Green function towards its continuous analog.  Before stating the lemma, we need to introduce some notation. 
We identify a complex number $z = a+\ii b$ with the vector $(a,b)$ of $\RR^2$, and denote by $\langle \cdot; \cdot\rangle$ the scalar product on $\RR^2$, that is, 
$$
	\langle z_1, z_2 \rangle = \Re(z_1 \overline{z_2}) .
$$

\begin{lemma}\label{lem:cv:green:derivative}
    Let $x_1 = x_1(\eps) \in \Primal$ and $w_2 = w_2(\eps) \in W$ corresponding to the middle of the edge $x_2^-x_2^+$. Uniformly for $x_1,w_2$ in compact subsets of $\Omega^2 \setminus \diagonal(\Omega^2)$,
        $$
            \begin{aligned}
                \daveragetwo \dpartialtwo \dmGreen(x_1,x_2^+) &= \frac{1}{4}\overline{\nabla_2} \mGreen(x_1,w_2) + o_{\eps \to 0}(1)\\
	        (x_2^+-x_2^-)\dpartialtwo \dmGreen(x_1,w_2) &= \frac{1}{2}\langle\nabla_2 \mGreen(x_1,w_2),x_2^+-x_2^-\rangle+o_{\eps \to 0}(\eps)
            \end{aligned}
	$$
    where $\nabla_2$ denotes the (ordinary) gradient in the second variable of the real-valued function $G^m(x_1, w_2)$; equivalently $\nabla_2 G^m (x_1, w_2) = \bar \partial_2 ( G^m (x_1, w_2))$ when it is viewed as a complex-valued function. 

    Uniformly for $x_1,w_2$ in compact subsets of $\Omega$,
    \begin{equation}\label{eq:point3_L31}
        \dpartialtwo \dmGreen(x_1,w_2) = O\left(\frac{1}{|x_1-w_2|}\right).
    \end{equation}
    When $\Omega$ and $\Primal$ have a straight boundary $\partial\Omega \cap B(\beta^*,\eta) = L$, $\Omega^2 \setminus \diagonal(\Omega^2)$ can be replaced by $\Omega_\ph^2 \setminus \{(x_1, x_2) \in \Omega_\ph^2 \colon x_1 = x_2 \text{ or } x_1 = \ph(x_2)\}$ in the first statement and $\Omega$ can be replaced by $\Omega_\ph$ in the second statement if we replace $|x_1-w_2|$ by $|x_1-w_2| \wedge |\ph(x_1)-w_2|$.
\end{lemma}
\begin{proof}
    We start by proving the first point of the lemma. Define $F(x_1,w_2) := [\dpartialtwo \dmGreen](x_1,w_2)$ and $\daveragetwo F(x_1,x_2) := [\daveragetwo\dpartialtwo \dmGreen](x_1,w_2)$. The proof follows two steps:
    \begin{itemize}
        \item Step 1: we prove that $\daveragetwo\dpartialtwo \dmGreen$ is a sequence of equi-continuous functions of two variables on compact subsets of $\Omega^2 \setminus \diagonal(\Omega^2)$. 
        \item Step 2: we show that any sub-sequential limit of $\daveragetwo\dpartialtwo \dmGreen$ is $\frac{1}{4}\overline{\nabla_2}\mGreen$
    \end{itemize} 
    \emph{First step.} 
    The main idea is to combine Lemma~\ref{lem:lipschitz} and Proposition~\ref{prop:discrete:gradient} with the bound Lemma~\ref{lem:green:bounded} to prove that $\daveragetwo F$ is Lipschitz in the first variable, $\beta$-Holder in the second variable for any $\beta < 1$ with constants independent of $\eps$. 
    Then, we will be able to use the Arzela--Ascoli theorem.

    We now give the details. 
    Let $\cC_1 \subset \Omega^2 \setminus \diagonal(\Omega^2)$ be a compact subset. Here and below, if $\cC\subset \Omega$ is a set, its interior will be denoted by $\mathring{\cC}$.
    By Lemma~\ref{lem:green:bounded},
    \begin{equation}\label{eq:bound:H}
        \sup_{\eps > 0, (x_1,x_2) \in \cC_1} |\dmGreen(x_1,x_2)| < \infty.
    \end{equation}
    Let $\cC_2 \subset \mathring{\cC_1}$ be a compact subset. 
    The above equation and Lemma~\ref{lem:lipschitz} imply that 
    \begin{equation}\label{eq:bound:F}
        \sup_{\eps >0, (x_1,w_2) \in \cC_2}|F(x_1,w_2)| < \infty.
    \end{equation}
    Since $\daveragetwo F(x_1,x_2)$ is by definition a linear combination with bounded coefficients of the $F(x_1,w_2)$ with $x_2 \sim w_2 \in W$, this implies that 
    \begin{equation}\label{eq:bound:AdG}
        \sup_{\eps >0, (x_1,x_2) \in \cC_2} |\daveragetwo F(x_1,x_2)|< \infty.
    \end{equation}
    We are ready to apply Proposition \ref{prop:discrete:gradient}. 
    Let $\cC_3 \subset \mathring{\cC_2}$ be a compact subset.
    Let $d = \dist(\cC_3, \cC_2^c)$. 
    For $x_1, w_2 \in \cC_3$, by the first point of Proposition~\ref{prop:discrete:gradient} applied in $B_\#(w_2,r)$ with $r = R = d/2$,
    \begin{equation}\label{eq:AF:first:point}
        \sup_{\eps>0, (x_1, w_2) \in \cC_3} \frac{|F(x_1,w_2)-\Proj[2[\daveragetwo F](x_1,x_2^+); \overline{x_2^+-x_2^-}]|}{{\eps(1+|\log(\eps)|)}} < \infty.
    \end{equation}
    This equation will only be useful at the very end of the proof. 
    The second point of Proposition~\ref{prop:discrete:gradient} implies that 
    $$
        \sup_{\eps > 0, (x_1, w_2) \in \cC_3, x_2, x_2' \in B(w_2, d/2)}
     \frac{|[\daveragetwo F](x_1,x_2)-[\daveragetwo F](x_1,x_2')|}{|x_2-x_2'|(1+|\log(|x_2-x_2'|)|)}< \infty.
    $$
    This implies that
\begin{equation}\label{eq:continuous:second:variable}
        \sup_{\eps > 0, (x_1, x_2), (x_1,x_2') \in \cC_3} \frac{|[\daveragetwo F](x_1,x_2)-[\daveragetwo F](x_1,x_2')|}{|x_2-x_2'|(1+|\log(|x_2-x_2'|)|)} < \infty.
    \end{equation}
    Indeed, if $|x_2-x_2'| \leq d/4$, we can apply the bound above with any $w_2 \sim x_2$, and if $|x_2-x_2'| \geq d/4$ we can simply use the bound Equation~\eqref{eq:bound:AdG}. 
    On the other hand, $\daveragetwo F(x_1, x_2)$ is a discrete massive harmonic function of the \emph{first} variable $x_1$ (or rather, its real and imaginary parts are real bounded discrete massive harmonic, as a linear combination of such functions). It is furthermore uniformly bounded for $(x_1,x_2) \in \cC_2$ by Equation~\eqref{eq:bound:AdG}, hence Lemma~\ref{lem:lipschitz} implies that
    \begin{equation}\label{eq:continuous:first:variable}
        \sup_{\eps > 0, (x_1,x_2), (x_1',x_2) \in \cC_3} \frac{|\daveragetwo F(x_1,x_2) - \daveragetwo F(x_1',x_2)|}{|x_1-x_1'|} < \infty.
    \end{equation}
    We are ready to conclude the first step. 
    Let $\cC_4 \subset \mathring{\cC_3}$ be a compact subset. 
    Combining Equations~\eqref{eq:bound:AdG}, \eqref{eq:continuous:second:variable} and~\eqref{eq:continuous:first:variable} with Lemma~\ref{lem:lip:two:var}, we obtain
    $$
        \sup_{\eps > 0, (x_1,x_2),(x_1',x_2') \in \cC_4}\frac{|\daveragetwo F(x_1,x_2) - \daveragetwo F(x_1',x_2')|}{|x_2-x_2'|(1+|\log(|x_2-x_2'|)|) + |x_1-x_1'|} < \infty. 
    $$
    Hence the sequence $\daveragetwo F$ is equi-continuous on $\cC_4$.
    Note that since we made no assumption on the sequence $\cC_i$ except that $\cC_1 \subset \Omega^2 \setminus \diagonal(\Omega^2)$, $\cC_{i+1} \subset \mathring{\cC_{i-1}}$, we proved that this holds for any compact subset $\cC_4 \subset \Omega^2 \setminus \diagonal(\Omega^2)$.
    By the Arzela--Ascoli theorem, upon taking a subsequence, $\daveragetwo F \to D \in \cC^0(\Omega^2 \setminus \diagonal(\Omega^2), \CC)$ uniformly on compact subsets of $\Omega^2 \setminus \diagonal(\Omega^2)$.
    If $\Omega$ and $\Primal$ have a straight boundary $L$, the modification is straightforward, replacing Equation~\eqref{eq:bound:H} by its improvement from Lemma~\ref{lem:green:bounded}.

    \bigskip

    \emph{Second step.} 
    To conclude the proof, it remains to identify the limit $D$.
    Let $x_1 = x_1(\eps) \to \chi_1 \in \Omega$, and denote by $H(\cdot) = \dmGreen(x_1, \cdot)$ and $F(\cdot) = F(x_1, \cdot)$. 
    Let $\chi, \chi' \in \Omega \setminus \{ \chi_1\}$.
    Let $\path$ be a smooth path between $\chi$ and $\chi'$ in $\Omega \setminus \{\chi_1\}$ of length $|\gamma|< \infty$. 
    Let $\cC \subset \Omega^2 \setminus \diagonal(\Omega^2)$ be a compact subset large enough such that $(x_1, \path(t)) \in \cC$ for all $t$.
    Let $x, x' \in \infWhite$ such that $x = x(\eps) \to \chi, x' = x'(\eps) \to \chi'$. 
    Let $\dpath = w_0,w_1,\dots,w_n$ be a path in $\infWhite$ from $x$ to $x'$, that is if we denote by $x_i^-x_i^+$ the edge of $\Primal$ corresponding to $w_i$, $x_0^- = x, x_n^+ = x'$ and $x_i^+ = x_{i+1}^-$ for all $0\leq i \leq n-1$. 
    We assume further that the length of $\dpath$ is $n \leq C|\path|$, and that $\dpath$ remains at distance $O(\eps)$ from $\path$.
    (this is always possible under the bounded angle assumption). 
    Then, using the fact that $\zeta \Proj(z, \bar \zeta) = \Re (z \zeta) = \Re ( \bar z \bar \zeta) = \langle \bar z; \zeta\rangle$,
    $$
        \begin{aligned}
            H(x')-H(x) = \sum_{i=0}^{n-1}H(x_i^+)-H(x_i^-)
            &= 2\sum_{i=0}^{n-1}(x_i^+ - x_i^-)F(w_i)\\
            &\overset{\eqref{eq:AF:first:point}}{=} 2 \sum_{i=0}^{n-1} (x_i^+ - x_i^-) \left(\Proj[2[\daverage F](x_i);\overline{x_i^+-x_i^-}]+O\left(\eps (1+|\log\eps|)\right)\right)\\
            &=  \sum_{i=0}^{n-1} 2 \left\langle \overline{[\daverage F](x_i)}; x_i^+-x_i^-\right\rangle +O\left(\eps (1+|\log\eps|)\right)
        \end{aligned}
    $$
    where the constant in the $O$ depends only $\cC$ and $|\path|$. 
    Since $\daverage F \to D(\chi_1,\cdot)$ uniformly on compact subsets of $\Omega \setminus \{\chi_1\}$ and $H \to \mGreen(\chi_1,\cdot)$ pointwise by Lemma~\ref{lem:cv:green}, we obtain
    $$
        \mGreen(\chi_1,\chi')-\mGreen(\chi_1,\chi) = 4\int_{\path} \langle \overline{D(\chi_1,\omega)};\mathrm{d}\omega\rangle.
    $$
    In the last integral, $\int_{\path} \langle \overline{D(\chi_1,\omega)};\mathrm{d}\omega\rangle$ stands for $\int_0^1 \langle \overline{D(\chi_1,\gamma(t))}; \gamma'(t) \rangle \mathrm{d}t$. 
    
     Since this holds for all $\chi, \chi' \in \Omega \setminus \{\chi_1\}$ and all paths between them, this identifies the projections of $D(\chi_1, \chi)$ in all directions, and we thus obtain that for $(\chi_1, \chi_2) \in \Omega^2 \setminus \diagonal(\Omega^2)$,
    $$
        D(\chi_1, \chi_2) = \frac{1}{4}\overline{\nabla_2} \mGreen(\chi_1,\chi_2),
    $$
    which proves the first point of the lemma. 
    The second point follows by Equation \eqref{eq:AF:first:point}.

    \bigskip 

    We now move towards a proof of \eqref{eq:point3_L31}.
    Let $\cC \subset \Omega$ be a compact subset ($\cC \subset \Omega_\ph$ if $\Omega$ and $\Primal$ have a straight boundary $L$), where as before we write $\mathring{\cC}$ for its interior.
    The resolvent identity of Proposition~\ref{prop:resolvent} can be written as
    $$
        \begin{aligned}
        \dmGreen(x_1,x_2) 
        &= \dGreen(x_1,x_2) + \sum_{x \in \Primal \cap \cC}m(x)\dmGreen(x_1,x) \dGreen(x,x_2) + \sum_{x \in \Primal \setminus \cC}m(x)\dmGreen(x_1,x) \dGreen(x,x_2)\\
        &= \dGreen(x_1,x_2) + \sum_{x \in \Primal \cap \cC}m(x)\dmGreen(x_1,x) \dGreen(x,x_2) + H(x_1,x_2),
        \end{aligned}
    $$
    where $H$ is defined by the last equation. 
    Observe that $H(x_1,x_2)$ is uniformly bounded for $x_1, x_2$ in any compact subset of $\mathring{\cC}$ by Lemma~\ref{lem:green:bounded}, hence so is $\dpartialtwo H$ by Lemma~\ref{lem:lipschitz}.
    Let $\cC' \subset \mathring{\cC}$ be a compact subset.
    Differentiating  with respect to the second variable and using Lemmas~\ref{lem:dmGreen<dGreen},~\ref{lem:green:bounded} and Lemma~\ref{lem:bound:nearby:derivative}, we obtain
    \begin{equation}\label{eq:differentiate:resolvent}
    \begin{aligned}
        \dpartialtwo \dmGreen(x_1,w_2) 
        &= \dpartialtwo \dGreen(x_1,w_2) + \sum_{x \in \Primal \cap \cC}m(x)\dmGreen(x_1,x) \dpartialtwo \dGreen(x,w_2)+\dpartialtwo H(x_1,w_2)\\
        &= O\left(\frac{1}{|x_1-w_2|}\right) + \sum_{x \in \Primal\cap \cC}m(x)O\left(\frac{1+|\log|x_1-x||}{|x-w_2|}\right)+O(1)\\
    \end{aligned}
    \end{equation}
    where the $O$ is uniform for $x_1, w_2 \in \cC' \cap \Omega$. 
    Note that when $\Omega$ and $\Primal$ have a straight boundary $L$ and when $x_1, x, w_2 \in \Primal$, $|\ph(x_1)-x| = O(|x_1-x|)$ and $|\ph(x)-w_2| =O(|x-w_2|)$ by Remark~\ref{rmk:twice:the:size} so this equation also holds in this case.
    Let $2d = |x_1-w_2|$. 
    We split the sum in two:
    $$
        S_1 
        = \sum_{|y-x_1| \leq d}m(y)\frac{1+|\log|x_1-y||}{|y-w_2|}
        \leq \sum_{|y-x_1| \leq d}m(y)\frac{1+|\log|x_1-y||}{d} \leq \frac{C_1 \cM}{d},
    $$
    because $\log$ is integrable in two dimensions. On the other hand,
    $$
        \begin{aligned}
        S_2 
        = \sum_{|y-x_1| \geq d}m(y)\frac{1+|\log|x_1-y||}{|y-w_2|}
        &\leq \sum_{ |y - x_1| \ge d }m(y)\frac{1+|\log(|y-w_2|/3)|}{|y-w_2|}\\
        &\leq \sum_{y \in \Omega \cap \cC}m(y)\frac{1+|\log(|y-w_2|/3)|}{|y-w_2|}
        \leq C_2 \cM
        \end{aligned}
    $$
    because if $|y-x_1| \geq d, |y-w_2| \geq d$,  then $|y-w_2| \leq |y-x_1|+2d \leq 3|y-x_1|$ and because $z \to (\log|z|)/|z-w_2|$ is integrable in two dimensions.
    This implies that 
    $$
        \dpartialtwo \dmGreen(x_1,w_2) = O\left(\frac{1}{|x_1-w_2|}\right),
    $$
    uniformly for $x_1, w_2 \in \cC' \cap \Omega$. 
    
    When $\Omega$ and $\Primal$ have a straight boundary $L$, since $\dpartialtwo\dmGreen(x_1,w_2)$ is extended to $\cC'$ by the discrete reflection principle, this gives
    $$
        \dpartialtwo \dmGreen(x_1,w_2) = O\left(\frac{1}{|x_1-w_2| \wedge |\ph(x_1)-w_2|}\right).
    $$
    uniformly for $x_1, w_2$ in $\cC'$, which concludes the proof.
\end{proof}

\subsection{Convergence of second order discrete derivative of massive Green function}\label{subsec:cv:green:second:derivative}
We state a lemma to prove convergence of the second order discrete derivative of the discrete massive Green function towards its continuous analog. Before stating the lemma, we need to introduce some more matrix notation. 
Recall that we identify a complex number $z = a+\ii b$ with the vector $(a,b)$ of $\RR^2$. We will denote by $\cdot$ the matrix product (it should be clear from context whether we are talking about multiplication of matrices or multiplication of complex numbers). In particular, the scalar product on $\RR^2$ is 
$$
	\langle z_1, z_2 \rangle = \Re(z_1 \overline{z_2}) = (z_1)^{\intercal} \cdot z_2.
$$

Recall also the (conjugate) discrete gradients from the previous section, namely  $\overline{\nabla_1}\mGreen$, and $\overline{\nabla_2}\mGreen$, that is for $i \in \{1,2\}$,
$$
    \overline{\nabla_i}\mGreen = \begin{pmatrix}\partial_{i,x}\mGreen \\ - \partial_{i,y}\mGreen\end{pmatrix}
$$
where the variables of $\mGreen$ are implicit in the notation, and where in $\partial_{i,z}$ the number $i \in \{1,2\}$ stands for the variable which is derivated and the letter $z \in \{x,y\}$ stands for the direction of the partial derivative. We also define 
	$$
        (\nabla_1) (\nabla_2)^{\intercal} \mGreen = 
        \begin{pmatrix}
            \partial_{1,x}\partial_{2,x}\mGreen & \partial_{1,x}\partial_{2,y}\mGreen \\
            \partial_{1,y}\partial_{2,x}\mGreen & \partial_{1,y}\partial_{2,y}\mGreen,
           \end{pmatrix} 
           \quad ; \quad 
            (\overline{\nabla_1}) (\overline{\nabla_2})^{\intercal}\mGreen = \begin{pmatrix}
            \partial_{1,x}\partial_{2,x}\mGreen & -\partial_{1,x}\partial_{2,y}\mGreen \\
            -\partial_{1,y}\partial_{2,x}\mGreen & \partial_{1,y}\partial_{2,y}\mGreen
        \end{pmatrix},
   	$$
    and
    $$
        (\nabla_2) (\nabla_1)^{\intercal} \mGreen = \left((\nabla_1) (\nabla_2)^{\intercal} \mGreen)\right)^{\intercal} = 
        \begin{pmatrix}
        \partial_{1,x}\partial_{2,x}\mGreen & \partial_{1,y}\partial_{2,x}\mGreen\\
            \partial_{1,x}\partial_{2,y}\mGreen & \partial_{1,y}\partial_{2,y}\mGreen
        \end{pmatrix}
    $$

\begin{lemma}\label{lem:cv:second:derivative}
    Let $w_1 = w_1(\eps), w_2 = w_2(\eps) \in W$ corresponding respectively to the middle of the primal edges $x_1^-x_1^+$, $x_2^-x_2^+$. Uniformly for $w_1,w_2$ in compact subsets of $\Omega^2 \setminus \diagonal(\Omega^2)$, 
	$$
        \begin{aligned}
	     \daverageone \dpartialone \daveragetwo \dpartialtwo \dmGreen(x_1^+,x_2^+) &= \frac{1}{16}(\overline{\nabla_1}) (\overline{\nabla_2})^{\intercal} \mGreen(w_1, w_2) + o_{\eps \to 0}(1)\\
        \left(\dpartialone \dpartialtwo \dmGreen(w_1,w_2)\right)(x_1^+-x_1^-)(x_2^+-x_2^-) &= \frac{1}{4}(x_1^+-x_1^-)^{\intercal} \cdot \nabla_1(\nabla_2)^{\intercal} \mGreen(w_1,w_2) \cdot (x_2^+-x_2^-) \\
        & \quad \quad + o_{\eps \to 0}(\eps^2).\\
	\end{aligned}
    $$
    The first equation means that the gradients of the real and imaginary parts of $\daveragetwo \dpartialtwo \dmGreen$ converge. More precisely, if we write the complex-valued function $\daveragetwo \dpartialtwo \dmGreen$ as $ U+iV$ (where $U,V$ are real-valued), then the complex-valued function $\daverageone \dpartialone U$ converges to the first column of the matrix on the right hand side (after identifying complex numbers with vectors, as usual). Likewise, $\daverageone \dpartialone V$ converges to the second column of that matrix. 
    
    Moreover,
    $$
       \dpartialone \dpartialtwo \dmGreen(w_1,w_2) = O\left(\frac{1}{|w_1-w_2|^2}\right) 
    $$
    uniformly for $w_1,w_2$ in compact subsets of $\Omega^2$.
    When $\Omega$ and $\Primal$ have a straight boundary $\Omega \cap B(\beta^*,\eta) = L$, $\Omega^2 \setminus \cD(\Omega^2)$ can be replaced by $\Omega_\ph^2 \setminus \{(x_1, x_2) \in \Omega_\ph^2 \colon x_1 = x_2 \text{ or } x_1 = \ph(x_2)\}$ in the first statement and $\Omega$ can be replaced by $\Omega_\ph$ in the second statement if we replace $|w_1-w_2|^2$ by $|w_1-w_2|^2 \wedge |\ph(w_1)-w_2|^2$.
\end{lemma}
\begin{proof}
    We start with the first point of the lemma. The proof idea is roughly the same as for Lemma~\ref{lem:cv:green:derivative}. We first prove using Lemmas~\ref{lem:lipschitz} and \ref{prop:discrete:gradient} that $\daverageone \dpartialone \daveragetwo \dpartialtwo \dmGreen$  is a sequence of equi-continuous functions in both variables and then identify uniquely the limit by the same argument as before since $\daveragetwo \dpartialtwo \dmGreen$ converges by Lemma~\ref{lem:cv:green:derivative}. 

    However, some new technical difficulties arise along the way, we highlight some of them:
    \begin{itemize}
        \item $\daveragetwo \dpartialtwo \dmGreen$ is a complex valued function, whose real part and imaginary parts are real discrete massive harmonic function. Since we did not develop the holomorphicity theory for complex-valued functions, we must treat them separately.
        \item We need to use some new symmetry argument to show that $\daveragetwo \dpartialtwo \daverageone \dpartialone \dmGreen$ is an equi-continuous sequence of functions of the second variable, since it is no more discrete massive harmonic.
        
        \item We know from the second point of Proposition~\ref{prop:discrete:gradient} that $\dpartialtwo \dmGreen$ is approximately the projection of $\daveragetwo\dpartialtwo \dmGreen$ on $x_2^+-x_2^-$, and that $\dpartialone \daveragetwo \dpartialtwo \dmGreen$ is approximately the projection of $\daverageone \dpartialone \daveragetwo \dpartialtwo \dmGreen$ on $x_1^+-x_1^-$. This does not imply that we can approximate $\dpartialone \dpartialtwo \dmGreen$ by a projection of $\daverageone \dpartialone \daveragetwo \dpartialtwo \dmGreen$.   
    \end{itemize}
    The main arguments to solve these problems are based on the following simple observations. Let $\ReIm_1$, $\ReIm_2$ denote either the real part $\Re$ or the imaginary part $\Im$.
    \begin{align}
            \dpartialone \dpartialtwo \dmGreen(w_1,w_2) &= \dpartialtwo \dpartialone \dmGreen(w_1,w_2) \nonumber\\
        \dpartialone \daveragetwo \dpartialtwo \dmGreen(w_1,x_2)
        & = \daveragetwo \dpartialtwo \dpartialone \dmGreen(w_1,x_2) \label{eq:commute:bis}\\
        \ReIm_1 \daverageone \dpartialone \ReIm_2 \daveragetwo \dpartialtwo \dmGreen(x_1,x_2)
        & = \ReIm_2 \daveragetwo \dpartialtwo \ReIm_1 \daverageone \dpartialone \dmGreen (x_1, x_2) \label{eq:commute}.
    \end{align}
    These observations can all be checked by direct computation: for example the last one is equal to
    $$
        \sum_{w_1 \sim x_1, w_2 \sim x_2}\ReIm_1\left(\frac{\mu_W(w_1)}{x_1^+-x_1}\right)\ReIm_2\left(\frac{\mu_W(w_2)}{x_2^+-x_2}\right)\frac{\dmGreen(x_1^+,x_2^+)-\dmGreen(x_1,x_2^+)-\dmGreen(x_1^+,x_2)+\dmGreen(x_1,x_2)}{16\mu_B(x_1)\mu_B(x_2)}.
    $$
    Let $\cC_3 \subset \Omega^2 \setminus \cD(\Omega^2)$ be a compact subset.
    Recall from Equation~\eqref{eq:bound:AdG} that 
    \begin{equation}\label{eq:bound:PAdG}
        \sup_{\eps > 0, (x_1, x_2) \in \cC_3} \ReIm_2 \daveragetwo \dpartialtwo \dmGreen(x_1,x_2) < \infty.
    \end{equation} 
    Let $\cC_4 \subset \mathring{\cC_3}$ be a compact subset.
    By Lemma~\ref{lem:lipschitz}, 
    \begin{equation}\label{eq:bound:dPAdG}
        \sup_{\eps > 0, x_1, x_2 \in \cC_4}  |\dpartialone \ReIm_2 \daveragetwo \dpartialtwo \dmGreen(x_1,x_2)| < \infty.
    \end{equation}
    Let $\cC_5 \subset \mathring{\cC_4}$ be a compact subset.
    Applying the first point of Proposition~\ref{prop:discrete:gradient} as in the proof of Lemma~\ref{lem:cv:green:derivative}, we obtain that
    \begin{equation}\label{eq:bound:dPAdG:first:var}
        \begin{aligned}
        \sup_{\underset{\eps > 0}{(x_1,x_2), (x_1',x_2) \in \cC_5}} &\frac{|\ReIm_1 \daverageone \dpartialone \ReIm_2 \daveragetwo \dpartialtwo \dmGreen(x_1,x_2)-\ReIm_1 \daverageone \dpartialone \ReIm_2 \daveragetwo \dpartialtwo \dmGreen(x_1',x_2)|}
        {|x_1-x_1'|(1+|\log(|x_1-x_1'|)|)}\\
        &< \infty.
        \end{aligned}
    \end{equation}
    Since $\dmGreen$ is symmetric in the two variables, and $\ReIm_1 \daverageone \dpartialone$ and $\ReIm_2 \daveragetwo \dpartialtwo$ commute by Equation~\eqref{eq:commute} and $\ReIm_1, \ReIm_2$ were chosen arbitrarily, we can interchange the variables:
    \begin{equation}\label{eq:bound:dPAdG:second:var}
        \begin{aligned}
        \sup_{\underset{\eps > 0}{(x_1,x_2), (x_1,x_2') \in \cC_5}} 
        &\frac
        {|\ReIm_1 \daverageone \dpartialone \ReIm_2 \daveragetwo \dpartialtwo \dmGreen(x_1,x_2)-\ReIm_1 \daverageone \dpartialone \ReIm_2 \daveragetwo \dpartialtwo \dmGreen(x_1,x_2')|}
        {|x_2-x_2'|(1+|\log(|x_2-x_2'|)|)}\\
        &< \infty.
        \end{aligned}
    \end{equation}
    Combining Equations~\eqref{eq:bound:dPAdG}, \eqref{eq:bound:dPAdG:first:var} and~\eqref{eq:bound:dPAdG:second:var} with Lemma~\ref{lem:lip:two:var}, this implies that if $\cC_6 \subset \mathring{\cC_5}$ is a compact subset,
    $$
        \begin{aligned}
        \sup_{\underset{\eps > 0}{(x_1,x_2), (x_1',x_2') \in \cC_6}}
         &\frac
        {|\ReIm_1 \daverageone \dpartialone \ReIm_2 \daveragetwo \dpartialtwo \dmGreen(x_1,x_2)-\ReIm_1 \daverageone \dpartialone \ReIm_2 \daveragetwo \dpartialtwo \dmGreen(x_1',x_2')|}
        {|x_1-x_1'|(1+|\log(|x_1-x_1'|)|)+|x_2-x_2'|(1+|\log(|x_2-x_2'|)|)}\\
        &< \infty.
        \end{aligned}
    $$
    Hence, for $\ReIm \in \{\Re, \Im\}$, the sequence $\daverageone \dpartialone \ReIm \daveragetwo \dpartialtwo \dmGreen$ is equi-continuous on $\cC_6$, and this holds for any compact subset $\cC_6 \subset \Omega^2 \setminus \cD(\Omega^2)$.
    By the Arzela--Ascoli theorem, upon taking a subsequence, $\daverageone \dpartialone \ReIm \daveragetwo \dpartialtwo \dmGreen \to \cD_{\ReIm}$ uniformly on compact subsets of $\Omega^2 \setminus \diagonal(\Omega^2)$. 
    Moreover, since $\ReIm \daveragetwo \dpartialtwo \dmGreen$ and $\dpartialone \ReIm \daveragetwo \dpartialtwo \dmGreen$ are bounded on $\cC_4$ by Equations~\eqref{eq:bound:PAdG} and~\eqref{eq:bound:dPAdG}, by the first point of Proposition~\ref{prop:discrete:gradient}
    \begin{equation}\label{eq:d:from:average}
        \begin{aligned}
        \sup_{\underset{\eps > 0}{(w_1,x_2) \in \cC_5}}
        &\frac
        {\left|\dpartialone \ReIm \daveragetwo \dpartialtwo \dmGreen(w_1,x_2) -\Proj\left[2\daverageone \dpartialone \ReIm \daveragetwo \dpartialtwo \dmGreen(x_1^+,x_2);\overline{x_1^+-x_1^-}\right]\right|}
        {\eps (1+|\log(\eps)|)}\\
        &< \infty.
        \end{aligned}
    \end{equation}
    Since by Lemma~\ref{lem:cv:green:derivative}, 
    $$
        \ReIm \daveragetwo \dpartialtwo \dmGreen(w_1,x_2) = \frac{1}{4}\ReIm \overline{\nabla_2}\mGreen(x_1,w_2) + o_{\eps \to 0}(1),
    $$
    the argument of the second step of the proof of Lemma~\ref{lem:cv:green:derivative} applies and we obtain that $\cD_{\ReIm} = \frac{1}{16}\overline{\nabla_1} \ReIm \overline{\nabla_2} \mGreen$. 
    This concludes the proof of the first point.
    When $\Omega$ and $\Primal$ have a straight boundary $L$, the modification is straightforward by replacing Equation~\eqref{eq:bound:PAdG} by its modification, see the proof of the first point of Lemma~\ref{lem:cv:green:derivative}.

    \bigskip

    To obtain the second point, we first note that by linearity of $\dpartialone$, Equation~\eqref{eq:commute:bis}, for all $w_1,x_2$,
    \begin{equation}\label{eq:interchange}
        \begin{aligned}
        (x_1^+-x_1^-)\daveragetwo \dpartialtwo \dpartialone  \dmGreen(w_1,x_2)&= (x_1^+-x_1^-)
        \dpartialone \daveragetwo \dpartialtwo \dmGreen(w_1,x_2) \\
        &= 
        \begin{pmatrix}
            (x_1^+-x_1^-)\dpartialone \Re \daveragetwo \dpartialtwo \dmGreen(w_1,x_2) \\
            (x_1^+-x_1^-) \dpartialone \Im \daveragetwo \dpartialtwo \dmGreen(w_1,x_2).
        \end{pmatrix}
        \end{aligned}
    \end{equation}
    We use the vector form rather than the complex form for the right-hand side to emphasize that both components are real. 
    The right-hand side can be estimated by Equation~\eqref{eq:d:from:average}.
    For the left-hand side, we use once more Proposition~\ref{prop:discrete:gradient}. 
    As before, applying Lemma~\ref{lem:lipschitz} consecutively to the discrete real massive harmonic functions $\dmGreen(\cdot, x_2)$ and $(x_1^+-x_1^-)\dpartialone 
    \dmGreen (w_1, \cdot)$, we obtain that for all compact subset $\cC_1 \subset \Omega^2 \setminus \cD(\Omega^2)$, 
    $$
        \sup_{\eps > 0, (w_1,x_2) \in \cC} |\dpartialone \dmGreen(w_1, x_2)| < \infty \quad ; \quad \sup_{\eps > 0, (w_1,w_2) \in \cC} |\dpartialtwo \dpartialone \dmGreen(w_1, w_2)| < \infty.
    $$
    By the first point of Proposition~\ref{prop:discrete:gradient} applied to the \emph{real-valued} discrete massive harmonic function $(x_1^+-x_1^+)\dpartialone \dmGreen(w_1, \cdot)$, if $\cC_2 \subset \mathring{\cC_1}$ is a compact subset,
    \begin{equation}\label{eq:evaluate:interchange}
        \begin{aligned}
        \sup_{\eps > 0, (w_1,w_2) \in \cC_2}
        &
        \frac
        {1}
        {\eps(1+|\log(\eps)|)}
        \Bigg|(x_1^+-x_1^-)\dpartialtwo\dpartialone \dmGreen(w_1, x_2)\\
        & - \Proj\left[2(x_1^+-x_1^-)\daveragetwo\dpartialtwo \dpartialone \dmGreen(w_1, x_2^+); \overline{x_2^+-x_2^-}\right]\Bigg|
        < \infty.
        \end{aligned}
    \end{equation}
    To obtain the second point of the lemma from Equations~\eqref{eq:d:from:average}, \eqref{eq:interchange}, \eqref{eq:evaluate:interchange} and the first point of the lemma is only a matter of basic algebra. We do it below for completeness.

    Let $w_1,w_2 \in \cC_5$. 
    From Equation~\eqref{eq:d:from:average}, the first point of the lemma, and the fact that
    \begin{equation}\label{eq:proj}
        z_1\Proj[z_2,\overline{z_1}] = \Re(z_1z_2) = z_1^{\intercal} \cdot \overline{z_2} = \overline{z_2}^{\intercal} \cdot z_1 \in \RR,
    \end{equation}
    we obtain that for $\ReIm \in \{\Re,\Im\}$
    $$
        (x_1^+-x_1^-)\dpartialone \ReIm \daveragetwo \dpartialtwo \dmGreen(w_1,x_2) = \frac{1}{8}\nabla_1^{\intercal} \ReIm \overline{\nabla_2} \mGreen(w_1,w_2) \cdot (x_1^+-x_1^-) + o(\eps),
    $$
    uniformly for $w_1,w_2 \in \cC_5$. 
    Note that both sides of this equation are real. 
    Equation~\eqref{eq:interchange} then becomes
    $$
        \begin{aligned}
        (x_1^+-x_1^-)\daveragetwo \dpartialtwo \dpartialone  \dmGreen(w_1,x_2) 
        &= \frac{1}{8}\overline{\nabla_2}\nabla_1^{\intercal} \mGreen(w_1,w_2) \cdot (x_1^+-x_1^-)+o(\eps),
        \end{aligned}
    $$
    uniformly for $w_1,w_2 \in \cC_5$. 
    From Equations~\eqref{eq:evaluate:interchange} and \eqref{eq:proj}, we finally obtain that
    $$
        \begin{aligned}
        (x_2^+-x_2^-)(x_1^+-x_1^-)\dpartialtwo \dpartialone \dmGreen(w_1,w_2) &= \frac{1}{4} (x_2^+-x_2^-)^{\intercal}\overline{\overline{\nabla_2}\nabla_1^{\intercal}\mGreen(w_1,w_2)\cdot (x_1^+-x_1^-)} + o(\eps^2)\\
        &=  \frac{1}{4} (x_2^+-x_2^-)^{\intercal}\nabla_2\nabla_1^{\intercal}\mGreen(w_1,w_2)\cdot (x_1^+-x_1^-) + o(\eps^2)
        \end{aligned}
    $$
    uniformly for $w_1,w_2 \in \cC_5$, which concludes the proof.

    \bigskip

    We now tackle the last point of the lemma. Let $\cC \subset \Omega$ ($\Omega_\ph$ if $\Omega$ and $\Primal$ have a straight boundary $L$) be a compact subset. As in the proof of Lemma~\ref{lem:cv:green:derivative}, we separate the boundary terms in the resolvent identity and differentiate in each variable to obtain
    $$
        \begin{aligned}
        \dpartialone \dpartialtwo \dmGreen(w_1,w_2) 
        &= \dpartialone \dpartialtwo \dGreen(w_1,w_2) + \sum_{x \in \Primal \cap \cC}m(x)\dpartialone \dmGreen(w_1,x) \dpartialtwo \dGreen(x,w_2)\\
        &\quad + \dpartialone \dpartialtwo H(w_1,w_2).
        \end{aligned}
    $$
    Recall that $H(w_1,w_2)$ is discrete massive harmonic in the first variable, discrete (non-massive) harmonic in the second variable.
    Moreover, it is bounded uniformly for $x_1, x_2$ in compact subsets of $\mathring{\cC}$.
    Hence its derivative in one and two variables are also bounded uniformly in any compact subset of $\mathring{\cC}$ by Lemma~\ref{lem:lipschitz}. 
    Let $\cC' \subset \mathring{\cC}$ be a compact subset.
    By Lemmas~\ref{lem:cv:green:derivative}, \ref{lem:bound:nearby:derivative} and~\ref{lem:bound:nearby:second:derivative} (and by symmetry of the discrete Green function),
    $$
        \begin{aligned}
        \dpartialone \dpartialtwo \dmGreen(w_1,w_2) 
        &= O\left(\frac{1}{|w_1-w_2|^2}\right) + \sum_{y \in \Primal \cap \cC}m(y)O\left(\frac{1}{|x-w_1||x-w_2|}\right) + O(1)\\
        \end{aligned}
    $$
    uniformly for $w_1,w_2 \in \cC' \cap \Omega$.
    Recall that as in the preceding lemma, $|\ph(w_1)-w_2| = O(|w_1-w_2|)$ for $w_1, w_2 \in \Omega$ so this holds also when $\cC'$ is a compact subset of $\Omega \cup B(\beta^*,\eta)$.
    Again, we can denote $2d = |w_1-w_2|$ and split the sum in three: for $i \in \{1,2\}$,
    $$
        S_i 
        = \sum_{|y-w_i| \leq d}\frac{m(y)}{|x-w_1||x-w_2|}
        \leq \frac{1}{d}\sum_{|y-w_i| \leq d}\frac{m(y)}{|x-w_i|} \leq \frac{C_1\cM}{d}
    $$
    because $z \to 1/|z|$ is integrable in two dimensions, and
    $$
     S_3 = \sum_{|y-w_1| \geq d, |y-w_2| \geq d}\frac{m(y)}{|x-w_1||x-w_2|}
        \leq \frac{1}{d^2}\sum_{\Primal \setminus \cC}m(y) \leq \frac{C_3\cM}{d^2}.
    $$
    This implies that 
    $$
        \dpartialone\dpartialtwo \dmGreen(w_1,w_2) = O\left(\frac{1}{|w_1-w_2|^2}\right),
    $$
    uniformly for $w_1, w_2 \in \cC' \cap \Omega$. 
    When $\Omega$ and $\Primal$ have a straight boundary $L$, since $\dmGreen(x_1,x_2)$ and its derivatives are extended to $\Primal_\ph$ by the discrete Schwarz principle, this gives
    $$
        \dpartialone\dpartialtwo \dmGreen(w_1,w_2) = O\left(\frac{1}{|w_1-w_2|^2 \wedge |\ph(w_1)-w_2|^2}\right).
    $$
    This concludes the proof.
\end{proof}

\subsection{Convergence of the inverse Kasteleyn matrix (proof of Theorem \ref{thm:isoradial})}\label{subsec:proof:thm:b0}

This section is dedicated to the proof of Theorem \ref{thm:isoradial}. Roughly speaking, Lemma \ref{lem:block} relates the inverse Kasteleyn matrix $K^{-1} (x, w)$ to the derivatives of the massive Green function when $x \in \Gamma$ is a (primal) black vertex. The tools developed in the preceding subsections therefore give us the asymptotics of $K^{-1} $ in this case, which yields the first point of Theorem \ref{thm:isoradial} in this case (i.e., the first asymptotics of \eqref{eq:thm:invK}). 

The asymptotics of $K^{-1} (y, w)$ when $y$ is instead a black (dual) vertex is however more complicated. Let us first recall an important idea concerning the classical (non-massive) case, which goes back in some form to the seminal paper of Kenyon \cite{Kenyon_confinv}. In this classical case, the idea is that $K^{-1}$ satisfies $KK^{-1} = I$, so is discrete holomorphic. This means that the derivatives of $K^{-1}(\cdot, w)$ in the $y$ direction are equal to $\ii$ times those in the $x$ direction of a given rhombus. In combination with the asymptotic analysis of the (derivatives of) $K^{-1}$ in the $x$ direction this gives us access to the asymptotics of the derivatives of $K^{-1}$ in the $y$ direction. Summing these asymptotics along a path emanating from the boundary (more specifically from the ``special'' Temperleyan corner $b^*$) then gives an integral formula for the asymptotics of $K^{-1}(y,w)$. 

In our massive case, it is of course still the case that $KK^{-1} = I$. However this now means that $K^{-1}$ is discrete \emph{massive} holomorphic, so the relation between the $x$- and $y$-derivatives of $K^{-1}( \cdot, w)$ on a given rhombus is more complex. Ultimately, we need to find a suitable  transformation of $K^{-1}$ for which an exact discrete form of the massive Cauchy--Riemann equation \eqref{eq:CR:geometric} holds. The resulting key identities are stated in \eqref{eq:asymptotic:CR}, \eqref{eq:exponential} and 
\eqref{eq:exponential2}.

\begin{proof}
    Suppose that $w_2 \in \White$ corresponds to the middle of the rhombus $x_2^-y_2^-x_2^+y_2^+$ with angle $\theta_2$ and  with $K_{w_2,x_2^+}>0$, $K_{w_2,x_2^-}<0$ as in the statement of the theorem. 
    From Lemma \ref{lem:block}, since by definition $\dmGreen = (\dDirichletDelta-m I)^{-1}$, we can write
    $$
    	\begin{aligned}
        		K^{-1}(x_1,w_2) &= (\dmGreen K^*)(x_1,w_2) \\
		&= \dmGreen(x_1,x_2^+)\overline{K(w_2,x_2^+)}+G^m(x_1,x_2^-)\overline{K(w_2,x_2^-)}\\
		&= \sqrt{\tan(\theta_2)}\left[\exp\left(\frac{V(x_2^-)-V(x_2^+)}{2}\right)\dmGreen(x_1,x_2^+)-\exp\left(\frac{V(x_2^+)-V(x_2^-)}{2}\right)\dmGreen(x_1,x_2^-)\right]
	\end{aligned}
    $$
    by definition of $K$. For any $w \in \White$ corresponding to the middle of the edge $x^-x^+$ of $\Primal$, we can perform an asymptotic expansion
    $$
    	\begin{aligned}
    	\exp\left(\frac{V(x^+)-V(x^-)}{2}\right) 
	&= \exp\left(\frac12\left\langle\nabla V(w),x^+-x^-\right\rangle+O(\eps^3)\right)\\
	& = 1 + \frac12 \left\langle\nabla V(w) , x^+-x^-\right\rangle +  R(w, \eps) +O(\eps^3)
	\end{aligned}
    $$
    uniformly for $w \in W$, with $R(w,\eps) = \frac18\left\langle\nabla V(w),x^+-x^-\right\rangle^2$ depends only on $w \in W$ (and not on the choice of $x^-$ and $x^+$) and $R(w,\eps)=O(\eps^2)$ uniformly for $w \in W$. Hence we can write, uniformly for $(x_1,w_2)$ in compact subsets of $(\Omega \cup L)^2$,
    	\begin{align}
        K^{-1}(x_1,w_2) 
        &= \sqrt{\tan(\theta_2)}\left[\dmGreen(x_1,x_2^+)\left(1-\frac{1}{2} \langle\nabla V(w_2) , x_2^+-x_2^-\rangle + R(w_2,\eps)+O(\eps^3)\right) \right.\nonumber \\
        &\quad \left. - \dmGreen(x_1,x_2^-)\left(1+\frac{1}{2} \langle\nabla V(w_2) , x_2^+-x_2^-\rangle+R(w_2,\eps)+O(\eps^3)\right)\right]\nonumber \\
        &=  \sqrt{\tan(\theta_2)}\bigg[2\dpartialtwo \dmGreen(x_1,w_2)(x_2^+-x_2^-) \left(1+R(w_2,\eps) +O(\eps^3)\right)\nonumber \\
        &\quad - (\dmGreen(x_1,x_2^+)+\dmGreen(x_1,x_2^-))\left(\frac{1}{2} \langle\nabla V(w_2) , x_2^+-x_2^-\rangle +O(\eps^3)\right) \bigg]\nonumber \\
        &=  \sqrt{\tan(\theta_2)}\left[2\dpartialtwo \dmGreen(x_1,w_2)(x_2^+-x_2^-)\left(1-\frac12 \langle\nabla V(w_2) ,x_2^+-x_2^-\rangle + R(w_2,\eps)\right) \right.\nonumber \\
        &\quad - \dmGreen(x_1,x_2^-) \langle\nabla V(w_2) ,x_2^+-x_2^-\rangle \bigg]
         + O(\eps^3 \log \eps). \label{eq:CR:1}
                \end{align}
	where at the last line we use Lemma~\ref{lem:green:bounded} to bound $\dmGreen(x_1,w_2) = O(1+|\log|x_1-w_2||) = O(\log \eps))$. We are ready to prove the first point of the theorem. 
    The discrete massive Green function and its discrete derivative $\dpartialtwo \dmGreen(x_1,w_2)$ are uniformly bounded for $x_1,w_2$ in compact subsets of $\Omega\cup L)^2 \setminus \cD((\Omega \cup L)^2)$ by Lemmas \ref{lem:green:bounded} and~\ref{lem:bound:nearby:derivative}. 
    Hence, in Equation \eqref{eq:CR:1} if we consider only terms of order up to $O(\eps^2)$, we obtain
	\begin{equation}\label{eq:K^-1:order2}
			K^{-1}(x_1,w_2)
			= \sqrt{\tan(\theta_2)}\left[2(x_2^+-x_2^-)\dpartialtwo \dmGreen(x_1,w_2)- \dmGreen(x_1,x_2^-) \langle\nabla V(w_2) , x_2^+ - x_2^-\rangle)\right]+O(\eps^2)
    \end{equation}
    Using Lemmas \ref{lem:cv:green} and \ref{lem:cv:green:derivative}, we obtain
    $$
        K^{-1}(x_1,w_2) = \sqrt{\tan(\theta_2)}\left\langle\nabla_2 \mGreen(x_1,w_2) -\mGreen(x_1,w_2)\nabla V(w_2), x_2^+-x_2^-\right\rangle+o(\eps),
    $$
    uniformly for $(x_1,w_2)$ in compact subsets of $(\Omega \cup L)^2 \setminus \cD((\Omega \cup L)^2)$. Moreover, since $\Omega$ and $\Primal$ have a straight boundary, we can plug the bounds of Lemmas~\ref{lem:cv:green} and~\ref{lem:cv:green:derivative} in Equation~\eqref{eq:CR:1} to obtain that uniformly for $(x_1,w_2)$ in compact subsets of $(\Omega \cup L)^2$, 
    \begin{equation}\label{eq:discreteCR}
        K^{-1}(x_1,w_2) = O\left(\frac{\eps}{|x_1-w_2|}\right).
    \end{equation}

    We prove the second point of the theorem. For $w_1 \neq w_2$ corresponding to the middle of the rhombus $x_1^-y_1^-x_1^+y_1^+$ with angle $\theta_1$ and with $K_{w_1,x_1^+} = 1$, since by definition $KK^{-1} = I$, the \emph{discrete massive Cauchy--Riemann equation} at $w_1$ for the function $w_1 \to K^{-1}_{w_1,w_2}$ (which is, by definition, \emph{discrete massive meromorphic} in the sense of Remark~\ref{rem:def:discrete:massive:differential}) is
    $$
    	0 = (KK^{-1})_{w_1,w_2} = K_{w_1,x_1^-}K^{-1}_{x_1^-,w_2} + K_{w_1,x_1^+}K^{-1}_{x_1^+,w_2} + K_{w_1,y_1^-}K^{-1}_{y_1^-,w_2} + K_{w_1,y_1^+}K^{-1}_{y_1^+,w_2}.
    $$
    
    \begin{rmk}\label{rmk:temperleyan:corner}
        Note that, by definition of the Temperleyan corner $b^*$, this equation also holds for $w_1$ corresponding to the middle of an edge $y_1^-b^*$ with $y_1^- \in \Dual$ if we let $K^{-1}_{b^*,w_2} = 0$. Intuitively, Temperleyan boundary condition imposes Dirichlet boundary on $\Primal$, and Neumann boundary conditions on $\Dual$ except at $b^*$ where the boundary condition is also Dirichlet.
    \end{rmk}

    The natural strategy would be to use a compactness argument to obtain convergence of $K^{-1}(\cdot,w_2)$ and its derivatives on $\Primal$ and $\Dual$, and to take the limit in the discrete Cauchy--Riemann equation which would identify uniquely the limit on $\Dual$ as the unique massive holomorphic conjugate of the limit on $\Primal$ (which we already identified in the first point of the theorem). However, $K^{-1}(\cdot, w_2)$ is not discrete massive holomorphic on $\Dual$: actually, we do not even have a definition of discrete massive holomorphicity on $\Dual$, so we need to use a different argument.
 
 	For $w_1, w_2$ in $\White$, let
	$$
		\begin{aligned}
			[\dCRx K^{-1}](w_1,w_2) &= \sqrt{\tan(\theta_1)}\left( \exp\left(\frac{V(x_1^-)-V(x_1^+)}{2}\right)K^{-1}_{x_1^+,w_2}-\exp\left(\frac{V(x_1^+)-V(x_1^-)}{2}\right)K^{-1}_{x_1^-,w_2}\right)\\
			[\dCRy K^{-1}](w_1,w_2) &= \frac{1}{\sqrt{\tan(\theta_1)}}
            \bigg[\exp\left(V(y_1^+)-\frac{V(x_1^+)+V(x_1^-)}{2}\right)K^{-1}_{y_1^+,w_2}\\
            & \qquad \qquad \qquad -\exp\left(V(y_1^-)-\frac{V(x_1^+)+V(x_1^-)}{2}\right)K^{-1}_{y_1^-,w_2}\bigg].\\
		\end{aligned}
	$$
	Then, the discrete massive Cauchy--Riemann equation \eqref{eq:discreteCR} can be rewritten as 
	\begin{equation}\label{eq:discrete:CR}
		-[\dCRx K^{-1}](w_1,w_2) = \ii [\dCRy K^{-1}](w_1,w_2).
	\end{equation}
	We first study the left-hand side.
     \begin{equation}\label{eq:asymptotic:CR}
    	\begin{aligned}
	\frac{[\dCRx K^{-1}](w_1,w_2)}{\sqrt{\tan(\theta_1)}} &= \left(1 + \frac12 \langle\nabla V(w_1) ,x_1^--x_1^+\rangle + R(w_1,\eps)+O(\eps^3)\right)K^{-1}_{x_1^+,w_2}\\
	&\quad -\left(1 + \frac12 \langle\nabla V(w_1) ,x_1^+-x_1^-\rangle +  R(w_1,\eps)+O(\eps^3)\right)K^{-1}_{x_1^-,w_2}\\
	\end{aligned}
	\end{equation}
	Note that by Equation \eqref{eq:K^-1:order2} and Lemmas~\ref{lem:cv:green}, \ref{lem:cv:green:derivative} and~\ref{lem:cv:second:derivative}, uniformly for $w_1 \neq w_2$ in compact subsets of $\Omega \cup L$,
	\begin{equation}\label{eq:K^-1bound}
		|K^{-1}_{x_1^\pm,w_2}| = O\left(\frac{\eps}{|w_1-w_2|}\right) \quad ; \quad |K^{-1}_{x_1^+,w_2}-K^{-1}_{x_1^-,w_2}| = O\left(\frac{\eps^2}{|x_1-w_2|^2}\right).
	\end{equation}
    Hence, uniformly for $w_1 \neq w_2$ in compact subsets of $\Omega \cup L$,
    \begin{equation}\label{eq:bound:K:dual}
        [\dCRx K^{-1}](w_1,w_2) =  O\left(\frac{\eps^2}{|w_1-w_2|^2}\right).
    \end{equation}
    For points which are far away from each other, we can be more precise. Multiplying by $\tan(\theta_2)^{-1/2}$ and plugging in the asymptotic expansion of Equation \eqref{eq:CR:1} in Equation \eqref{eq:asymptotic:CR} (and using that $\dpartialone \dmGreen$ and $\dpartialone \dpartialtwo \dmGreen$ are bounded by Lemmas~\ref{lem:cv:green:derivative} and~\ref{lem:cv:second:derivative} to approximate $\dmGreen(x_1^+,x_2^-) = \dmGreen(x_1^+,x_2^-) + O(\eps)$ and $\dpartialtwo \dmGreen(x_1^+,w_2) = \dpartialtwo \dmGreen(x_1^-,w_2) + O(\eps)$), we obtain that uniformly for $(w_1,w_2)$ in compact subsets of $(\Omega \cup L)^2 \setminus \cD((\Omega \cup L)^2)$,
	 $$
	 	\begin{aligned}
			\frac{[\dCRx K^{-1}](w_1,w_2)}{\sqrt{\tan(\theta_1)\tan(\theta_2)}} 
            &= 4\dpartialone \dpartialtwo \dmGreen(w_1,w_2)(x_1^+-x_1^-)(x_2^+-x_2^-)\\
            &\quad -2\dpartialone \dmGreen(w_1,x_2^-)(x_1^+-x_1^-)\langle\nabla V(w_2) ,x_2^+-x_2^-\rangle\\
			&\quad -2 \langle\nabla V(w_1) ,x_1^+-x_1^-\rangle\dpartialtwo \dmGreen(x_1^-,w_2)(x_2^+-x_2^-)\\
            &\quad + \langle\nabla V(w_1) ,x_1^+-x_1^-\rangle\dmGreen(x_1^-,x_2^-) \langle\nabla V (w_2) ,x_2^+-x_2^-\rangle
             +O(\eps^3).
		\end{aligned}
	 $$
    Using Lemmas \ref{lem:cv:green:derivative}~and~\ref{lem:cv:second:derivative}, and the fact that for two vectors $z_1$ and $z_2$, $\langle z_1,z_2\rangle = (z_1)^{\intercal}\cdot z_2 = (z_2)^{\intercal}\cdot z_1$, we obtain that uniformly for $(w_1,w_2)$ in compact subsets of $(\Omega \cup L)^2 \setminus \cD((\Omega \cup L)^2)$,
	 \begin{equation}\label{eq:asymptotic:LHS:CR}
	 	\begin{aligned}
			\frac{[\dCRx K^{-1}](w_1,w_2)}{\sqrt{\tan(\theta_1)\tan(\theta_2)}} &=
			  (x_1^+-x_1^-)^{\intercal}\cdot(\nabla_1) (\nabla_2)^{\intercal} \mGreen \cdot (x_2^+-x_2^-) - \langle \nabla_1 \mGreen, x_1^+-x_1^-\rangle\langle\nabla V(\omega_2),x_2^+-x_2^-\rangle\\
     &+\langle \nabla V(w_1), x_1^+-x_1^-\rangle\left[-\langle \nabla_2 \mGreen , x_2^+-x_2^-\rangle + \mGreen \langle \nabla V(w_2),x_2^+-x_2^-\rangle\right]+o(\eps^2)\\
    &= (x_1^+-x_1^-)^{\intercal}\cdot\left[((\nabla_1)(\nabla_2)^{\intercal} \mGreen - \nabla_1 \mGreen \cdot (\nabla V(w_2))^{\intercal}\right.\\
    & \left.\quad - \nabla V(w_1) \cdot (\nabla_2 \mGreen)^{\intercal} + \mGreen \nabla V(w_1)\cdot (\nabla V(w_2))^{\intercal}\right]\cdot(x_2^+-x_2^-) + o(\eps^2)
		\end{aligned}
	 \end{equation}
	 where some arguments are omitted in the notation when there is no risk of confusion.

    Let $\cC \subset \Omega^2 \setminus \cD(\Omega^2)$ be a compact subset. We can find 
    $\cC_1 \subset (\Omega \cup L)^2 \setminus \cD((\Omega \cup L)^2)$ and $C>0$ a constant such that for all $(y_1,w_2) \in \cC$, there exists a continuous smooth path $\path$ between $y_0 := b^*$ and $y_1$ of bounded length $|\path| < C$ such that $(\path(t),w_2) \in \cC_1$ for all $t$. 
    Let $\dpath = \tw_0, \tw_1, \dots, \tw_n$ be a path in $\White$ from $y_0$ to $y_1$, that is if we denote by $\tx_k^-\ty_k^-\tx_k^+\ty_k^+$ the rhombus corresponding to $\tw_k$, $\ty_0^- = y_0, \ty_n^+ = y_1$ and for all $0 \leq k \leq n-1$, $\ty_k^+ = \ty_{k+1}^-$. 
    By the bounded angle assumption, we can assume further that the length of $\dpath$ is $n \leq C|\gamma|\eps^{-1}$ and that $\dpath$ remains at distance $\leq C\eps$ from $\path$ upon increasing $C$. 
    If we sum the $[\dCRy K^{-1}](\tw_k,w_2)$ along $\dpath$, consecutive terms do not cancel out, so we need to modify them slightly. We can write
	 $$
	 	\begin{aligned}
	 	[\dCRy K^{-1}](\tw_k,w_2) &= \frac{\exp\left(V(\ty_k^-)-\frac{V(\tx_k^+)+V(\tx_k^-)}{2}\right)}{\sqrt{\tan(\ttheta_k)}}\left(\exp(V(\ty_{k+1}^-)-V(\ty_k^-))K_{\ty_{k+1}^-,w_2}-K_{\ty_k^-,w_2}\right)\\
		&= q_k\left(\exp(V(\ty_{k+1}^-)-V(\ty_k^-))K_{\ty_{k+1}^-,w_2}-K_{\ty_k^-,w_2}\right),
		\end{aligned}
	 $$
	 where $q_k$ is defined by the last equation. Note that
	 	 \begin{equation}\label{eq:exponential}
	 \frac{\exp(V(\ty_k^-))}{q_k} [\dCRy K^{-1}](\tw_k,w_2) = \exp(V(\ty_{k+1}^-))K_{\ty_{k+1}^-,w_2}-  \exp(V(\ty_k^-)) K_{\ty_k^-,w_2},
	 \end{equation}
	 so summing along $\dpath$, the intermediate terms cancel out and since $\exp(V(y_0))K^{-1}_{y_0,w_2} = 0$ by Remark~\ref{rmk:temperleyan:corner}, we obtain
     \begin{equation}\label{eq:exponential2}
	 	\exp(V(y_1))K^{-1}_{y_1,w_2} = \sum_{k=0}^n \frac{\exp(V(\ty_k^-))}{q_k} [\dCRy K^{-1}](\tw_k,w_2).
        \end{equation}
	 Note that 
	 \begin{align*}
	 	q_k &= \frac{\exp(V(\ty_k^-)-V(\tw_k)+O(\eps^2))}{\sqrt{\tan(\ttheta_k)}} = \frac{1-\frac12 \nabla V(\tw_k)\cdot(\ty_k^+-\ty_k^-)}{\sqrt{\tan(\ttheta_k)}}+O(\eps^2)\\
        & = \frac{1+ O(\eps)}{\sqrt{\tan(\tilde \theta_k)}} ,
\end{align*}
	where the $O$ is uniform over $\Omega$. Hence, using the discrete massive Cauchy--Riemann Equation \eqref{eq:discrete:CR} as well as the \emph{a priori} bounds on $K^{-1}$ from \eqref{eq:K^-1bound}, and plugging in the asymptotic expansion of Equation \eqref{eq:asymptotic:LHS:CR}, we obtain: 
	 \begin{equation}\label{eq:K^-1:sum}
	 	\begin{aligned}
	  	\exp(V(y_1)K^{-1}_{y_1,w_2})
        &=\sum_{k=0}^n  \exp(V(\ty_k^-))\frac{\sqrt{\tan(\ttheta_k)}}{1+ O(\eps)}[\dCRy K^{-1}](\tw_k,w_2)\\
        %
        &=\ii\sum_{k=0}^n  \exp(V(\ty_k^-))(1+ O(\eps)){\sqrt{\tan(\ttheta_k)}}[\dCRx K^{-1}](\tw_k,w_2)\\
        %
		& = \ii\sqrt{\tan(\theta_2)}(1+ O(\eps))\Big\{\sum_{k=0}^n  \exp(V(\ty_k^-))\tan(\ttheta_k) (\tx_k^+-\tx_k^-)^{\intercal} \cdot \bigg[(\nabla_1)(\nabla_2)^{\intercal} \mGreen\\
        & \quad - \nabla_1 \mGreen \cdot (\nabla V(w_2))^{\intercal} - \nabla V(\tw_k) \cdot (\nabla_2 \mGreen)^{\intercal} + \mGreen \nabla V(\tw_k) \cdot (\nabla V(w_2))^{\intercal}\bigg]\\
        &\quad \cdot (x_2^+-x_2^-) + o(\eps^2)\Big\} 
		\end{aligned}
	 \end{equation}
    where the $o$ is uniform for $(y_1,w_2) \in \cC$. Finally, observe that for all rhombus $x^-y^-x^+y^+$ with angle $\theta$, it holds that
	 $$
	 	y^+-y^- = \tan(\theta)\Rot_{\pi/2} \cdot (x^+-x^-).
	 $$
   where we identify a complex number with the associated vector in $\RR^2$, and we let $\Rot_{\alpha}$ denote the only matrix such that $\Rot_{\alpha} z$ is the rotation of $z$ by an angle $\alpha \in [0,2\pi)$ for all vector $z$. Note that $\Rot_{\alpha}^{\intercal} = \Rot_{-\alpha}$. Hence, bearing in mind that the number of terms $n$ in this sum is at most $n =|\dpath|\le C \eps^{-1}$, and using the fact that for vectors $x,y$ and a matrix $M$ we have the identity $y^\intercal \cdot M \cdot x = \langle x, M^\intercal \cdot y \rangle$, we obtain  
	 $$
	 	\begin{aligned}
	  	\frac{\exp(V(\upsilon_1))K^{-1}_{y_1,w_2}}{\sqrt{\tan(\theta_2)} (1+ O(\eps))}
		& = \ii\Big\{\sum_{k=0}^n  \exp(V(\ty_k^-)) (\Rot_{-\pi/2}(\ty_k^+-\ty_k^-))^{\intercal} \cdot \left[(\nabla_1)(\nabla_2)^{\intercal} \mGreen - \nabla_1 \mGreen \cdot (\nabla V(w_2))^{\intercal}\right.\\
  &\quad \left.- \nabla V(\tw_k) \cdot (\nabla_2 \mGreen)^{\intercal} + \mGreen \nabla V(\tw_k)\cdot (\nabla V(w_2))^{\intercal}\right]\cdot (x_2^+-x_2^-)\Big\}+ o(\eps^2) n \\
		&= \ii\int_{\path}\exp(V(\upsilon))\mathrm{d}\path(\upsilon)^{\intercal} \cdot \Rot_{\pi/2} \cdot \left[(\nabla_1)(\nabla_2)^{\intercal} \mGreen - \nabla_1 \mGreen \cdot (\nabla V(\omega_2))^{\intercal}\right.\\
  &\quad \left.- \nabla V(\upsilon) \cdot (\nabla_2 \mGreen)^{\intercal} + \mGreen \nabla V(\upsilon) \cdot (\nabla V(\omega_2)))^{\intercal}\right]\cdot (x_2^+-x_2^-)+ o(\eps)\\
  &= \ii\Bigg\langle x_2^+-x_2^-,\int_{\path}\exp(V(\upsilon)) \left[(\nabla_2)(\nabla_1)^{\intercal} \mGreen - \nabla V(\omega_2) \cdot (\nabla_1 \mGreen)^{\intercal}\right.\\
  &\quad \left.- \nabla_2 \mGreen \cdot (\nabla V(\upsilon))^{\intercal} + \mGreen \nabla V(\omega_2)) \cdot (\nabla V(\upsilon))^{\intercal}\right]\cdot \Rot_{-\pi/2}\cdot \mathrm{d}\path(\upsilon)\Bigg\rangle + o(\eps)
		\end{aligned}.
	 $$
    uniformly for $(y_1,w_2) \in \cC$. Hence, we have proved the second point of Equation~\eqref{eq:thm:invK}, with the following definition of $\kappastarbar$: for all smooth path $\path$ from  $\beta^*$ to $\omega_1$ avoiding $\omega_2$, 
    $$
        \begin{aligned}
	   \kappastarbar(\omega_1,\omega_2) 
            &= \exp(-V(\omega_1))\int_{\path}\exp(V(\omega)) \left[(\nabla_2)(\nabla_1)^{\intercal} \mGreen(\omega,\omega_2) - \nabla V(\omega_2) \cdot (\nabla_1 \mGreen(\omega,\omega_2))^{\intercal}\right.\\
            &\quad \left.- \nabla_2 \mGreen(\omega,\omega_2) \cdot (\nabla V(\omega))^{\intercal} + \mGreen \nabla V(\omega_2)) \cdot (\nabla V(\omega))^{\intercal}\right]\cdot \Rot_{-\pi/2}\cdot \mathrm{d}\omega
	\end{aligned}
    $$
    The operators $\kappabar$ and $\kappastarbar$ are massive harmonic conjugates in the sense of Definition~\ref{def:massive:harmonic:conjugate}, since
    $$
        \begin{aligned}
            e^{V(\omega_1)}(\nabla_1)^{\intercal}(\kappabar e^{-V(\omega_1)}) 
            &= (\nabla_1)^{\intercal}\kappabar - \kappabar \cdot (\nabla)^{\intercal}V(\omega_1)\\
            &= \nabla_2(\nabla_1)^{\intercal} \mGreen - \nabla V(\omega_2) \cdot (\nabla_1)^{\intercal}\mGreen \\
            &\quad - \nabla_2 \mGreen \cdot (\nabla)^{\intercal}V(\omega_1) +\mGreen   \nabla V(\omega_2) \cdot (\nabla V)^{\intercal}(\omega_1)\\
            &= e^{-V(\omega)}(\nabla_1)^{\intercal}(e^{V(\omega_1)}\kappastarbar)\cdot \Rot_{\pi/2}.
        \end{aligned}
    $$
    Hence, the operators $\kappa$ and $\kappastar$ are also massive harmonic conjugates since their real parts, resp. imaginary parts, are massive harmonic conjugates .

    It remains to prove the bound for nearby points. Let $\cC \subset (\Omega \cup L) \times (\Omega \cup L \setminus \{\beta^*\})$ be a compact subset. We can find 
    $\cC_1$ a compact subset of $(\Omega \cup L) \times (\Omega \cup L \setminus \{\beta^*\})$ such that for some $C>0$ and all $(y_1,w_2) \in \cC$, there is a continuous smooth path $\path$ between $y_0 := b^*$ and $y_1$ such that $(\path(t),w_2) \in \cC_1$ for all $t$, $|\path| \leq C$ and $\path$ remains at distance at least $|y_1-w_2|/C$ from $w_2$. Moreover, we can assume that for some $\nu > 0$ if $w_2$ is within $\nu$ of $y_1$, then the portion of the path starting from $y_1$ is locally a straight line escaping $B(w_2, 2\nu)$. 

    Let $(y_1,w_2) \in \cC$. 
    As before, by the bounded angle assumption we can find $\dpath = (\tw_0, \dots, \tw_n)$ of length $n\leq C \eps^{-1}$ remaining at distance at most $C\eps$ from $\path$. Using the same argument as before and the estimate Equation~\eqref{eq:bound:K:dual}, Equation~\eqref{eq:K^-1:sum} implies that
    $$
        K^{-1}_{y_1,w_2}
		 = O\left(\sum_{k=0}^n  \frac{\eps^2}{|\tw_k-w_2|^2}\right)
    $$
    Using that $\path$ is a straight line near $y_1$ and the fact that
    $$
        \int_r^1 \frac{\mathrm{d}s}{s^2} = O\left(\frac{1}{r}\right),
    $$
    we obtain that
    $$
        K^{-1}(y_1,w_2) = O\left(\frac{\eps}{|b^*-w_2| \wedge |y_1-w_2|}\right),
    $$
    uniformly for $(y_1,w_2)$ in compact subsets of $(\Omega \cup L)^2$.
\end{proof}

\subsection{Convergence of height moments (proof of Theorem \ref{thm:height})}\label{sec:Li}
In this section, we briefly summarize and adapt the arguments of~\cite{Li17} to explain how Theorem~\ref{thm:isoradial} implies Theorem~\ref{thm:height}.

\begin{remark}
    The interested reader who will compare this with Li's work might be interested by the following two remarks on notation:
    \begin{itemize}
        \item Li calls “dual graph” and denotes by $\cG'$ what we denote by $\Primal$ and call “primal graph”. Even if they bear different names, these two objects are defined exactly the same way in her work and ours. 
        \item for $f : \RR^2 \to \RR$, Li calls “directional derivative” of $f$ at $z$ in the direction $\zeta$ the complex number $\lim_{\eps \to 0}\frac{f(z+\eps \zeta)-f(z)}{\eps\zeta}$. We will avoid this notation.
    \end{itemize}
\end{remark}

We define a massive Dirac operator, corresponding to the $\dirac$ operator of~\cite{Kenyon2009},~\cite{Li17} in the critical case.
For $w \in \White$, if $x^+$ denotes the neighbor of $w$ in $\Primal$ with $K(w,x^+) > 0$, let $\psi(w) = 2\sqrt{\cos(\theta)\sin(\theta)}\frac{x^+-w}{|x^+-w|}$.
The operator $\dirac$ is defined as the gauge change of $K$ via $\psi$:
\begin{equation}\label{eq:gauge:change}
    \dirac(w,b) = \psi(w)K(w,b), w \in \White, b \in \Primal \cup \Dual.
\end{equation}

Let $f, f'$ be two faces of the superposition graph, separated by an edge $e = wb$ of the superposition graph. Let $\xi(wb) = 0$ if $w$ is on the left of $ff'$, $\xi(wb) = 1$ otherwise. The height difference between $f$ and $f'$ is
\begin{equation}\label{eq:height:difference}
    h_{\eps}(f')-h_{\eps}(f) = (-1)^{\xi(wb)}(1_e-\PP(e))
\end{equation}
where $1_e$ denotes the indicator that $e$ belongs to the dimer configuration, $\PP(e) = \EE(1_e)$. 

With these notations, we can reformulate Kasteleyn theory for the edge probabilities (see \cite[Lemma~11]{Li17}) as follows: for edges $e_1 = (w_1b_1), \dots, e_k = (w_kb_k)$,
\begin{equation}\label{eq:expected:height}
    \EE[(1_{e_1}-\PP(e_1))\dots(1_{e_k}-\PP(e_k))] = \prod_{i=1}^k \dirac(w_i,b_i) \det\big(\dirac^{-1}(w_i,b_j)1_{i \neq j}\big).
\end{equation}
This holds as well in the massive case, since the proof uses only the local statistics formula and determinantal identities.

Let $k \geq 2$, $\beta_1, \dots, \beta_k \in L \setminus \{\beta^*\}$ be points on the boundary of $\partial \Omega$. 
Let $\zeta_1, \dots, \zeta_k$ be faces of the superposition graph. Assume first that these faces $\zeta_i$ lie in a given compact subset of $\Omega^k \setminus \cD(\Omega^k)$.
 Let $\path_i$ be pairwise disjoint paths of faces of the superposition graph, starting from the nearest face to $\beta_i$ below the straight boundary and ending at $\zeta_i$. 
For $z_i$ a face of the superposition graph along the path $\path_i$, we let $\Delta z_i$ denote the increment between $z_i$ and the next face.
Let $E_i$ be the set of edges $w_ib_i$ intersecting $\path_i$.
 
Li notes that if we fix $h_\eps(\beta_1)=0$, then $h_\eps(\beta_i)=0$ for all $i$. This also holds here.
Writing $h_{\eps}(\zeta_i) = \sum_{z_i \in \path_i}(h_\eps(z_i + \Delta z_i)-h_\eps(z_i))$, the joint moments of the height can be computed using Equations~\eqref{eq:height:difference} and~\eqref{eq:expected:height}.
Expanding the determinant, we obtain
\begin{equation}\label{eq:det:expanded}
    \EE\left[\prod_{i=1}^k h_\eps(\zeta_i)\right]
    = \sum_{\sigma}\sum_{w_ib_i \in E_i}(-1)^{\xi(w_1b_1)+\dots+\xi(w_kb_k)}\sgn(\sigma)\prod_{i=1}^k \dirac(w_i,b_i)\dirac^{-1}(b_{\sigma(i)},w_i).
\end{equation}
It is a sum over all permutations with no fixed point, and all $k$-tuples of edges $(w_1b_1, w_kb_k)$ with $w_ib_i \in E_i$.
Note that since $\dirac$ is obtained from $K$ by Equation~\eqref{eq:gauge:change}:
$$
    \dirac^{-1}(b,w) = \psi^{-1}(w)K^{-1}(b,w).
$$
Hence, Theorem~\ref{thm:isoradial} implies that uniformly for $(x_1,w_2)$, resp. $(y_1,w_2)$, in compact subsets of $\Omega^2 \setminus \cD(\Omega^2)$,
$$
    \begin{aligned}
    \dirac^{-1}(x_1,w_2) 
    &= \frac{\eps}{x_2^+-x_2^-}\langle \kappabar(x_1,w_2),x_2^+-x_2^-\rangle + o(\eps)\\
    \dirac^{-1}(y_1,w_2) 
    &= \frac{i\eps}{x_2^+-x_2^-}\langle \kappastarbar(y_1,w_2),x_2^+-x_2^-\rangle + o(\eps).
    \end{aligned}
$$
We will write 
$$
    \dirac^{-1}(b,w) = \eps i^{\eta(b)}f_{\eta(b)}(b,w) + o(\eps)
$$
uniformly for $(w,b)$ in compact subsets of $\Omega^2 \setminus \cD(\Omega^2)$, with $\eta(b)= 1_{b \in \Dual}$ and
$$
    f_0(x_1,w_2) = \frac{\langle \kappabar(x_1,w_2),x_2^+-x_2^-\rangle}{x_2^+-x_2^-} \quad ; \quad f_1(y_1,w_2) = \frac{\langle \kappastarbar(x_1,w_2),x_2^+-x_2^-\rangle}{x_2^+-x_2^-}
$$
This extends Sections~3 and~4 of~\cite{Li17}, where she obtains an asymptotic expression for $\dirac^{-1}$.
Hence, when the path are disjoint, plugging this asymptotic in Equation~\eqref{eq:det:expanded}, we can write
$$
    \EE\left[\prod_{i=1}^k h_\eps(\zeta_i)\right]
    = \sum_{\sigma}\sum_{w_ib_i \in E_i}(-1)^{\xi(w_1b_1)+\dots+\xi(w_kb_k)}\sgn(\sigma)\prod_{i=1}^k \eps\dirac(w_i,b_i)i^{\eta(b_{\sigma(i)})}f_{\eta(b_{\sigma(i)})}(b_{\sigma(i)},w_i) + o(\eps^k).
$$
Now, Li observes that if $\Delta z_i$ is the increment of $\path_i$ along the edge dual to $w_ib_i$, then 
$$
    \ii(-1)^{\xi(w_ib_i)}\eps \dirac(w_i,b_i) = 2\Delta z_i +O(\eps^{2}).
$$
We find a factor $2$ missing from her computations here, which cancels out later. 
In the non-massive case, this is an equality, see her Lemma~12: there is no $O(\eps^2)$. 
In the massive case, this holds up to $O(\eps^2)$ (uniformly for $w,b \in \Omega$) since the massive Kasteleyn operator $K(w,b)$ is equal to its critical value, up to $1+O(\eps)$. 
Hence
\begin{equation}\label{eq:joint:height}
    \EE\left[\prod_{i=1}^k h_\eps(\zeta_i)\right]
    = \sum_{\sigma} \sum_{w_ib_i \in E_i}\sgn(\sigma)\prod_{i=1}^k 2\Delta z_i (-i)\ii^{\eta(b_{\sigma(i)})}f_{\eta(b_{\sigma(i)})}(b_{\sigma(i)},w_i) + o(\eps^k).
\end{equation}
Now, Li observes (see her Lemma~13) that
$$
    \overline{x_i^+-x_i^-}\Delta z_i = (x_i^+-x_i^-)(-1)^{\eta(b_i)+1}\overline{\Delta z_i}.
$$
This is true because if $\eta(b_i) = 0$, i.e. $b_i \in \Primal$, then $x_i^+-x_i^-$ and $\Delta z_i$ are orthogonal while if $\eta(b_i) = 0$, i.e. $b_i \in \Dual$, then $x_i^+-x_i^-$ and $\Delta z_i$ are parallel.
Hence, if we let $\overline{\kappa^{(0)}} = \kappabar$, $\overline{\kappa^{(1)}} = \kappastarbar$ we obtain
\begin{equation}\label{eq:dz:f}
    \begin{aligned}
        (-\ii)\ii^{\eta(b_j)}2\Delta z_i f_{\eta(b_j)}(b_j,w_i)
        &= (-\ii)\ii^{\eta(b_j)}2\Delta z_i  \frac{\langle \overline{\kappa^{(\eta(b_j))}}(b_j,w_i), x_i^+-x_i^-\rangle}{x_i^+-x_i^-}\\
        &= (-\ii)\ii^{\eta(b_j)}\Delta z_i \frac{\kappa^{(\eta(b_j))}(b_j,w_i)(x_i^+-x_i^-)+\overline{\kappa^{(\eta(b_j))}}(b_j,w_i)\overline{(x_i^+-x_i^-)}}{x_i^+-x_i^-}\\
        &= (-\ii)\ii^{\eta(b_j)}\left(\kappa^{(\eta(b_j)}(b_j,w_i) \Delta z_i + (-1)^{\eta(b_i)+1} \overline{\kappa^{(\eta(b_j)}}(b_j,w_i) \overline{\Delta z_i}\right)\\
        &= \big(F_0(b_j,w_i)+(-1)^{\eta(b_j)}F_1(b_j,w_i)\big)\Delta z_i + \\
        &\quad + (-1)^{\eta(b_j)+\eta(b_i)}\big(\overline{F_0(b_j,w_i)}+(-1)^{\eta(b_j)}\overline{F_1(b_j,w_j)}\big)\overline{\Delta z_i},
    \end{aligned}
\end{equation}
with 
$$
    F_0 = \frac{\kappastar - \ii\kappa}{2} \quad ; \quad F_1 = -\frac{\kappastar + \ii\kappa}{2},
$$
which should be compared with the statement of \cite[Lemma~13]{Li17}.
If we denote $z^{(0)} = z, z^{(1)} = \overline{z}$, developing Equation~\eqref{eq:joint:height}, we obtain
\begin{equation}\label{eq:joint:height:bis}
    \begin{aligned}
    \EE\left[\prod_{i=1}^k h_\eps(\zeta_i)\right]
    &= \sum_{s_i, t_i \in \{0,1\}}\sum_{\sigma}\sgn(\sigma)\sum_{w_ib_i \in E_i}\prod_{i=1}^k F^{(s_i)}_{t_i}(b_{\sigma(i),w_i})\Delta z_i^{(s_i)}(-1)^{t_i \eta(b_{\sigma(i)}) + s_i\eta(b_{\sigma(i)})+ s_i\eta(b_i)}\\
    &= \sum_{s_i, t_i \in \{0,1\}}\sum_{\sigma}\sgn(\sigma)\sum_{w_ib_i \in E_i}\prod_{i=1}^k F^{(s_i)}_{t_i}(b_{\sigma(i),w_i})\Delta z_i^{(s_i)}(-1)^{\eta(b_{\sigma(i)})(t_i+s_i+s_{\sigma(i)})}
    \end{aligned}
\end{equation}
The indices $s_i,t_i$ denote which of the four terms we choose in Equation~\eqref{eq:dz:f} for each term of the product in Equation~\eqref{eq:joint:height}.
The last equation is obtained by reindexing the last factor.
Li proves in her Lemma~14 that along any portion of the path $\path_i$,
$$
    \sum (-1)^{\eta(b_i)}\Delta z_i = O(\eps),
$$
hence every term with $s_i+s_{\sigma(i)}+t_i =1 \mod 2$ for some $i$ remains $O(\eps)$ when taking the sum along the path $\path_i$, since it is the product of an oscillating term with a smooth function. We obtain
$$
    \EE\left[\prod_{i=1}^k h_\eps(\zeta_i)\right]
    = \sum_{s_i \in \{0,1\}}\sum_{\sigma}\sgn(\sigma)\sum_{w_ib_i \in E_i}\prod_{i=1}^k F^{(s_i)}_{s_i+s_{\sigma(i)}}(b_{\sigma(i)},w_i)\Delta z_i^{(s_i)} + o(\eps)
$$
This is a Riemann-sum for a $k$-fold integral: 
$$
    \begin{aligned}
    \EE\left[\prod_{i=1}^k h_\eps(\zeta_i)\right]
    \overset{\eps \to 0}{\longrightarrow} &\sum_{s_i \in \{0,1\}}\sum_{\sigma}\sgn(\sigma)\int_{\path_1}\cdots \int_{\path_k}\prod_{i=1}^k F^{(s_i)}_{s_i+s_{\sigma(i)}}(z_{\sigma(i)},z_i)\mathrm{d} z_i^{(s_i)}\\
    &= \sum_{s_i \in \{0,1\}} \int_{\path_1}\cdots \int_{\path_k}  \det_{i \neq j}\big[F_{s_i+s_j}^{(s_j)}(z_i,z_j)\big] \prod_{i=1}^k \mathrm{d}z_i^{(s_i)}.
    \end{aligned}
$$

Note that the diagonal terms in the determinant are $0$ since we sum over permutations $\sigma$ with no fixed points.
Note that this holds uniformly for $(\zeta_1, \dots, \zeta_k) \in \Omega^k \setminus \cD(\Omega^k)$, upon choosing paths $\path_i$ of length $O(1/\eps)$ which remain at macroscopic distance from one another.
This proves Equation~\eqref{eq:thm:height}.

We now prove the bound Equation~\eqref{eq:bound:height} which holds for any faces $(\zeta_1, \dots, \zeta_k)$ in compact subsets of $\Omega^k$.
In the non-massive case, this is done in Lemma~16 of~\cite{Li17}. We prefer to use the argument of Appendix~A.3 of~\cite{LT15}, where it is explained in more details.
Note that they can give a more precise bound since they have a more precise asymptotic estimate of $K^{-1}$. We will be satisfied with the following very rough bound.

Let $\cC \subset \Omega$ be a compact subset, $(\zeta_1, \dots, \zeta_k)$ be faces with $\zeta_i \subset \cC$. 
We choose paths $\path_1, \dots, \path_k$ from respectively $\zeta_1, \dots, \zeta_k$ towards distinct points on the straight boundary, at distance of order $1$ from each other.
We choose paths which are macroscopically far away except possibly near their extremities $\zeta_i$. 
Near the extremities, we choose $\path_i$ which remains within $O(\eps)$ of a straight line.
We choose a different asymptotic direction for each of the straight lines.
An example is provided in Figure~12 of~\cite{LT15}, in the case where the faces are partitioned into two sets, $5$ of them near a face $f$ and $5$ of them near a face $f'$ at macroscopic distance from $f$.
Even though~\cite{LT15} is in the square lattice, we can do the same thing using the bounded angle assumption. 
We can also assume that each path has length $N_i = |\path_i| = O(1/\eps)$.

Number the edges of each path starting from the extremity $\zeta_i$: $e_i(1), e_i(2), \dots, e_i(N_i)$. 
Let $\delta_{ij} = |\zeta_i-\zeta_j|$.
Note that for such a choice of paths, we have 
$$
    \dist(e_i(m_i),e_j(m_j)) \geq c_{ij}(\delta_{ij}+\eps(m_i+m_j)) 
$$
for all $1 \leq i,j \leq k$, $1 \leq m_i \leq N_i$, $1 \leq m_j \leq N_j$ and some constant $c_{ij}>0$.
Since 
$$
    \dirac(b_i,w_j) = O\left(\frac{1}{|b_i-w_j|}\right)
$$
uniformly for $b_i,w_j$ in compact subsets of $\Omega \cup L$ by Theorem~\ref{thm:isoradial}, and since $\dirac(w,b) = O(1)$ uniformly for $w\sim b$, Equation~\eqref{eq:det:expanded} gives, for some constant $C>0$ changing from line to line
$$ 
    \begin{aligned}
    \EE\left[\prod_{i=1}^k h_\eps(\zeta_i)\right]
    &\leq C \sum_{\sigma} \sum_{1 \leq m_i \leq N_i} \prod_{i=1}^k \frac{\eps}{\delta_{i\sigma(i)}+\eps(m_i+m_{\sigma(i)})}\\
    &\leq C \sum_{\sigma} \sum_{1 \leq m_i \leq N_i} \prod_{i=1}^k \frac{\eps}{\delta_{i\sigma(i)}+\eps m_i}\\
    &\leq C \sum_{\sigma} \prod_{i=1}^k \sum_{1 \leq m_i \leq N_i} \frac{\eps}{\delta_{i\sigma(i)}+\eps m_i}\\
    &\leq C \sum_{\sigma} \prod_{i=1}^k (1+|\log(\delta_{i\sigma(i)})|)
    \end{aligned}
$$
where at the last line we use a classical bound for the harmonic series and the fact that $N_i = O(1/\eps)$.

\section{Identification with sine-Gordon model}

\subsection{Boundary value problem for correlation functions (Proof of Proposition \ref{lem:bc})}
\label{SS:corr_bvp}

Before proving Proposition \ref{lem:bc}, we first elaborate a bit on $\mGreen, \kappa, \kappastar, F_0, F_1$. 
Recall that $\mGreen$ is defined by Equation~\eqref{eq:def:mGreen} or~\eqref{eq:green:distribution}.
In the critical case $\mass = 0$ and in the upper half-plane $\Omega = \HH$, it is well-known (see for example~\cite{KenyonGFF}) that
$$
    G_{\HH}(w,z) = \frac{1}{4\pi}\left(\log|z-w|^2-\log|z-\bar w|^2\right).
$$
Still in the critical case $\mass = 0$ but for a general domain $\Omega$, the expression of the Green function is obtained by conformal invariance: if $g : \Omega \to \HH$ is a conformal map, 
\begin{equation}\label{eq:Green:half-plane:domain}
    G(w,z) = G_{\HH}(g(w),g(z)) = \frac{1}{4\pi}\left(\log|g(z)-g(w)|^2-\log|g(z)-\bar g(w)|^2\right)
\end{equation}
There is no such explicit expression of the massive Green function in the general case $M \neq 0$. 
A direct consequence of Definition~\ref{def:mGreen} is the bound
\begin{equation}\label{eq:bound:mGreen}
    0 \leq \mGreen(w,z) \leq G(w,z).
\end{equation}
The massive and non-massive Green functions can be related through the \textbf{resolvent identity} (Proposition 4.10 of~\cite{BHS}, see also~\cite{MakarovSmirnov, ChelkakWan} for earlier use of this identity):
\begin{equation}\label{eq:resolvent:identity:continuous}
    G^m(w,z) = G(w,z) + \int_{y \in \Omega}\mass(y)G^m(w,y)G(y,z)\mathrm{d}y.
\end{equation}
This is proved by observing that the right-hand side is solution to the Dirichlet problem~\eqref{eq:green:distribution} defining the massive Green function.
Note that the integral in the right-hand side of Equation~\eqref{eq:resolvent:identity:continuous} is well-defined thanks to the bound~\eqref{eq:bound:mGreen}.
The resolvent identity is very useful since it expresses the massive Green function as an explicit singular object (the critical Green function) + a non-singular non-explicit object which is easier to bound.

\textbf{The operators $\kappa$ and $\kappastar$.}
Recall that $\kappa$ is defined by Equation~\eqref{eq:def:kappa} 
\begin{equation}
    \begin{aligned}
        \kappa(w,z) 
        &= \overline{\nabla_z} \mGreen(w,z) - \overline{\alpha}(z)\mGreen(w,z)
    \end{aligned}
\end{equation}
and that $z \in \Omega$, $\kappastar(\cdot,z)$ is defined as the $\alpha$-conjugate of $\kappa(\cdot,z)$ in the sense of Definition~\ref{def:massive:harmonic:conjugate}.

In the critical case $\alpha = 0$, we write $\kappa_0,\kappastar_0$.
By Equation~\eqref{eq:Green:half-plane:domain},
    \begin{equation}\label{eq:kappa_0}
        \kappa_0(w,z) = \frac{1}{4\pi}\left(\frac{g'(z)}{g(z)-g(w)}-\frac{g'(z)}{g(z)-\bar g(w)}\right).
    \end{equation}
Be careful that $\kappa_0^*$ is defined as the conjugate in the sense of Definition~\ref{def:massive:harmonic:conjugate} with $\alpha = 0$. 
    It is easily checked that the conjugate (in this sense) of a holomorphic function $f(w)$ is $-\ii f(w)$, while the conjugate of an anti-holomorphic function $h(w)$ is $\ii h(w)$, hence
    $$
        \kappastar_0(w,z) = -\frac{\ii}{4\pi}\left(\frac{g'(z)}{g(z)-g(w)} + \frac{g'(z)}{g(z)-\bar g(w)} \right).
    $$
    They satisfy
    \begin{equation}\label{eq:dirac:critical}
        \dzbar \kappa_0(w,z) = \frac{1}{2}\delta_w(z) \quad ; \quad \dzbar \kappastar_0(w,z) = -\frac{\ii}{2}\delta_w(z).
    \end{equation}
In the general case, we write $\kappa = u+\ii v$, $\kappastar = u^* + \ii v^*$, where $u^*(\cdot,z)$, resp. $v^*(\cdot,z)$, is the $\alpha$-conjugate of $u(\cdot,z)$, resp. $v(\cdot,z)$.
Let us be more precise. Let $z \in \Omega$.
For all $w \in \Omega$,
$$
    e^{V(w)}\nabla_w(u(w,z)e^{-V(w)}) = e^{-V(w)}(-\ii)\nabla_w(u^*(w,z)e^{V(w)}),    
$$
    (and the same holds for $v,v^*$). 
    For any $w_1, w_2 \in \Omega$, integrating along any path from $w_1$ to $w_2$ avoiding $z$,
    \begin{equation}\label{eq:uu^*}
        u^*(w_2,z)e^{V(w_2)} - u^*(w_1,z)e^{V(w_1)} 
        = \int_{w_1}^{w_2}e^{2V(w)}\left\langle \ii\nabla_w(u(w,z)e^{-V(w)}),\mathrm{d}w\right\rangle.
    \end{equation}
    (and the same holds for $v, v^*)$.
    In particular, since $\kappastar$ vanishes at $\beta^*$ by definition (see Theorem~\ref{thm:isoradial},
    \begin{equation}\label{eq:uu^*:zero}
        u^*(w,z) = e^{-V(w)} \int_{\beta^*}^{w}e^{2V(\omega)}\left\langle \ii\nabla_\omega(u(\omega,z)e^{-V(\omega)}),\mathrm{d}\omega\right\rangle,
    \end{equation}
    and the same holds for $v,v^*$.
    The following consequence of the resolvent identity for the operators $\kappa, \kappastar$ will be crucial in the proof of Proposition~\ref{lem:F0F1}.
    \begin{lem}\label{lem:kappa:kappa_0}
        For all $z,w \in \Omega$,
        \begin{equation}
            \label{eq:step3}
        \kappa(w,z) = \kappa_0(w,z) + O(\log|z-w|) \quad ; \quad \kappastar(w,z) = \kappastar_0(w,z) + O(\log|z-w|).
        \end{equation}
    \end{lem}

    \begin{proof}
    Differentiating the resolvent identity, we obtain
    \begin{equation}\label{eq:resolvent:dgreen}
    \begin{aligned}
        \dz_z \mGreen(w,z) 
        &= \dz_z G(w,z) + \int_{y \in \Omega}\mass(y)G^m(w,y)\dz_z G(y,z)\mathrm{d}y.
    \end{aligned}
\end{equation}
    so
    \begin{equation}\label{eq:resolvent:kappa}
    \begin{aligned}
        \kappa(w,z) 
        &= \kappa_0(w,z) + \int_{y \in \Omega}\mass(y)G^m(w,y)\kappa_0(y,z)\mathrm{d}y - \bar \alpha(z)\mGreen(w,z).
    \end{aligned}
\end{equation}
    Since $\mGreen(w,z) = O(\log|z-w|)$ and $\dz_z G(y,z) = \kappa_0(y,z) = O(1/|z-y|)$ (by the explicit expression of $\kappa_0$ Equation~\eqref{eq:kappa_0}), we obtain that 
    \begin{equation}\label{eq:dgreen=dgreen_0+error}
        \dz_z \mGreen(w,z) = \dz_z G(w,z) + O(1)
    \end{equation}
    (the integral remains bounded) which implies 
    \begin{equation}\label{eq:k=k_0+error}
        \kappa(w,z) = \kappa_0(w,z) + O(\log(|z-w|).
    \end{equation}
    We establish the same for $\kappa_0^*$.
    Taking the real part of Equation~\eqref{eq:resolvent:kappa}, we obtain
    \begin{equation}
    \begin{aligned}
        u(w,z) 
        &= u_0(w,z) + \int_{y \in \Omega}\mass(y)G^m(w,y)u_0(y,z)\mathrm{d}y-\Re(\alpha)(z)\mGreen(w,z)\\
        &=: u_0(w,z) + u_1(w,z) + u_2(w,z).
    \end{aligned}
\end{equation}
    We denote $\kappa_0 = u_0+\ii v_0$, $\kappastar_0 = u_0^*+\ii v_0^*$. 
    Be careful that $u_0^*, v_0^*$ denote the $\alpha$-conjugate with $\alpha = 0$ (that is $\nabla u^* = \ii\nabla u$) while $u^*, v^*$ denote the $\alpha$-conjugate with a general $\alpha$ (that is $e^{-V} \nabla(u^* e^V) = \ii e^V \nabla(ue^{-V})$).
    We show that $u^* = u_0^* + O(\log|z-w|)$.
    By Equation~\eqref{eq:uu^*:zero},
    \begin{equation}\label{eq:u*=U0+U1+U2}
    \begin{aligned}
        u^*(w,z) 
        &= e^{-V(w)} \int_{\beta^*}^{w}e^{2V(\omega)}\left\langle \ii\nabla_\omega(u(\omega,z)e^{-V(\omega)}),\mathrm{d}\omega\right\rangle\\
        &= e^{-V(w)} \int_{\beta^*}^{w}e^{2V(\omega)}\left\langle \ii\nabla_\omega((u_0(\omega,z)+u_1(\omega,z)+u_2(\omega,z))e^{-V(\omega)}),\mathrm{d}\omega\right\rangle\\
        &=: U_0(w,z) + U_1(w,z) + U_2(w,z).
    \end{aligned}
    \end{equation}
    The first term is 
    \begin{equation}
    \begin{aligned}
        U_0(w,z) 
        &= e^{-V(w)} \int_{\beta^*}^{w}e^{2V(\omega)}\left\langle \ii\nabla_\omega(u_0(\omega,z)e^{-V(\omega)}),\mathrm{d}\omega\right\rangle\\
        &= e^{-V(w)} \int_{\beta^*}^{w}e^{V(\omega)}\left\langle \ii\nabla_\omega u_0(\omega,z),\mathrm{d}\omega\right\rangle
        - e^{-V(w)} \int_{\beta^*}^{w}e^{V(\omega)}u_0(w,z)\left\langle \ii\nabla_\omega V(\omega),\mathrm{d}\omega\right\rangle\\
        &= \int_{\beta^*}^{w}(1+O(w-\omega))\left\langle\nabla_\omega u_0(\omega,z),\mathrm{d}\omega\right\rangle
        - e^{-V(w)} \int_{\beta^*}^{w}e^{V(\omega)}u_0(w,z)\left\langle \ii\nabla_\omega V(\omega),\mathrm{d}\omega\right\rangle\\
        &= u_0^*(w,z) + \int_{\beta^*}^{w}O(w-\omega)\left\langle \ii\nabla_\omega(u_0(\omega,z)),\mathrm{d}\omega\right\rangle
        - e^{-V(w)} \int_{\beta^*}^{w}e^{V(\omega)}u_0(w,z)\left\langle \ii\nabla_\omega V(\omega),\mathrm{d}\omega\right\rangle
    \end{aligned}
    \end{equation}
    since $e^{ V(w)- V(\omega)} = 1+O(w-\omega)$ and by definition of $u_0^*$.
    Using the explicit form of $\kappa_0$ from Equation~\eqref{eq:kappa_0} and the definition $u_0(w,z) = \Re(\kappa_0(w,z))$, we obtain that 
    \begin{equation}\label{eq:bound:u_0}
        u_0(w,z) = O(1/(z-w)) \quad ; \quad \nabla_\omega u_0(\omega,z) = O(1/(z-w)^2),
    \end{equation}
    and after integrating this gives 
    \begin{equation}\label{eq:bound:U0}
        U_0(w,z) = u_0^*(w,z) + O(\log|z-w|).
    \end{equation}
    The second term is
    \begin{equation}
    \begin{aligned}
        U_1(w,z) 
        &= e^{- V(w)} \int_{\beta^*}^{w}e^{2 V(\omega)} \left\langle \ii\int_{y \in \Omega}\mass(y)u_0(y,z)\nabla_\omega \mGreen(\omega,y)\mathrm{d}y,\mathrm{d}\omega\right\rangle 
    \end{aligned}
    \end{equation}
    By symmetry of the Green function and Equation~\eqref{eq:dgreen=dgreen_0+error},
    $$
        \nabla_\omega\mGreen(\omega,y) = \nabla_\omega G(\omega,y) + O(1) = O\left(\frac{1}{\omega-y}\right)
    $$
    Using this and Equation~\eqref{eq:bound:u_0},
    $$
        \int_{y \in \Omega}\mass(y)u_0(y,z)\nabla_\omega \mGreen(\omega,y)\mathrm{d}y = O\left(\frac{1}{\omega-z}\right) 
    $$
    And hence
\begin{equation}\label{eq:bound:U1}
        U_1(w,z) = \int_{\beta^*}^{w} \left\langle O\left(\frac{1}{\omega-z}\right),\mathrm{d}\omega\right\rangle = O(\log|z-w|)).
    \end{equation}
    Finally, using Equation~\eqref{eq:bound:u_0} and the fact that $ V$ is a smooth function, we obtain
    \begin{equation}\label{eq:bound:U2}
        U_2(w,z) = \int_{\beta^*}^{w} \left\langle O\left(\frac{1}{\omega-z}\right),\mathrm{d}\omega\right\rangle = O(\log|z-w|).
    \end{equation}
    Plugging Equations~\eqref{eq:bound:U0},~\eqref{eq:bound:U1} and~\eqref{eq:bound:U2} in Equation~\eqref{eq:u*=U0+U1+U2} gives
    $$
        u^*(w,z) = u_0^*(w,z) + \log(|w-z|).
    $$
    The same arguments hold for $v^*$, which implies
    \begin{equation}
        \kappa^*(w,z) = \kappa_0^*(w,z) + O(\log|z-w|)).
    \end{equation}
    This concludes the proof of Lemma \ref{lem:kappa:kappa_0}.
    \end{proof}

    We are ready to prove Proposition~\ref{lem:F0F1}.
\begin{proof}[Proof of Proposition~\ref{lem:F0F1} ]
    We first study the off-diagonal behavior, that is we establish Proposition~\ref{lem:F0F1}  for $z \neq w$.
    We start with the first column.
    Set $z \in \Omega$. 
    The fact that $\kappa(\cdot,z)$ and $\kappastar(\cdot,z)$ are $\alpha$-conjugate for $z \neq w$ can be reformulated as follows.
    If we define $f(w) = u(w,z)+\ii u^*(w,z)$ and $g(w) = v(w,z)+\ii v^*(w,z)$, they are $\alpha$-holomorphic away from $z$ i.e. satisfy for all $w \neq z$,
    \begin{equation}\label{eq:kappa:f:g}
    \bar\partial_w f(w) = \frac12 \alpha(w)\bar f(w) \quad ; \quad  \bar\partial_w g(w) = \frac12 \alpha(w)\bar g(w).
\end{equation}
Since
    $$
        \begin{aligned}
        F_0 
        &= \frac12 (\kappastar - \ii \kappa)
        = \frac12 (u^*+\ii v^* - \ii u+v)
        = \frac12(v+\ii v^* -\ii(u+\ii u^*)) 
        = \frac12(g -\ii f)\\
        F_1 
        &= -\frac12 (\kappastar + \ii \kappa)
        = -\frac12 (u^*+\ii v^* + \ii u -v)
        = -\frac12(-(v-\ii v^*) +\ii(u-\ii u^*)) 
        = \frac12(\bar g -\ii \bar f),
        \end{aligned}
    $$
Equation~\eqref{eq:kappa:f:g} implies that for $w \neq z$,
    \begin{equation}\label{eq:F0F1:offdiag:w}
        \begin{aligned}
        \bar \partial_w F_0(w,z) 
        &= \bar \partial_w \frac{g(w) -\ii f(w)}{2}
        = \frac{\alpha(w)}{2}\frac{\bar g(w) - \ii \bar f(w)}{2}
        = \frac{\alpha(w)}{2}F_1(w,z)\\
        \partial_w F_1(w,z) 
        &= \bar \partial_w \frac{\bar g(w) -\ii\bar f(w)}{2}
        = \frac{\bar\alpha(w)}{2}\frac{g(w) - \ii f(w)}{2}
        = \frac{\bar\alpha(w)}{2}F_0(w,z).
        \end{aligned}
    \end{equation}

    Set $w \in \Omega$.
    We claim that for $z \neq w$,
    \begin{equation}\label{eq:kappa:kappastar}
        \bar \partial_z \kappa(w,z) = -\frac{\bar \alpha}{2}\bar\kappa(w,z) \quad ; \quad \bar \partial_z \kappastar(w,z) = -\frac{\bar \alpha}{2}\bar\kappastar(w,z).
    \end{equation}
    This is proved below.
    Assuming Equation~\eqref{eq:kappa:kappastar}, for $z \neq w$,
    \begin{equation}\label{eq:F0F1:offdiag:z}
        \begin{aligned}
        \bar\partial_z F_0(w,z) = \bar\partial_z \frac{\kappastar(w,z)-\ii\kappa(w,z)}{2} = -\frac{\bar\alpha(z)}{2}\frac{\kappastarbar(w,z)-\ii\kappabar(w,z)}{2} = \frac{\bar\alpha(z)}{2}\bar F_1(w,z)\\
        \bar\partial_z F_1(w,z) = -\bar\partial_z \frac{\kappastar(w,z)+\ii\kappa(w,z)}{2} = \frac{\bar\alpha(z)}{2}\frac{\kappastarbar(w,z)+\ii\kappabar(w,z)}{2} = \frac{\bar\alpha(z)}
        {2}\bar F_0(w,z).
        \end{aligned}
    \end{equation}
    We check Equation~\eqref{eq:kappa:kappastar}.
    The first part of this equation is Equation~\eqref{eq:holomorphy:kappa} with $z \neq w$. 
    This is equivalent to the $(-\bar \alpha)$-Cauchy Riemann Equation~\eqref{eq:a:CR} 
    $$
    \begin{cases}
     \partial_{z_x} ( u(w,z)e^{ V(z)}) & =
     \partial_{z_y} ( v(w,z)e^{ V(z)}) \\
        \partial_{z_y} (u(w,z)e^{- V(z)}) & = -  \partial_{z_x} ( v(w,z)e^{- V(z)}) 
    \end{cases} \quad ; \quad w \neq z.
    $$
    To check the second part of Equation~\eqref{eq:kappa:kappastar}, we must prove that this also hold for $\kappastar$. 
    Recall that $u^*, v^*$ are obtained from $u, v$ by Equation~\eqref{eq:uu^*:zero}. 
    Since the derivatives in $z$ and $w$ commute
    $$
        \begin{aligned}
            \partial_{z_x}(e^{ V(z)}u^*(w,z)) 
        &= e^{- V(w)}\int_{\beta^*}^w e^{2 V(\omega)}\left\langle \ii\nabla_\omega(e^{- V(\omega)} \partial_{z_x}(e^{ V(z)}u(\omega,z)),\mathrm{d}\omega\right\rangle\\
        &= e^{- V(w)}\int_{\beta^*}^w e^{2 V(\omega)}\left\langle \ii\nabla_\omega(e^{- V(\omega)} \partial_{z_y} ( e^{ V(z)}v(\omega,z)),\mathrm{d}\omega\right\rangle\\
        &= \partial_{z_y} (e^{ V(z)}v^*(w,z))
        \end{aligned}
    $$
    along any fixed path from $\beta^*$ to $w$ avoiding $z$.
    The second part of the $\-\bar \alpha$-Cauchy-Riemann Equation works for the same reason, which gives the second part of Equation~\eqref{eq:kappa:kappastar}.


    We finally consider the diagonal behavior.
    Combining Lemma~\ref{lem:kappa:kappa_0} with the explicit expression of $\kappa_0$ and $\kappastar_0$, we obtain that
    \begin{equation}\label{eq:asymp:F0F1}
        F_0(w,z) = -\frac{\ii}{4\pi}\frac{g'(z)}{g(z)-g(w)}+O(\log|w-z|) \quad ; \quad F_1(w,z) = O(\log|w-z|).
    \end{equation}
    This directly implies the diagonal behavior 
    \begin{equation}\label{eq:diagonal:behavior}
        F_0(w,z) = \frac{\ii}{4\pi(w-z)}+O(\log|w-z|) \quad ; \quad F_1(w,z) = O(\log|w-z|)
    \end{equation}
    which is exactly Equation~\eqref{eq:singularity:F0F1}.
    It is standard to prove that the Dirac term in~\eqref{eq:F0F1_bulk} follows from the off-diagonal equalities~\eqref{eq:F0F1:offdiag:w} and~\eqref{eq:F0F1:offdiag:z} and the diagonal behavior \eqref{eq:diagonal:behavior}, by integration by parts and the residue theorem. 
    We detail it below for completeness.
    Let $B(w,\eps)$ denote the Euclidean ball of radius $\eps$ around $w$, and for $z \in \partial B(w,\eps)$ let $n(z)$ denote the outward unit normal vector.  
    Let $dV(z)$ and $dl(z)$ denote respectively the infinitesimal volume and perimeter.
    Note that along $\partial B(w,\eps)$, $n(z)dl(z) = dz$.
    Let $f$ be any smooth test function with zero boundary conditions in $\Omega$.
    Integrating by parts $\kappastar$ and $f$ (and integrating back after some algebraic manipulations) gives
    \begin{equation}\label{eq:proof:dirac:term}
        \begin{aligned}
            \int_{\Omega} F_0(w,z)\dz f(z)\mathrm{d}V(z)
            &= \lim_{\eps \to 0} \int_{\Omega \setminus B(w,\eps)}F_0(w,z)\dz f(z)\mathrm{d}V(z)\\
            &= \lim_{\eps \to 0} \left[\int_{\partial B(w,\eps)}F_0(w,z)f(z)n(z)\mathrm{d}l(z)-\int_{\Omega \setminus B(w,\eps)} \dzbar_z F_0(w,z)f(z)\mathrm{d}V(z)\right] \\
            &= \lim_{\eps \to 0} \left[\int_{\partial B(w,\eps)}\left(-\frac{\ii}{4\pi}\frac{g'(z)}{g(z)-g(w)}+O(\log|w-z|)\right)f(z)n(z)\mathrm{d}l(z)\right.\\
            &\qquad \left. +\frac12\int_{\Omega \setminus B(w,\eps)} \bar\alpha(z) \kappastarbar(w,z)f(z)\mathrm{d}V(z)\right] \\
            &= \lim_{\eps \to 0} \left[\frac{1}{4\ii\pi}\int_{\partial B(w,\eps)}\frac{g'(z)}{g(z)-g(w)}f(z)(-\ii)\mathrm{d}z + O(\eps \log(\eps)) \right]\\
            &\qquad +\frac12\int_{\Omega} \bar\alpha(z) \kappastarbar(w,z)f(z)\mathrm{d}V(z) \\
            &= -\frac{\ii}2f(w) +\frac12\int_{\Omega} \bar\alpha(z) \kappastarbar(w,z)f(z)\mathrm{d}V(z) 
        \end{aligned}
    \end{equation}
    which is the second part of Equation~\eqref{eq:kappa:kappastar}.
    The same argument works for the other equations.

We now verify the statement concerning the boundary conditions in Proposition \ref{lem:F0F1}. To obtain the boundary condition in the first variable, we simply notice that $\kappa(w,z) = 0$ when $w \in \partial \Omega$, which directly give the result.

    We turn to the boundary conditions in the second variable.
   By definition (Equation~\eqref{eq:green:distribution}), $\mGreen(w,z) = 0$ for $z \in \partial \Omega$.
    Hence, since $\mGreen$ vanishes in the tangent direction to the boundary, for $w \in \Omega, z \in \partial \Omega$, $\nabla_z \mGreen(w,z) = h_z(w) \eta(z)$ where $h_z : \Omega \to \RR$. Hence,
    $$
        \kappa(w,z) = h_z(w) \bar\eta(z) \quad ; \quad z \in \partial \Omega, w \in \Omega.
    $$
    Recall that for $z \in \Omega$, $\kappastar(\cdot, z)$ is defined as the unique $\alpha$-conjugate of $\kappa(\cdot,z)$.
    We claim that this also holds for $z \in \partial \Omega$ (this will be proved below).
    Since the $\alpha$-conjugate is defined coordinate by coordinate, if $h_z^* : \Omega \to \RR$ denotes the unique $\alpha$-conjugate of $h_z: \Omega \to \RR$, we simply have 
    $$
        \kappastar(w,z) = h_z^*(w)\bar\eta(z) \quad ; \quad z \in \partial \Omega, w \in \Omega.
    $$
    Recall the definition of $F_0,F_1$ from Equation~\eqref{eq:def:F_i}.
    For $z \in \partial \Omega$, the above implies
    $$
        F_0 = \bar\eta(z)\frac{h_z^*(w)-\ii h_z(w)}{2} \quad ; \quad F_1 = -\bar\eta(z)\frac{h_z^*(w)+\ii h_z(w)}{2}.
    $$
    Hence,
    $$
        F_1(w,z) 
        = -\frac{\bar\eta(z)}{\eta(z)}\bar F_0(w,z)
        = \frac{\bar\tau(z)}{\tau(z)}\bar F_0(w,z)
        = \bar\tau^2(z) \bar F_0(w,z).
    $$
    
    We now check that $\kappa(\cdot, z)$ and $\kappastar(\cdot,z)$ are also $\alpha$-conjugate when $z \in \partial \Omega$. 
    We already know that Equation~\eqref{eq:uu^*} holds for $z \in \Omega$, that is for all $w_1,w_2 \neq z$,
    \begin{equation}
        u^*(w_2,z)e^{ V(w_2)} - u^*(w_1,z)e^{ V(w_1)} 
        = \int_{w_1}^{w_2}e^{-2 V(w)}\Rot_{\pi/2}\nabla_v(u(w,z)e^{- V(w)})\mathrm{d}w.
    \end{equation} 
    To prove that this remains true at $z \in \partial \Omega$, we can apply a dominated convergence theorem. 
    For this, it suffices to prove that $u(w,z)$ and $\nabla_w u(w,z)$ are continuous functions in $\bar\Omega \times \bar \Omega \setminus \cD(\bar\Omega \times \bar \Omega)$, bounded away from the diagonal.  
    By definition, if we denote by $\partial_{w_1}, \partial_{w_2},\partial_{z_1}, \partial_{z_2}$ the partial derivatives with respect to the real and imaginary parts of $z$ and $w$,
    \begin{align}
        u(w,z) &= \Re \kappa(w,z) = \partial_{z_1} G^m(w,z) - \alpha_1(z) G^m(w,z),\\
        v(w,z) &= \Im \kappa(w,z) = - \partial_{z_2} G^m(w,z) + \alpha_y(z) G^m(w,z).
    \end{align}
    Hence, we only need to prove that for $i, j \in \{1,2\}$, $G^m(w,z)$, $\partial_{z_i}G^m(w,z)$ and $\partial_{w_j} \partial_{z_i} G^m(w,z)$ are continuous in $(z,w)$ and bounded away from the diagonal.
    
    (i) For $G^m(w,z)$, this is true by definition, and furthermore
    $$
        0 \leq G^m(w,z) \leq G(w,z) = O(\log|z-w|).
    $$

    (ii) The partial derivatives of the critical Green function in the half-plane are explicit:
    $$
        \partial_{z_x}G_{\HH}(w,z) = \frac{1}{2\pi}\left(\frac{z_x-w_x}{|z-w|^2}-\frac{z_x-w_x}{|z-\bar w|^2}\right)~;~\partial_{z_x}G_{\HH}(w,z) = \frac{1}{2\pi}\left(\frac{z_y-w_y}{|z-w|^2}-\frac{z_y+w_y}{|z-\bar w|^2}\right).
    $$
    Hence, for $i \in \{1,2\}$, $\partial_{z_i}G_{\HH}(w,z)$ is continuous in $(w,z)$ away from the diagonal with a $O(1/|z-w|)$ singularity.
    Since the critical Green function in a general domain $G$ is obtained from $G_\HH$ by Equation~\eqref{eq:Green:half-plane:domain}, this is also true for $G$.  
    Deriving the resolvent identity~\eqref{eq:resolvent:identity:continuous} with respect to the second variable, we obtain for $(w,z) \in \Omega \times \Omega \setminus \cD(\Omega \times \Omega)$,
    \begin{equation}\label{eq:resolvent:derivative}
        \partial_{z_i} G^m(w,z) = \partial_{z_i} G(w,z) + \int_{y \in \Omega}\mass(y)G^m(w,y) \partial_{z_i} G(y,z)\mathrm{d}y.
    \end{equation}
    By applying the dominated convergence theorem in Equation~\eqref{eq:resolvent:derivative}, we obtain that $\partial_{z_i} G^m(w,z)$ is continuous in $(w,z)$ away from the diagonal, with a $O(1/|z-w|)$ singularity. 

    (iii) Differentiating the resolvent identity with respect to the first variable, we obtain
    \begin{equation}\label{eq:resolvent:derivative:twice}
        \partial_{w_j}\partial_{z_i} G^m(w,z) = \partial_{w_j}\partial_{z_i} G(w,z) + \int_{y \in \Omega}\mass(y)\partial_{w_j}G^m(w,y) \partial_{z_i} G(y,z)\mathrm{d}y.
    \end{equation}
    Similarly as before, an explicit computation shows that $\partial_{w_j}\partial_{z_i} G(w,z)$ is continuous and bounded away from the diagonal. By symmetry of $G^m$ and (ii), $\partial_{w_j}G^m(w,y)$ is continuous and bounded away from the diagonal with a $O(1/|w-y|))$ singularity. Moreover, as we already proved above, $\partial_{z_i}G(y,z)$ is continuous and bounded away from the diagonal with a $O(1/|z-y|))$ singularity. 
    Hence, the integrand in the right-hand side of Equation~\eqref{eq:resolvent:derivative:twice} remains integrable for $(w,z)$ away from the diagonal and the dominated convergence theorem applies.
\end{proof}

\subsection{Overview of proof of Theorem \ref{T:SG}}
\label{SS:identification}

In this section we give an overview of what remains to be proved for Theorem \ref{T:SG}. We start by rephrasing Theorem~\ref{thm:height} and Proposition~\ref{lem:F0F1} as follows.  Let $\EE_{(\Omega, \alpha)}$ denote the expectation corresponding to the law of the limiting height function $\phi$, which is already known to exist as soon as $\Omega$ is bounded and simply connected and $\alpha$ derives from a log-convex potential which is ($C^1$) on $\bar \Omega$, by \cite{BHS}.
For a pair of function $(F_0,F_1)$ and $\mathbf{s} = (s_1, \ldots, s_n) \in \{0,1\}^n$, we call
    $$
    v_{(F_0,F_1)}^{(\mathbf{s})}(z_1, \ldots, z_k): = \det_{i \neq j}\Big[F_{s_i+s_j}^{(s_j)}(z_i,z_j)\Big] 
    $$ 

\begin{corollary}\label{cor:BV:alpha}
    Assume that $\Omega$ is bounded with a straight boundary, and that $\alpha$ derives from a log-convex potential which is smooth ($C^1$) on $\bar \Omega$.
    Then,
    \begin{itemize}
        \item $\BV_1(\Omega, \alpha)$ has a solution $(F_0,F_1)$ (provided by Theorem~\ref{thm:height}) 
        \item Then for $ \mathbf{s} = (s_1, \ldots, s_n) \in \{0,1\}^n$,
    \begin{equation}\label{eq:Wirt_height1}
        \EE_{(\Omega, \alpha)} [ \prod_{i=1}^n \dz^{(s_i)}_{z_i}\phi(z_i)] = v_{(F_0,F_1)}^{(\mathbf{s})}(z_1, \ldots, z_k) 
    \end{equation}
    \end{itemize}
\end{corollary}
Here, as before, $\partial^{(0)}$ and $\partial^{(1)}$ refer to the Wirtinger derivatives $\partial_z$ and $\bar\partial_z$ respectively. Of course, the limiting height function $\phi$ is rough, so the Wirtinger derivatives of $\phi$ do not make literal sense, and these are understood in the weak sense. 
Concretely, the meaning of \eqref{eq:Wirt_height1} is that for any smooth, compactly supported test function $f:\Omega \to \RR$ one can write
\begin{equation}\label{eq:momentsv}
\EE_{(\Omega, \alpha)} [ (\phi, f)^n] = \int \ldots \int f(z_1)\ldots f(z_n) V_{(F_0,F_1)} (z_1, \ldots, z_n) \mathrm{d}A(z_1) \ldots \mathrm{d}A(z_n)
\end{equation}
where $dA(z)$ denotes integration with respect to the area (Lebesgue) measure, and 
$$
    V_{(F_0,F_1)}(z_1, \ldots, z_n) = \sum_{\mathbf{s} \in \{0,1\}^n} \int_{\path_1}\cdots \int_{\path_n}  v_{(F_0,F_1)}^{(\mathbf{s})}(z_1, \ldots, z_n) \prod_{i=1}^n \mathrm{d}z_i^{(s_i)}.
$$
{Here, $\gamma_1, \ldots, \gamma_n$ are arbitrary paths that are disjoint, connecting $z_1, \ldots, z_n$ respectively to the straight edge portion of $\partial \Omega$}. 
The identification with the sine-Gordon field of Conjecture \ref{eq:conjBHS2} proceeds in four steps:
\begin{itemize}
\item First, we prove a \textbf{conformal change of coordinate} formula which, together with results from~\cite{BHS} implies that Corollary~\ref{cor:BV:alpha} holds for general domains

\item Next, we check that the determinant in Corollary~\ref{cor:BV:alpha} can be written in terms of \textbf{fermionic expressions}.
This involves performing a gauge change to transform the boundary conditions of the boundary value problem.
These fermionic expressions match the sine-Gordon field via the Coleman transform.


\item We finally explain how to check in concrete cases our conditions on the domain $\Omega$ and the vector field $\alpha$. 
Together with the Coleman transform of \cite{PVW}, this proves in particular that in the case of the unit disc, the limiting field that we identify is indeed the sine-Gordon model.
\end{itemize}

\paragraph{Step 1.}
We verify that Corollary~\ref{cor:BV:alpha} extends to more general domains with the conformal covariance rule described in \cite{BHS}.
This is an important step in the identification of the limiting height function as the sine-Gordon field for the following reason: on the one hand, our Theorem \ref{thm:height} requires the boundary of the domain to have a straight portion. On the other hand, the Coleman correspondence (on which we rely for this identification) is established in \cite{PVW} only in the case where the domain $\Omega = \DD$ is the unit disc. The result below will therefore allow us to compare these two results and eventually conclude that in the unit disc, the limiting moments of the height function match those of the sine-Gordon model.  

This change of domain relies on the following change of coordinates formula.
Let $T:\Omega \to T \Omega$ be a conformal isomorphism.
Denote by  
\begin{equation}\label{eq:covariance}
     T\alpha(w) = \overline{\psi'(w)} \alpha ( \psi(w)). 
\end{equation}
the change of drift vector under the conformal covariance rule of \cite{BHS}.
    
\begin{prop}
    \label{thm:covariance}
    Let $\Omega$ be any simply connected domain.
    Let $T:\Omega \to T\Omega$ be a conformal isomorphism and let $\psi = T^{-1}$.
    For $(F_0,F_1)$, let 
    $$
    TF_s (w_1, w_2) = \psi'(w_2) F_s(\psi(w_1), \psi(w_2)); \quad s\in \{0,1\}. 
    $$ 
    Then, $(F_0, F_1)$ solves $\BV_1(\Omega,\alpha)$ if and only if $(TF_0, TF_1)$ solves $\BV_1(T\Omega, T\alpha)$.
    When that is the case, for all $w_1, \ldots, w_k \in T\Omega$ which are pairwise disjoint,
    $$
    V_{(F_0,F_1)} (\psi(w_1), \ldots, \psi(w_k)) = V_{(TF_0, TF_1)} (w_1, \ldots, w_k).
    $$
\end{prop}
The proof of this Proposition is given in Section~\ref{SS:covariance}.

\begin{remark}
    Note that the identity \eqref{eq:covariance} matches exactly (1.23) in the known conformal covariance of the limiting near-critical dimer height function (Theorem 1.7 of \cite{BHS}).
\end{remark}

Recall that by Theorem 1.7 in \cite{BHS}, if $T: \Omega \to T\Omega$ is a conformal isomorphism of bounded simply connected domains with inverse $T^{-1}$, and $T$ extends analytically to a neighbourhood of $\Omega$ (which we will henceforth denote as "smooth"), 
then the height function $\phi^{\Omega,\alpha}$ with law $\PP_{(\Omega, \alpha)}$ and the height function $\phi^{T\Omega,T\alpha}$ with law $\PP_{(T\Omega, T\alpha)}$ are related by
\begin{equation}\label{eq:thm:BHS}
    \phi^{\Omega, \alpha} \circ \psi = \phi^{T\Omega,T\alpha}
\end{equation}
Together, Corollary~\ref{cor:BV:alpha}, Proposition~\ref{thm:covariance} and \eqref{eq:thm:BHS} directly imply the following
\begin{corollary}\label{cor:BV:alpha:conformal}
    Assume that $\Omega$ is the image by a smooth conformal transformation of a bounded domain with a straight boundary, and that $\alpha$ derives from a log-convex potential which is smooth ($C^1$) on $\bar \Omega$.
    \begin{itemize}
        \item $\BV_1(\Omega, \alpha)$ has a solution $(F_0,F_1)$ (provided by Theorem~\ref{thm:height} and the conformal change of coordinate)
        \item For $ \mathbf{s} = (s_1, \ldots, s_n) \in \{0,1\}^n$,
    \begin{equation}\label{eq:Wirt_height}
        \EE_{(\Omega, \alpha)} [ \prod_{i=1}^n \dz^{(s_i)}_{z_i}\phi(z_i)] = v_{(F_0,F_1)}^{(\mathbf{s})}(z_1, \ldots, z_k) 
    \end{equation}
    \end{itemize}
\end{corollary}

\paragraph{Step 2.} The next step gives a fermionic representation for the correlation function $v^{(\mathbf{s})}_{(F_0,F_1)}$. 
In order to do this, we check that they coincide with the correlation functions $\chi^{(\mathbf{s})}_{SG(\rho|\Omega)}$ of the mixed Wirtinger derivatives of the sine-Gordon introduced in Equation~\eqref{eq:def:correlations:SG}.
Recall the definition of the correlations of the mixed Wirtinger derivatives of the sine-Gordon introduced in Equation \eqref{eq:def:correlations:SG}.
\begin{equation}
    \chi^{(\mathbf{s})}_{\mathrm{SG}(\rho)}(z_1, \ldots, z_n) 
    = \Big\langle \prod_{j=1}^n \partial^{(s_j)}_{z_j} \phi(z_j) \Big\rangle_{\mathrm{SG}(\rho|\Omega)}.
\end{equation}
The proof is based on a change of gauge compared to the expression for $(F_0, F_1)$ appearing in Corollary \ref{cor:BV:alpha:conformal} and some rather surprising exact identities. 
We first introduce a modified electromagnetic field  
\begin{equation}\label{eq:alphatilde}
\tilde \alpha(x) = \alpha(x) e^{\mathbf{i} \arg g'(x)}.
\end{equation}
where $g: \Omega \to \HH$ is a conformal isomorphism, and $\arg g'$ is any determination of the argument.
The significance of this modified vector field (and how it arises naturally from Conjecture \eqref{eq:conjBHS2}) will be explained in Section \ref{SS:Coleman}.
Define the change of gauge 
\begin{equation}
    \label{eq:gaugechangeF}
\tilde F_0(w,z) = \sqrt{\frac{ g'(w)}{g'(z)}} F_0(w,z); \quad  \tilde F_1(w,z) = \sqrt{\frac{\overline{g'}(w)}{g'(z)}} F_1(w,z).
\end{equation}
We define a new boundary value problem, $\BV_2(\Omega, \tilde\alpha)$:
\begin{equation}\label{eq:BV2:alpha}
\begin{cases}
\bar \partial_w F_0(w,z) 
        = \frac{\tilde\alpha(w)}{2} F_1(w,z) + \frac{\ii}{2}\delta_w(z) 
        \quad ; \quad 
        \partial_w F_1(w,z) 
        = \frac{\overline{\tilde{\alpha}}(w)}{2}F_0(w,z)
        &\text{ on }\Omega^2\\
\tilde F_1(w,z) = -\tau(w) \tilde F_0(w,z) ; & \text{ if } w \in \partial \Omega.
\end{cases}
\end{equation}
where as before $\tau(w)$ is the tangent vector at $w \in \partial \Omega$ such that $\Omega$ lies locally to the left of $\tau(w)$.
Note that only the boundary conditions changes compared to $\BV_1(\Omega,\alpha)$.
Again, the first two lines can be conveniently rewritten in terms of the Dirac operator
$$
\slashed{\partial}_{\tilde\alpha} 
\begin{pmatrix}
    \ii \tilde{F}_1 & \ii \overline{\tilde F_0}\\
    -\ii \tilde F_0 & -\ii \overline{\tilde F_1}  
\end{pmatrix} 
= I \text{ on }  \Omega^2
$$

\begin{lemma} \label{lem:detgaugechange}
    $(F_0,F_1)$ satisfies $\BV_1(\Omega, \alpha)$ if and only if $(\tilde F_0, \tilde F_1)$ satisfies $\BV_2(\Omega, \tilde \alpha)$.  
    Moreover,
    \begin{equation}\label{eq:reform_thmheight2}
        \det_{i \neq j} [ F^{(s_j)}_{s_i + s_j} (z_i, z_j)] = \det_{i \neq j} [\tilde F^{(s_j)}_{s_i + s_j} (z_i, z_j)].
    \end{equation}   
\end{lemma}
This lemma will be proved in Section~\ref{SS:fermionic}.

\medskip 
As a summary of the first two steps, we have already proved that 
$$
    v_{(F_0,F_1)}^{(\mathbf{s})}(z_1, \ldots, z_k) = \EE_{(\Omega, \alpha)} [ \prod_{i=1}^n \dz^{(s_i)}_{z_i}\phi(z_i)]  =  \det_{i \neq j}\Big[F_{s_i+s_j}^{(s_j)}(z_i,z_j)\Big] 
$$ 
where $(F_0, F_1)$ satisfy $\BV_1(\Omega, \alpha)$.
By Lemma~\ref{lem:detgaugechange}, we find that this is equivalent to
\begin{equation}\label{eq:fermionic_inverseDirac}
v^{(\mathbf{s})}_{(F_0,F_1)} (z_1, \ldots, z_n) 
= \det_{i \neq j} [\tilde F^{(s_j)}_{s_i + s_j} (z_i, z_j)]
\end{equation}
with $(\tilde F_0, \tilde F_1)$ the unique solution to $\BV_2(\Omega, \tilde\alpha)$.
It only remains to connect this expression with the fermionic expression for correlations of the sine-Gordon obtained in~\cite{PVW} to conclude: this is done in Section~\ref{SS:identificationPVW}.

We will as a result be able to make a connection to the sine-Gordon model, and more precisely to its Coleman form, i.e., to its expression as a Thirring model, as we now explain. In this model the ``spins'' are fermionic variables   $(\psi(z))_{z\in \Omega}$ where for each $z$, $\psi(z) = (\psi_1(z), \psi_2(z))$ is a vector of size two, whose formal density is given by 
 \begin{equation}\label{eq:fermion_density}
 \exp\Big[-\int_{\Omega}\mathrm{d}x \ \big( \bar\psi \slashed \partial \psi +\tilde \alpha \bar\psi_1\psi_1+\overline{\tilde \alpha} \bar\psi_2\psi_2 \big) \Big] .
 \end{equation}
 See \cite{BenfattoFalcoMastropietro_Thirring} for a rigorous construction (this is a particular case of the so-called \textbf{Thirring model}). We call this field the \textbf{Coleman form of the sine-Gordon with electromagnetic field $\tilde\alpha$} ($\CSG(\tilde\alpha)$ for short).


\begin{theorem}
    \label{thm:fermionic} 
    Assume that $\Omega$ is the image by a smooth conformal transformation of a bounded domain with a straight boundary, and that $\alpha$ derives from a log-convex potential which is smooth ($C^1$) on $\bar \Omega$.
    \begin{itemize}
        \item  {$\BV_2(\Omega, \tilde \alpha)$} has a solution (provided by Theorem~\ref{thm:height}, the conformal change of coordinate and the above gauge change)
        \item If the solution to  $\BV_2(\Omega, \tilde \alpha)$ is unique, the correlation kernels $v^{(\mathbf{s})}$ of holomorphic and antiholomorphic derivatives  of the height function $(\Omega, \alpha)$ 
        agree with the Majorana fermionic correlations of the above field:
        for $\mathbf{s}\in \{0,1\}^n$, letting $\lambda = -1/ (4\sqrt{\pi})$,    \begin{equation}\label{eq:fermionic}
v^{(\mathbf{s})} (z_1, \ldots,z_n) =         \EE_{(\Omega, \alpha)} [ \prod_{i=1}^n \dz^{(s_i)}_{z_i}\phi(z_i)]  
        = 
        \lambda^{n}  \det\left(  \Big\langle \psi^{s_i}(z_i), \psi^{s_j} (z_j)\Big\rangle\right)_{1\le i, j\le n}
    \end{equation}
    where $\psi^0 $ stands for $\psi$, and $\psi^1$ for $\psi^*$, defined as $\psi = \psi_1 + \ii \psi_2$ and $\psi^*  = \psi_1 - \ii \psi_2$. 
    Alternatively, $\Big\langle \psi^{s_i}(z_i), \psi^{s_j} (z_j)\Big\rangle$ is the unique solution of the boundary value problem $\BV_3(\Omega,  \theta)$, with $\theta = \tilde \alpha/(2\ii)$ defined in \eqref{eq:BV3:alpha}.

\item If furthermore, $\Omega = \DD$ and $\theta = \tilde \alpha / (2\ii)$ is real-valued and in $C^\infty (\Omega)$, then 
$$
v^{(\mathbf{s})} (z_1, \ldots,z_n) =         \EE_{(\Omega, \alpha)} [ \prod_{i=1}^n \dz^{(s_i)}_{z_i}\phi(z_i)]  
      =\lambda^n \Big\langle \prod_{j=1}^n \partial^{(s_j)} \phi(z_j))\Big\rangle_{\mathrm{SG}(-\frac{4\theta}{\pi}|\Omega)}
$$
    \end{itemize}
    
\end{theorem}

\paragraph{Step 3.}
The above result Theorem~\ref{thm:fermionic} only applies under some assumption on $\Omega$ and $\alpha$ which may look hard to check in practice.
We provide an explicit criterion for the uniqueness of the boundary value problem, and explain how this can be applied in practice.
Following a similar argument as that found in \cite{BauerschmidtMasonWebb}, we observe that the solutions to the boundary value problem $\BV_2(\Omega, \tilde \alpha)$ (or equivalently $\BV_1(\Omega, \alpha)$ by Lemma~\ref{lem:detgaugechange}) are unique.
\begin{lemma}\label{lem:unicity}
    Let $\Omega$ be an arbitrary (not necessarily bounded) simply connected domain, and $\tilde\alpha$ be a $C^1$ vector field on $\bar \Omega$.
    Assume that $\Im(\tilde\alpha) \leq 0$ and that the support of $\Im(\alpha)$ is not empty.
    Then if $\BV_2(\Omega, \tilde\alpha)$ has a solution, it is unique.
\end{lemma}
This lemma will be proved in Section~\ref{SS:uniqueness:BV}.
This criterion may appear arbitrary, and we believe that uniqueness in the boundary value problem holds under more general conditions.
However, this will be enough to check the assumption of Theorem~\ref{thm:fermionic} in many cases in combination with Proposition~\ref{thm:covariance}, Lemma~\ref{lem:detgaugechange}, and the fact that when $\Omega$ is the half-plane, $\tilde \alpha = \alpha$.
In particular, we can prove Theorem~\ref{T:SG}.

\begin{proof}[Proof of Theorem~\ref{T:SG}]
    Let $\Omega = \DD$, $T: \Omega \to \HH$ with inverse $g$ and $V_\HH, V = V_\HH \circ g, \alpha_\HH = \nabla V_\HH, \alpha = \nabla V$ be as in the statement of the theorem.
    Note that 
    \begin{equation}\label{eq:alpha:disk}
        \alpha(z) = \nabla V_\HH \circ g = \overline{g'(z)} V(g(z)).
    \end{equation}
    We check the assumptions of Theorem~\ref{thm:fermionic}.
    First of all, $\alpha$ derives from the log-convex potential $V$.

    In the half-plane, uniqueness holds for $\BV_2(\HH,\tilde\alpha)$ by Lemma~\ref{lem:unicity} since $\tilde\alpha_\HH = \alpha_\HH$ has non-positive imaginary part by our assumptions on $V_\HH$.
    By Lemma \ref{lem:detgaugechange}, uniqueness also holds for $\BV_1(\HH, \alpha_\HH)$ since $\tilde \alpha_\HH = \alpha_\HH$.
    By Proposition~\ref{thm:covariance}, uniqueness also holds for $\BV_1(\Omega, \alpha)$.

    Finally, we check that $\tilde\alpha$ is imaginary-valued. By definition of $\tilde \alpha$ and Equation~\eqref{eq:alpha:disk},
    $$
        \tilde\alpha(z) 
        = e^{\ii \arg g'(z)} \alpha(z)
        = e^{\ii \arg g'(z)} \overline{g'(z)} \alpha_\HH(g(z))
        = |g'(z)|\alpha_\HH(g(z)),
    $$
    where $\alpha_\HH(g(z))$ is pure imaginary. This concludes the proof under the assumption that $\alpha$ satisfies \eqref{eq:Vcond}. 

    Let us briefly explain what happens under the assumptions of Remark \ref{R:alternate_alpha} instead, i.e., when $V_\DD (z) = f ( \Im (g(z))$, where $f$ is as in Theorem \ref{T:SG}, and $\alpha(z) = \nabla V_\DD(z), z \in \DD$. While $\alpha$ derives from a potential by definition, there are several that we need to check:
    \begin{itemize}
        \item $\tilde \alpha$ is purely imaginary, with negative imaginary part. 
        
        \item the potential $V_\DD$ is log-convex. 
    \end{itemize}
    In fact these two aspects are closely related to one another, and we briefly explain how to verify this. We start with the verification of the first point. Let $\tilde f (w) = f ( \Im (w)), w \in \HH$. Then $V_\DD(z) = \tilde f \circ g$. Thus another expression for $\alpha$ is $\alpha = 2 \bar \partial V_\DD$. Applying the chain rule for Wirtinger derivatives, we find, letting $w = g(z)$:
    \begin{align*}
        \bar \partial V_\DD(z) & = \partial \tilde f (w) \cdot  \bar \partial g (z) + \bar \partial \tilde f (w)  \cdot\bar \partial \bar g(z) \\
        & = 2 \nabla \tilde f (w) \overline{ g'(z)}
    \end{align*}
(since $g$ is holomorphic, the first term in the right hand side of the first equality is zero).   Now, $\tilde f$ depends only on the imaginary part of $w$ so its gradient is purely imaginary, in fact given by $\ii f'(\Im(w))$. We deduce that 
\begin{equation}\label{eq:alphaVD}
\alpha (z) = \ii f'(\Im (w)) \overline{ g'(z)}.
\end{equation}
Therefore, 
$$
\tilde \alpha(z) = e^{ \ii \arg g'(z)} \alpha(z) = \ii |g'(z)| f' ( \Im (w)).
$$
In particular, $\tilde \alpha$ is purely imaginary, with negative imaginary part, as desired in the first point. 

For the second part, we need to check if $\Delta V_\DD + |\nabla V_\DD|^2 \ge 0$. We already computed $\nabla V_\DD$ so let us compute 
\begin{align*}
\Delta V_\DD &= \tfrac14 \partial \bar \partial V_\DD\\
& = \tfrac12 \partial \Big( \ii f'(\Im (g(z))) \overline{g'(z)}\Big)
\end{align*}    
where we used \eqref{eq:alphaVD}. Since $\bar g'$ is antiholomorphic, we have 
$$
 \partial \Big( \ii f'(\Im (g(z))) \overline{g'(z)}\Big) = \ii \overline{g'(z)} \partial \Big( f'(\Im(g(z))) \Big)
 $$
 The above holmorphic derivative can be computed exactly in the same manner as above; in fact, since $f'$ is also a real valued function we can simply take the conjugate in \eqref{eq:alphaVD} and replace $f$ by $f'$. We obtain:
 \begin{equation}
\Delta V_\DD (z) = |g'(z)|^2 f''( \Im g(z))     \end{equation}
so that altogether
$$
\Delta V_\DD (z) + |\nabla V_\DD(z)|^2 = |g'(z)|^2 \Big( f'' + (f')^2 \Big) (\Im g(z)) \ge 0
$$
since we assumed that $f$ is log-convex (i.e., $(e^f)'' \ge 0$). This concludes the proof of Theorem \ref{T:SG} also in this case. 
\end{proof}

\subsection{Change of coordinate formula}

\label{SS:covariance}

\begin{proof}[Proof of Proposition \ref{thm:covariance}.]
    Recall that $T: \Omega \to T\Omega$ is a conformal isomorphism and that we define the functions $TF_0, TF_1$ in $T\Omega$ by 
    $$
    TF_s (w_1, w_2) = \psi'(w_2) F_s(\psi(w_1), \psi(w_2)); \quad s\in \{0,1\}. 
    $$
    Let us fix $s_1, \ldots, s_k \in \{0,1\}$. For $1\le i \le k$, let $\tilde \gamma_i = T(\gamma_i)$ denote the paths in $T\Omega$ obtained by mapping $\gamma_i $ through $T$. Then by change of variable $w_i = \psi(z_i)$, one immediately has
    $$
\int_{\path_1}\cdots \int_{\path_k}  \det_{i \neq j}\Big[F_{s_i+s_j}^{(s_j)}(z_i,z_j)\Big] \prod_{i=1}^k \mathrm{d}z_i^{(s_i)} = 
\int_{\tilde \gamma_1}\cdots \int_{\tilde \gamma_k}  \det_{i \neq j}\Big[F_{s_i+s_j}^{(s_j)}(\psi(w_i),\psi(w_j))\Big]  \prod_{i=1}^k \psi'(w_i)^{(s_i)}\mathrm{d}w_i^{(s_i)}
    $$
    Using the multilinearity of the determinant (columnwise) we can rewrite this multiple integral as 
    $$
    \int_{\tilde \gamma_1}\cdots \int_{\tilde \gamma_k}  \det_{i \neq j}\Big[\psi'(w_j)^{(s_j)}F_{s_i+s_j}^{(s_j)}(\psi(w_i),\psi(w_j))\Big]  \prod_{i=1}^k \mathrm{d}w_i^{(s_i)} = \int_{\tilde \gamma_1}\cdots \int_{\tilde \gamma_k}  \det_{i \neq j}\Big[TF_{s_i+s_j}^{(s_j)}(w_i,w_j)\Big]  \prod_{i=1}^k \mathrm{d}w_i^{(s_i)}
    $$
    When summing over the signs $s_1, \ldots, s_k \in \{0,1\}$ we get a function of $w_1, \ldots, w_k$ which has the right formal expression. 
    
    It remains to verify that $TF_0, TF_1$ solves $\BV_1(T\Omega,T\alpha)$.
    It is straightforward to check that $TF_s$ has the correct diagonal behavior, that is the second equation of $\BV_1(T\Omega,T\alpha)$: it follows from the diagonal behavior of $F_s$, the definition of $TF_s$ and the analycity of $T$.
    Now, we verify that $TF_0, TF_1$ satisfy the first equation of $\BV_1(T\Omega,T\alpha)$. Consider for instance $TF_0$. Then 
    \begin{align*}
\bar\partial_{w} TF_0 (w, z) & = \bar\partial_{w} [ \psi'(z) F_0 ( \psi(w), \psi(z)] \\
& = \psi'(z) \bar\psi'(w) \bar\partial_{w} F_0( \psi(w), \psi(z))\\
& = \psi'(z) \bar\psi'(w) \left[\frac{\alpha(w)}{2} F_1 ( \psi(w), \psi(z)) + \frac{\ii}{2}\delta_{\psi(w)}(\psi(z)) \right]\\
& = \psi'(w) \frac{\bar \alpha (w)}{2} TF_1(w, z) + \frac{\ii}{2}|\psi'(w)|^2\delta_{\psi(w)}(\psi(z))\\
& = \frac{T\alpha(w)}{2} TF_1( w, z) + \frac{\ii}{2}|\psi'(w)|^2\delta_{\psi(w)}(\psi(z))\\
& = \frac{T\alpha(w)}{2} TF_1( w, z) + \frac{\ii}{2}\delta_{w}(z),
    \end{align*}
    where at the last line we use that $|\psi'(w)|^2\delta_{\psi(w)}(\psi(z)) = \delta_{w}(z)$ due to the change of variable formula, since these equations are understood in the weak sense.
    The other equation is checked similarly.
    
    To conclude the proof of Proposition \ref{thm:covariance}, we only need to check the boundary conditions, that is the third equation of $\BV_1(T\Omega,T\alpha)$. We simply write, for $w \in \partial \Omega$,
    $$
        TF_1(w,z) 
        = \psi'(z) F_1(\psi(w),\psi(z))
        = -\psi'(z) F_0(\psi(w),\psi(z))
        = - TF_0(w,z).\qedhere
    $$
\end{proof}



\subsection{Gauge change in boundary value problem (Lemma \ref{lem:detgaugechange})}

\label{SS:fermionic}



\begin{proof}[Proof of Lemma~\ref{lem:detgaugechange}]
First of all, it is straightforward to check that the gauge change does not change the diagonal behavior, thus $(F_0, F_1)$ satisfies the second equation of $\BV_1(\Omega, \alpha)$ if and only if $(\tilde F_0, \tilde F_1)$ satisfies the second equation of $\BV_2(\Omega, \tilde\alpha)$.

We now check that $(F_0, F_1)$ satisfies  \eqref{eq:F0F1_bulk} in $\Omega$ if and only if $(\tilde F_0, \tilde F_1)$ satisfies the same equation in $\Omega$ with $\alpha$ replaced by $\tilde \alpha$. For instance, since $g'$ is analytic (so that $\bar \partial_w g'(w) = 0$),
\begin{align*}
    \bar \partial_w \tilde F_0(w,z) & = \bar \partial_w \left\{\sqrt{\frac{ g'(w)}{g'(z)}} F_0(w,z) \right\}\\
    & = \sqrt{\frac{ g'(w)}{g'(z)}} \bar \partial_w F_0(w,z) \\
    & = \sqrt{\frac{ g'(w)}{g'(z)}} \left( \frac{\alpha(w)}{2} F_1( w,z) + \frac{\ii}{2} \delta_w(z) \right)\\
    & = \sqrt{\frac{g'(w)}{\bar g'(w)}} \frac{\alpha(w)}{2} \tilde F_1(w,z) + \frac{\ii}{2} \delta_w(z) \\
    & = \frac{\tilde \alpha(w)}{2} \tilde F_1(w,z) + \frac{\ii}{2} \delta_w(z) .
\end{align*}
In the last line, we used the fact that 
$$
\frac{g'(w)}{\bar g'(w)} = e^{  2\ii \arg g'(w)}
$$
as well as the definition of $\tilde \alpha(w)$ in \eqref{eq:alphatilde}. The other three equations can be checked the same way.

To check the boundary conditions, we simply observe that by definition of $\tilde F_1, \tilde F_0$, and using the boundary conditions satisfied by $F_1, F_0$ (see again Proposition \ref{lem:F0F1}), we have 
    for $w \in \partial\Omega, z \in \Omega$:
    \begin{align*}
        \tilde F_1(w,z) & = -\sqrt{\frac{\bar g'(w)}{g'(z)}} \tilde F_0 (w, z)
        = -\tau(w) \tilde F_0(w,z),
    \end{align*}
    using the fact that $g^{-1}: \HH \to \Omega$ gives a parameterization of the boundary $\partial \Omega$  (thus $1/ g'(w)$ is proportional to $\tau(w)$).
    (Although this is not needed in the following, one can also observe that in the other variable $z$, the boundary condition reads 
    \begin{equation}\label{eq:BV2z}
    \tilde F_1(w,z) = - \bar\tau (z) \overline{\tilde F}_0(w,z) ; \ z \in \partial \Omega.
    \end{equation}
Thus $(F_0, F_1)$ satisfies BV$_1(\Omega, \alpha)$ if and only if $(\tilde F_0, \tilde F_1)$ satisfies BV$_2(\Omega, \tilde \alpha)$. 

For the next item, the starting point is the basic observation that for a matrix $K$ and a nonzero complex valued function $f$, \begin{equation}\label{eq:basicdet}
        \det [K(x_i, x_j)]_{1\le i,j\le n} = \det \left[ \frac{c(x_i)}{c(x_j)} K(x_i, x_j)\right]_{1\le i,j\le n}.
    \end{equation}
    (This follows immediately from the multilinearity of the determinant which implies that in the right hand side we pick up a term $c(x_i)$ for each line and another ter $1/c(x_j)$ for each column, so the determinant is unchanged overall).

    Here the transformation we apply to $F$ is not obviously of this form. However, note that if $s \in \{0,1\}$ then we can write compactly 
    \begin{align}
    \tilde F_s (w,z) = \sqrt{\frac{g'(w)^{(s)}}{g'(z)}} F_s(w,z).\label{FtildeinF}
    \end{align}
    Thus, if
    $s_i, s_j$ are given then 
    \begin{align*}
        \tilde F^{(s_j)}_{s_i + s_j} (z_i, z_j) & = \left\{\sqrt{\frac{g'(z_i)^{(s_i+s_j)}}{g'(z_j)}}\right\}^{(s_j)} F_{s_i + s_j}^{(s_j)} ( z_i, z_j)\\
        & = \sqrt{\frac{g'(z_i)^{(s_i)}}{g'(z_j)^{(s_j)}}} F_{s_i +s_j}^{(s_j)} (z_i, z_j)
    \end{align*}
    (consider separately the case $s_j = 0$ and $s_j =1$ to check the last identity). We can thus apply \eqref{eq:basicdet} to conclude with the proof of Lemma \ref{lem:detgaugechange}.
        \end{proof}

\subsection{Uniqueness of boundary value problem (Lemma~\ref{lem:unicity})}\label{SS:uniqueness:BV}
\begin{proof}
Assume that $\BV_2(\Omega, \tilde \alpha)$ has two solutions $(\tilde F_0, \tilde F_1)$ and $(\tilde G_0, \tilde G_1)$.
Let $z \in \Omega$ be fixed, and denote $u_i(w) = \tilde F_i(w,z) - \tilde G_i(w,z)$ for $i = 0,1$.
We prove that $u_0 = u_1 = 0$. 
First note that $u_0,u_1$ satisfy $\BV_2(\Omega, \tilde\alpha)$ without the Dirac term, thus is smooth by elliptic regularity theory.
By definition of $\BV_2(\Omega, \tilde \alpha)$, $u_0(w)=-\tau u_1(w)$ for $w \in \partial\Omega$ and, for $w \in \Omega$,
\begin{align}
\begin{pmatrix}  \alpha & 2\bar\partial \\ 2\partial & \overline{ \alpha}\end{pmatrix}\begin{pmatrix}u_0\\ u_1
\end{pmatrix} =0.
\end{align}
First note that
\begin{align*}
    4\Im(\partial (u_0\bar u_1))&= \frac{2}{\ii}\Big(\partial(u_0\bar u_1) - \bar\partial(\bar u_0u_1)\Big).
 \end{align*}
Furthermore, 
\begin{align*}
 2\partial ( u_0 \bar u_1) &= 2\bar u_1 \partial u_0 + 2u_0 \partial \bar u_1 \\   
 &  = - \bar \alpha u_1 \bar u_1  - \bar \alpha \bar u_0 u_0 \\
 & = - \bar \alpha (|u_0|^2+ |u_1|^2).
\end{align*}
 Likewise, 
 \begin{align*}
     \bar\partial(\bar u_0u_1) & = - \alpha ( |u_0|^2+ |u_1|^2),
 \end{align*}
 so that 
\begin{align}
    4\Im(\partial (u_0\bar u_1))& = \frac1{\ii} \Big( - \bar \alpha + \alpha\Big) ( |u_0|^2+ |u_1|^2)\\
    & = 2 \Im (\alpha) ( |u_0|^2+ |u_1|^2).
\end{align}
Now, by Green's theorem, if $f$ is a $C^1$ function on $\bar \Omega$,  
 $\int_\Omega \partial f(w) \mathrm{d}A(w)=\frac{1}{2\ii}\int_{\partial \Omega}f(w)\mathrm{d}\bar w$. 
 We therefore have:
\begin{align}
     \int_{\Omega}\Im(\alpha)(|u_0|^2+|u_1|^2) \mathrm{d}A(w)
     &= \Im \left( 2 \int_\Omega \partial (u_0 \bar u_1) \mathrm{d}A(w) \right) \nonumber \\
     & = \Im \left( \frac1{\ii} \int_{\partial \Omega} u_0(w) \bar u_1 (w) \mathrm{d}\bar w\right )\nonumber\\
     &=\Im \left( \frac1{\ii} \int_{\partial \Omega} - \tau(w) u_1(w) \bar u_1(w) \mathrm{d} \bar w \right) \nonumber\\
     & = \Im \left( \ii \int_0^{1} \tau (\gamma(s) ) |u_1(\gamma(s))|^2  \overline{ \gamma'(s)} \mathrm{d}s \right)   
     \ge 0\label{eq:greens:theorem}
 \end{align}
 where $\gamma(s), s\in [0,1]$ is a parametrisation (in the anticlockwise direction) of $\partial \Omega$, and  in the last line, we use the fact that $\gamma'(s)$ is a positive multiple of $\tau(s)$.

Since $\Im (\alpha) <0$ in its support, this shows that 
%
 $u_0=u_1=0$ in the support of $\Im(\alpha)$. 
 Outside of the support of $\Im(\alpha)$, $u_0$ and $u_1$ are holomorphic/anti-holomorphic, and thus also vanish on the complement of the support. 
 \end{proof}
 

\subsection{Identification with sine-Gordon: proof of Theorem \ref{thm:fermionic}}

\label{SS:identificationPVW}
We start with giving more details about 
the main theorem of \cite{PVW} (their Theorem (1.1)), which is a bosonisation theorem for the above observables together with further trigonometric functions. Since we only need a subset of their results, we state here for convenience the part which is needed for our purposes. Fix functions $(f_j)_{1\le j \le p}$ and $(g_{j'})_{1\le j'\le q}$ as in Theorem \ref{thm:PVWintro}.

\begin{theorem}\label{thm:PVW}
(Part of Theorem 1.1 and Proposition 6.16 in \cite{PVW}, $U = V = A = \emptyset$ in their notation). Let $\Omega = \mathbb{D}$ be the unit disc and let $\theta:\Omega \to \RR$ be a \emph{real-valued} smooth function, i.e., $\theta \in C_c^\infty(\Omega)$. Then the correlations of the derivatives of the sine-Gordon field are given by 
\begin{align}\label{eq:PVWtheorem}
&\Big\langle \prod_{j=1}^p (2\ii\sqrt{\pi}\partial \phi(f_j))\prod_{j'=1}^q (-2\ii\sqrt{\pi}\bar\partial \phi(f'_{j'}))\Big\rangle_{\mathrm{SG}(-\frac{4\theta}{\pi}|\Omega)} \nonumber \\
= & \int\Big \langle \prod_{j} \psi\tilde\psi(z_j)\prod_{j'}\psi^*\tilde\psi^*(w_{j'})\Big \rangle_\theta\prod_{j}f(z_j)\prod_{j'}f'(w_{j'})\mathrm{d}z\mathrm{d}w.
\end{align}
Here, $\psi $ and $\tilde \psi$ are independent and the right hand side may be defined by
\begin{align}\label{eq:PVW_Prop6_16}
\Big \langle \prod_{j} \psi^{s_j}(z_j)\Big \rangle_\theta: = \mathrm{Pf}\Big [\langle\psi^{s_i}(z_i)\psi^{s_j}(z_j)\rangle_\theta\Big]_{i,j}
\end{align}
where, compared to \eqref{eq:PVWtheorem}, we identify $s_j=0$ with no $^*$ and $s_j=1$ with $^*$, i.e., $\psi^{s_j} = \psi$ when $s_j =0$, and $\psi^{s_j} = \psi^*$ when $s_j=1$.
\end{theorem}
(In \cite{PVW}, \eqref{eq:PVW_Prop6_16} is a property proved in Proposition 6.16 rather than a definition.) Since $\tilde\psi$ denotes an independent copy of $\psi$, note that an equivalent formulation of \eqref{eq:PVWtheorem} is: 
\begin{align}
\Big\langle \prod_j (2\ii^{(s_j)}\sqrt{\pi}\partial^{(s_j)} \phi(f_j))\Big\rangle_{SG(-\frac{4\theta}{\pi}|\Omega)}&=\int\text{Pf}^2\Big [\langle\psi^{s_i}(z_i)\psi^{s_j}(z_j)\rangle_\theta\Big]_{i,j}\prod_jf_j(z_j)\mathrm{d}z_j^{(s_j)}\nonumber \\&=\int\det\Big [\langle\psi^{s_i}(z_i)\psi^{s_j}(z_j)\rangle_\alpha\Big]_{i,j}\prod_jf_j(z_j)\mathrm{d}z_j^{(s_j)}\label{eq:psiSG1}
\end{align}
since fermions anticommute. 

The two-point Ising correlators $\langle\psi(z_i)\psi(z_j)\rangle, \langle\psi(z_i)\psi^*(z_j)\rangle$ are defined in Definition 6.4 in \cite{PVW} and can be conveniently studied in terms of a real-fermion spinor, defined in (6.1.3) as 
$$
f^{[\eta]}(w,z)=\frac{1}{2}\Big(\bar\eta\langle \psi(w)\psi(z)\rangle_\theta+\eta\langle\psi(w)\psi^*(z)\rangle_\theta\Big)$$ 
(they denote this $f^{[\eta]}$ simply by $f$). $f^{[\eta]}$ is then a massive holomorphic function, in their sense (i.e., satisfying (6.0.1)), which reads:
\begin{align}\label{eq:BVPf_eta}
\bar\partial_w f^{[\eta]}(w,z)=-\ii\theta(w)\overline{f^{[\eta]}(w,z)}
\end{align}
and satisfies $f^{[\eta]}(w,z)=\bar\eta(w-z)^{-1}+o(|w-z|^{-1})$ (note that the antiholomorphic correlation thus does not contribute to the singularity near $z = w$) as well as the boundary condition 
\begin{align}
\overline{f^{[\eta]}}(w,z)=\tau(w) f^{[\eta]}(w,z), w \in \partial \Omega
\end{align}
(see Definition 6.4 in \cite{PVW}; note that there are no punctures in our case, i.e., the set $\{a_1,...,a_n\}$ is empty). By taking $\eta=1$ and $\eta=\ii$, this implies the holomorphic and antiholomorphic correlations satisfy the bulk problem, for $z \neq w \in \Omega$,
\begin{align}\label{eq:bulk:problem}
\bar\partial_w \langle \psi(w)\psi(z)\rangle_\theta=-\ii\theta(w)\overline{\langle \psi(w)\psi^*(z)\rangle_\theta} \quad ; \quad
\bar\partial_w \langle \psi(w)\psi^*(z)\rangle_\theta=-\ii\theta(w)\overline{\langle \psi(w)\psi(z)\rangle_\theta}, 
\end{align}
the singularity
\begin{equation}\label{eq:singularity:PVW}
\langle \psi(w)\psi(z)\rangle_\theta = 2(w-z)^{-1} + o(|w-z|^{-1})
\quad ; \quad 
\langle \psi(w)\psi^*(z)\rangle_\theta = o(|w-z|^{-1}).
\end{equation}
and the boundary condition for $w\in \partial \Omega$
\begin{equation}\label{eq:bc:PVW}
\overline{\langle\psi(w)\psi(z)\rangle_\theta}=\tau(w){\langle \psi(w)\psi^*(z)\rangle_\theta}.
\end{equation}

Using the residue theorem, one sees that the singularity~\eqref{eq:singularity:PVW} contributes a Dirac term $4\pi\delta_w(z)$ to the bulk equation (this is the exact same proof as in Equation~\eqref{eq:proof:dirac:term}), thus the pair $(H_0(w,z), H_1(w,z)) = (\langle \psi(w)\psi(z)\rangle_\theta,  \langle \psi(w)\psi(z)\rangle_\theta)$ satisfy the boundary value problem which we denote by $\BV_3(\Omega, \theta)$
\begin{equation}\label{eq:BV3:alpha}
\begin{cases} 
\bar\partial_w H_0(w,z)=-\ii\theta(w)\overline{H_1(w,z)}+4\pi\delta_w(z)
\quad ; \quad
\bar\partial_w H_1(w,z)=-\ii\theta(w)\overline{H_0(w,z)} & \text{ on }\Omega^2\\
\overline{H_0(w,z)}=\tau(w)H_1(w,z) & \text{ if } w \in \partial \Omega.
\end{cases}
\end{equation}

All in all, we obtain that
\begin{equation}\label{eq:BV:PVW}
    \chi^{(\mathbf{s})}_{\mathrm{SG}(\rho)}(z_1, \ldots, z_n) 
    = \Big\langle \prod_{j=1}^n \partial^{(s_j)}_{z_j} \phi(z_j) \Big\rangle_{\mathrm{SG}(\rho|\Omega)}
    =  (2\sqrt{\pi}\ii)^{-n}\prod_j (-1)^{s_j}\det\Big [\langle\psi^{s_i}(z_i)\psi^{s_j}(z_j)\rangle_\alpha\Big]
\end{equation}
where $(H_0(w,z), H_1(w,z)) = (\langle \psi(w)\psi(z)\rangle_\theta,  \langle \psi(w)\psi(z)\rangle_\theta)$ satisfy the boundary value problem $\BV_3(\Omega, \theta)$

We finally give a proof of Theorem \ref{thm:fermionic} which identifies the holomorphic and antiholomorphic derivative correlation of our limiting height function with those of the sine-Gordon model. 
\begin{proof}[Proof of Theorem~\ref{thm:fermionic}]
Recall that by Lemma~\ref{lem:detgaugechange}, the correlation kernel of the holomorphic and antiholomorphic derivatives (associated with sign $\mathbf{s} \in \{0,1\}^n$) is given by 
$$
v^{(\mathbf{s})} ( z_1, \ldots, z_n) = \det \left( \tilde F^{(s_j)}_{s_i + s_j} (z_i, z_j) \right)_{i \neq j}.
$$
Let 
\begin{align}\label{eq:F_00F_01}
\begin{cases}H_0(w,z) &= -8\pi\ii \tilde  F_0(w,z)\\
H_1 (w,z) & =-8\pi \ii \overline{\tilde{F}_1}(w,z),
\end{cases}
\end{align}

The first equation of $\BV_2(\Omega, \alpha)$ implies
\begin{align*}
   \begin{cases}
    \dzbar_w H_0(w,z) & = -8\pi\ii \dzbar_w \tilde F_0(w,z)  = - \frac{\tilde \alpha(w)}{2} \overline{H_1} (w,z) 
    + 4\pi\delta_z(w) 
    \\
    \dzbar_w F_{0,1}(w,z) & = -8\pi\ii \dzbar_w \overline{\tilde F_1}(w,z) = - \frac{\tilde \alpha(w)}{2} \overline{F_{0,0}} (w,z)
    \end{cases}
\end{align*}

This is exactly the first equation of $\BV_3(\Omega, \theta)$ if we take $\theta = \tilde \alpha/{2\ii}$. The boundary values for $(\tilde F_0, \tilde F_1)$ in $\BV_2(\Omega, \tilde\alpha)$ imply 
$$
H_1 (w,z) 
=-8\pi \ii \overline{\tilde{F}_1}(w,z)
= 8\pi\ii \bar \tau(w) \overline{\tilde{F}_0}(w,z)
= \bar \tau (w) \overline{H_0}(w,z); \ w \in \partial \Omega. 
$$

We have proved that $(\tilde F_0, \tilde F_1)$ satisfies $\BV_2(\Omega, \tilde\alpha)$ if and only if $(H_0 (w,z),H_1(w,z))$ verifies the boundary value problem $\BV_3(\Omega, \theta)$ for $\theta = \tilde \alpha/{2\ii}$ (the proof of the "if" part of the statement is exactly the same).
Thus, when $(\tilde F_0, \tilde F_1)$ satisfies $\BV_2(\Omega, \tilde\alpha)$ and the solution to $\BV_2(\Omega, \tilde\alpha)$ is unique as in the statement of Theorem~\ref{thm:fermionic}, 
\begin{equation}\label{eq:H=psi}
\Big(H_0(w,z), H_1(w,z)\Big) = \Big(\langle \psi(w)\psi(z) \rangle_{\theta},\langle \psi(w)\psi^*(z)\rangle_\theta\Big).
\end{equation}

It remains to connect the entries of the matrix describing the correlation kernel $v^{(\mathbf{s})}$ to those of the sine-Gordon model $\chi^{(\mathbf{s})}$. To this end, let $G_s(w,z) = (-1)^s\ii \tilde F_s(w,z)$. Then 
$$
 G^{(s_j)}_{s_i + s_j} = (-1)^{s_i+s_j}\ii^{(s_j)} \tilde F^{(s_j)}_{s_i + s_j}
 = (-1)^{s_i}\ii \tilde F^{(s_j)}_{s_i + s_j}
$$
so the matrix with entries given by  the left hand side differ by a factor $\ii$ and an additional factor $(-1)$ for each line where $s_i = 1$. Hence
\begin{align}
    \det\left(\tilde F^{(s_j)}_{s_i+ s_j}(z_i, z_j)\right)_{i\neq j} & = (- \ii)^n \det \left( G^{(s_j)}_{s_i + s_j}(z_i, z_j) \right)_{i\neq j} \prod_{i=1}^n (-1)^{s_i}
\end{align}
Now observe that $$
G^{(s_j)}_{s_i + s_j} (z_i,z_j) = (-8\pi)^{-1}H_{s_i+s_j}^{(s_i)}(z_i, z_j) = (-8\pi)^{-1}\langle \psi^{s_i}(z_i) , \psi^{s_j}(z_j)\rangle.
 $$
The first equation is easily verified for $s_i =0$, $s_j =0$ or $s_j = 1$ by \eqref{eq:F_00F_01}, and follows for $s_i = 1$ and $s_j = 0,1$ by conjugation. 
The second equation follows from Equation~\eqref{eq:H=psi} in the cases $s_i = 0$ and by conjugation for the cases $s_i = 1.$
Consequently,
\begin{align}
     v^{(\mathbf{s})}(z_1, \ldots, z_n) & = (\ii/8\pi)^n \prod_{i=1}^n (-1)^{s_i} \det\Big( \langle \psi^{s_i}(z_i) , \psi^{s_j}(z_j)\rangle\Big)_{i\neq j}\nonumber \\
     & = (\ii/8\pi)^n \prod_{i=1}^n (-1)^{s_i} \cdot \Big\langle \prod_{j=1}^n (2\ii^{(s_j)}\sqrt{\pi}\partial^{(s_j)} \phi(z_j))\Big\rangle_{\mathrm{SG}(-\frac{4\theta}{\pi}|\Omega)} \label{eq:psiSG2}\\
     & = \left(-4\sqrt{\pi}\right)^{-n}  \Big\langle \prod_{j=1}^n \partial^{(s_j)} \phi(z_j))\Big\rangle_{\mathrm{SG}(-\frac{4\theta}{\pi}|\Omega)}\label{eq:SGconclusion} 
\end{align}
where in \eqref{eq:psiSG2} we used the relation between the fermionic correlations and the sine-Gordon model \eqref{eq:psiSG1}. We also used the fact that the function $\theta$ does not need to be compactly supported (as mentioned in \eqref{eq:rho_notcompact}). We briefly provide a justification for this last fact below.
Altogether, this proves \eqref{eq:SGconclusion} and thus concludes the proof of Theorem \ref{T:SG}.
\end{proof}

\begin{lemma} \label{Cor:PVW_notcompact} Let $\Omega = \mathbb{D}$ be the unit disc and let $\alpha:\bar \Omega\rightarrow \mathbb{R}$ be smooth up to and including the boundary ($\alpha \in C^\infty(\bar \Omega))$. Fix signs $(s_j)_{1\le j \le n} \in \{0,1\}^n$ and some test functions $f_1, \ldots,f_n$ such that $\partial^{(s_j)} f_j$ have disjoint support. Then
    \begin{align}
\Big\langle \prod_j (2\ii^{(s_j)}\sqrt{\pi}\partial^{(s_j)} \phi(f_j))\Big\rangle_{\mathrm{SG}(-\frac{4\theta}{\pi}|\Omega)}&=\int\det\Big [\langle\psi^{s_i}(z_i)\psi^{s_j}(z_j)\rangle_\alpha\Big]_{i,j}\prod_jf_j(z_j)\mathrm{d}z
\end{align}
\end{lemma}
\begin{proof}
This follows by checking the convergence of the right hand side in \eqref{eq:PVWtheorem}. This is already done in Lemma 6.39 in \cite{PVW}, namely, their results imply that for $\kappa_b >0$ fixed, if
\begin{align}
\text{dist}(z,\partial \Omega)^2|\nabla\alpha_N(z)|\leq \kappa_b, && \text{dist}(z,\partial \Omega)^2|\nabla\alpha(z)|\leq \kappa_b
\end{align}
and $\alpha_N\rightarrow \alpha$ pointwise, then the two-point Ising correlators converge. The convergence of the integrals follows from the fact that the $f_i$ have disjoint support and uniform bounds and the massive holomorphicity (resp. antiholomorphicity) of the fermionic correlators. 
\end{proof}

\begin{remark}
    Note that a consequence of the identification \eqref{eq:H=psi} is that $F_0 (w,z) = - F_0(z,w)$ and $ F_1 (w,z) = - \bar F_1(z,w)$, since fermions anticommute and since $\overline{\langle \psi(w) , \psi^*(z)\rangle} =\langle \psi^*(w), \psi(z) \rangle $. This is by no means obvious from the definition of $F_0 $ and $F_1$.  
\end{remark}

\subsection{Fredholm regularised determinant formula (proof of Theorem \ref{thm:schatten})}

\label{SS:Fredholm}

We include here a few brief reminders concerning such regularised determinants as well as the related notion of $p$-th Schatten class, and refer to the book by Barry Simon \cite[Chapter 5]{Simon} for more details. The idea is probably more familiar in the context of trace-class operators, so we start with this. Roughly, the idea is the following. Let $H$ be a Hilbert space and let $A:H \to H$ be a positive, bounded linear operator. The trace of $A$ is then the (possibly infinite) series $\Tr(A) = \sum_n \langle A e_n, e_n \rangle$ for a choice of orthonormal basis $(e_n)_{n\ge 1}$ of $H$ (the convergence and value of the trace does not depend on the choice of orthonormal basis).
Let $K: H \to H$ be a bounded operator, not necessarily positive. We say that $K$ is \emph{trace-class} if the series $\Tr(|K|)< \infty $, where $|K | = \sqrt{K^*K}$. When $K$ is trace-class, one can check that the $K^2, K^3, \ldots$ are all trace class operators, with $|\Tr(K^n)|\le \Tr(|K|)^n$. This entails that the power series $\sum_{n\ge 1} (-1)^n \mu^n \Tr (K^n)/ n$ has positive radius of convergence. By definition, we set
$$
\det_1(I + \mu K) = \exp \left( \sum_{n=1}^\infty (-1)^{n+1}\frac{\mu^n \Tr (K^n)}n \right)
$$
the Fredholm-regularised determinant of $I + \mu K$ (for $\mu$ small enough). (This generalizes the well-known expression of the log of a determinant $I+A$ for a finite-dimensional matrix $A$). 

Now let $p\ge 1$ be an integer. We say that $K$ lies in the Schatten class $S_p$ of order $p$ if $\|K\|_p: = \Tr [(\sqrt{K^*K})^p]^{1/p}< \infty$. In particular, $K \in S_1$ if and only if $K$ is trace-class. Also, $K\in S_4$ if and only if $\Tr [(K^*K)^2]< \infty$. The condition $K \in S_p$ is weaker than $K\in S_1$ if $p\ge 1$: that is, $S_1 \subset S_p$. However, even if we do not assume that $K$ is trace-class but only $K\in S_p$ we can still make sense of the Fredholm regularized determinant $\det_p(I + \mu K)$ of order $p$ (for $\mu $ small), namely via the formula
\begin{equation}\label{eq:Schatten_p}
\det_p(I+\mu K)=\exp\Big(\sum_{n\geq p}\frac{(-\mu)^{n+1}}{n}\text{Tr}(K)^n\Big).
\end{equation}



Recall our formula for the moments of the limiting height field given in Corollary \ref{thm:fermionic}, which concretely boils down to \eqref{eq:momentsv}. For convenience, this reads  
\begin{align}\label{eq:mometsv2}
\EE_{(\Omega, \alpha)} [ (\phi, f)^n] = \int \ldots \int f(z_1)\ldots f(z_n) V_{(\Omega, \alpha)} (z_1, \ldots, z_n) \mathrm{d}A(z_1) \ldots \mathrm{d}A(z_n)
\end{align}
where
\begin{align}
    V(z_1, \ldots, z_n) = V_{\Omega; \alpha} (z_1, \ldots, z_n) = \sum_{\mathbf{s} \in \{0,1\}^n} \int_{\path_1}\cdots \int_{\path_n}  v^{(\mathbf{s})}(z_1, \ldots, z_n) \prod_{i=1}^n \mathrm{d}z_i^{(s_i)}.
\end{align}
and 
\begin{equation}\label{eq:fermionic_inverseDirac2}
v^{(\mathbf{s})} (z_1, \ldots, z_n) = \prod_{i=1}^n \mathbf{i}^{-1+2s_i}
\det \Big((\slashed \partial_\alpha)_{2-s_i,1+s_j}^{-1}(z_i,z_j)\Big)_{i\neq j}.
\end{equation}
First, we need a trace formula for the cumulants. By the moments-to-cumulants map (see \cite[(3.27)]{PeccatiTaqqu})
\begin{align}\label{eq:cumulants-moments}
    \Big <\phi(f);...;\phi(f)\Big>_{(\Omega, \alpha)}=\sum_\pi(|\pi|-1)!(-1)^{|\pi|-1}\prod_{B\in \pi}\mathbb{E}_{(\Omega,\alpha)}[(\phi,f)^{|B|}].
\end{align}

Since the moments are determinantal, we obtain an expression for the cumulants by expanding the determinant as a sum over permutations; only the cyclic permutations remain. This yields the following expression for the cumulants:

\begin{lemma}
The cumulants satisfy
\begin{align*}
& 
\Big\langle\phi(f);...;\phi(f)\Big\rangle_{(\Omega, \alpha)}\\
&=
(n-1)!(-1)^{n+1}\sum_{s_1,...,s_n}\int_{\Omega^n}dxf(x_1)...f(x_n)\int_{\gamma(x_1)}...\int_{\gamma(x_n)}\prod_{j=1}^n\langle \psi^{s_j}(z_j)\psi^{s_{j+1}}(z_{j+1})\rangle\prod_{j=1}^n\mathrm{d}z_j^{(s_j)}.
%
\end{align*}
where $\langle \psi^{s}(z)\psi^{r}(w)\rangle := \mathbf{i}^{-1+2s}(\slashed \partial_\alpha)_{2-s,1+r}^{-1}(z,w)$.
\end{lemma}

\begin{proof}
Insert \eqref{eq:mometsv2} into \eqref{eq:cumulants-moments} to get
\begin{align}
\langle \phi_\eps(x_1);...;\phi_\eps(x_n)\rangle=\sum_{s_1,...,s_n}\int_{\gamma_1}...\int_{\gamma_n}\prod_{j=1}^n\mathrm{d} z_j^{s_j}\sum_\pi (|\pi|-1)!(-1)^{|\pi|-1}\prod_{B\in \pi}\det\Big(K^{s_i,s_j}(z_i,z_j)\Big)_{i,j\in B}
\end{align}
Expanding the determinant,
\begin{align}
&\sum_\pi(|\pi|-1)!(-1)^{|\pi|-1}\prod_{B\in \pi}\sum_{\sigma_B\in S_B}\text{sgn}(\sigma_B)\prod_{j \in B}K^{s_j,s_{\sigma_B(j)}}(x_j,x_{\sigma_B(j)})\\
&=\sum_{\sigma\in S_n}\text{sgn}(\sigma)\prod_{j=1}^nK^{s_j,s_{\sigma(j)}}(x_j,x_{\sigma(j)})\sum_{\pi}\mathbbm{1}_{\sigma=\prod_{B\in \pi}\sigma_B}(|\pi|-1)!(-1)^{|\pi|-1}
\end{align}
where we used that $\text{sgn}(\prod_{B\in \pi}\sigma_B)=\prod_{B\in \pi}\text{sgn}(\sigma_B)$. Now recall that any $\sigma$ has a unique cycle decomposition, and the condition that $\sigma=\prod_{B\in \pi}\sigma_B$ is the condition that each cycle $c$ of $\sigma$ lies in a block $B$ of $\pi$, in the sense that $\mathrm{Im}(c)\subset B$. If $|\pi|=r$ and there are $k$ cycles in $\sigma$, note that the number of ways to partition $k$ distinct objects (the cycles) into $r$ subsets (the blocks of $\pi$) is precisely the Stirling number of the second kind, $S(k,r)$. From the combinatorial identity
\begin{align}
    \sum_{r=1}^kS(k,r)(r-1)!(-1)^{r-1}=\delta_{k,1}
\end{align}
(where $\delta$ here is the Kronecker delta) we thus deduce:
\begin{align*}
& \lim_{\eps\rightarrow 0}\Big\langle\phi_\eps(x_1);...;\phi_\eps(x_n)\Big\rangle_{(\Omega, \alpha)}\\
&=
\sum_{s_1,...,s_n}\int_{\gamma_1, \ldots, \gamma_n}\sum_{\sigma}\textnormal{sgn}(\sigma)\prod_{j=1}^n 
\langle \psi^{s_j}(z_j)\psi^{s_{\sigma(j)}}(z_{\sigma(j)})\rangle
\prod_{j=1}^n\mathrm{d}z^{(s_j)},
\end{align*}
where the sum is over all \emph{cyclic} permutations.
In particular, $\text{sgn}(\sigma)=(-1)^{n+1}$. Furthermore, all cyclic permutations contribute the same:
indeed for a given cyclic permutation $\sigma = (i_1 ...i_n)$, where $i_j$ are all distinct, consider the permutation $\pi$ such that $\pi(i_j)=j$. Now relabel the variables $s'_{\pi(j)}=s_j, x'_{\pi(j)}=x_j, z'_{\pi(j)}=z_j$, observe that 
\begin{align}
\prod_{j=1}^n\langle \psi^{s_j}(z_j)\psi^{s_{\sigma(j)}}(z_{\sigma(j)})\rangle=\prod_{j=1}^n\langle \psi^{s_{\pi^{-1}(j)}}(z_{\pi^{-1}(j)})\psi^{s_{\sigma(\pi^{-1}(j))}}(z_{\sigma(\pi^{-1}(j))})\rangle=\prod_{j=1}^n\langle \psi^{s'_j}(z'_j)\psi^{s'_{j+1}}(z'_{j+1})\rangle
\end{align}
where the first equality holds by commutativity and the second equality holds since $\pi(\sigma(\pi^{-1}(j)))=\pi(\sigma(i_j))=\pi(i_{j+1})=j+1$. Hence, each term in the sum over $\sigma$ is equal, and since there are $(n-1)!$ cyclic permutations we deduce the lemma.
\end{proof}

 We obtain
\begin{align}
    \Big <\phi(f);...;\phi(f)\Big>_{(\Omega, \alpha)}=(n-1)!\sum_s\int_{\Omega^n\times(\prod_i\gamma_i)}\prod_i f(x_i)\prod_i \mathbf{i}^{-1+2s_i}(\slashed \partial_\alpha)_{2-s_i,1+s_{i+1}}^{-1}(z_i,z_{i+1}) \prod_i \mathrm{d}x_i \mathrm{d}z_i^{(s_i)}.
\end{align}
If we parametrise $z_i=\gamma(x_i)(t_i)$, so that $\mathrm{d}z_i^{(s_i)}=\dot\gamma(x_i)^{(s_i)}(t_i)\mathrm{d}t_i$ we get
\begin{align}
&(n-1)!\sum_s\int_{\Omega^n\times(\times_i\gamma_i)}\prod_i f(x_i)\prod_i K_{s_i,s_{i+1}}((x_i,t_i);(x_{i+1},t_{i+1})) \prod_i \mathrm{d}x_i \mathrm{d}t_i\label{temp2352sf}\\
&= (n-1)! \Tr((fK)^n)\label{temp45yhtfd}
\end{align}
where 
\begin{align}
K_{r,s}((x,t),(y,t'))=\ii^{-1+2r}(\slashed\partial_\alpha)^{-1}_{2-r,1+s}(\gamma(x)(t),\gamma(y)(t'))\dot \gamma(x)^{(r)}(t)\label{defKop}
\end{align}
defines an operator $K:\mathcal{H}\rightarrow \mathcal{H}$ with
\begin{align}
\mathcal{H}=L^2(\Omega\times [0,1])\oplus L^2(\Omega\times[0,1]).
\end{align}

Note in the above that we have been somewhat imprecise in how we select the paths so that the operator $K$ is truly the same for each $n$. We select these paths as follows, in the case that $\Omega$ is the unit disc $\mathbb{D}$, define the paths with endpoint at $x\in \Omega$ as the horizontal path $(\sqrt{1-x_2^2}(1-t)+tx_1,x_2),$ $t\in [0,1]$. These paths are disjoint whenever the second coordinate $x_2$ of $x$ differs.

Note that
\begin{align}
    B_n =\{(x_1,...,x_n)\in \Omega^n: \exists i \neq j \text{ with } \Im(x_i)=\Im(x_j)\}
\end{align}
has Lebesgue measure zero, or in words, the set of points where there is some pair $(x_i,x_j)$ on the same horizontal line is measure zero. On $\Omega^n \setminus B_n$, the $n$ paths $\gamma(x_i)$ are disjoint, hence removing $B_n$ from the integral in \eqref{temp2352sf} allows \eqref{temp45yhtfd}.

If instead $\Omega$ is an arbitrary $C^1$ Jordan domain, there is a conformal map $\Phi$ from the unit disc to $\Omega$, and so we select the path at $x\in \Omega$ as the image of the horizontal path in the disc originating at $\Phi^{-1}(x)$. Since the image of $B_n$ under $\Phi$ is measure zero in $\Omega,$ similar to before, we can remove the set $\Phi(B_n)$ from the integral in \eqref{temp2352sf} which gives \eqref{temp45yhtfd}.

\begin{proof}[Proof of Theorem \ref{thm:schatten}.]
We first show that $K$ lies in the fourth Schatten class $S_4=\{L: \text{Tr}[(L^*L)^2]<\infty\}$ (recall the discussion above \eqref{eq:Schatten_p} and in \cite[Chapter 5]{Simon}).
In particular, this gives convergence of the series
\begin{align}
\sum_{n=4}^\infty\frac{\mu^n}{n}\text{Tr}[(fK)^n]< \infty
\end{align}
for $|\mu|$ sufficiently small.

From the definition of $K$ in \eqref{defKop}, the discussion in section 4.1, \eqref{FtildeinF}, the definition of $F_0, F_1$ in terms of $\kappa, \kappa^*$ \eqref{eq:def:F_i} and the near-diagonal estimates for $\kappa, \kappa^*$ \eqref{lem:kappa:kappa_0} we have the estimate
\begin{align}\label{eq:K1/x}
|K^{r,r'}((x,s),(y,t))|\leq \frac{C_\Omega |\dot \gamma(x)^{r}(s)|}{|\gamma(x)(s)-\gamma(y)(t)|} \leq \frac{C}{|\gamma(x)(s)-\gamma(y)(t)|}
\end{align}
where we used that $\Omega$ is a $C^1$ Jordan domain for the boundedness of $\dot \gamma.$ Recall above that we defined the paths in $\Omega$ by fixing a conformal map $\Phi$ from the unit disc to $\Omega$, then the path originating from $z\in \Omega$ was defined as the image of the horizontal line segment in the disc from $\Phi^{-1}(z)$ to the boundary. 
We want to show that
\begin{align}
\textnormal{Tr}[(K^*K)^2]<\infty\label{K4thSchatten}
\end{align}
where
\begin{align}
\text{Tr}[(K^*K)^2]&=\sum_{\xi\in\{\pm\}}\int_{\Omega\times[0,1]}[K^*K\cdot K^*K]^{\xi,\xi}(y,y)\mathrm{d}y\\
&= \sum_{s,u,r_1,r_2\in\{\pm\}}\int_{(\Omega\times[0,1])^4}\overline{K^{r_1,s}(\mathbf{y_1};\mathbf{x})}K^{r_1,u}(\mathbf{y_1};\mathbf{z})\overline{K^{r_2,u}(\mathbf{y_2};\mathbf{z})}K^{r_2,s}(\mathbf{y_2};\mathbf{x})\mathrm{d}\mathbf{y_1}\mathrm{d}\mathbf{y_2}\mathrm{d}\mathbf{z}\mathrm{d}\mathbf{x}\label{eq:K2xx}
\end{align}
where $\mathbf{y}_1=(y,t)$, $y\in \Omega, t\in[0,1]$, etc. Observe that it is sufficient that the kernel $K^*K(\mathbf{x},\mathbf{y})$ is square-integrable on $(\Omega\times[0,1])^2$,
\begin{align}
\int_{(\Omega\times[0,1])^2}|K^*K^{r,r'}(\mathbf{x},\mathbf{y})|^2\mathrm{d}\mathbf{x}\mathrm{d}\mathbf{y}< \infty.\label{HilbertschmidtK*K}
\end{align}
Hence, using \eqref{eq:K1/x} and Fubini, \eqref{K4thSchatten} can be deduced by showing the function on the right hand side of
\begin{align}
|K^*K^{r,r'}((x,s),(y,t))|\leq \int_{\Omega}\int_0^1 \frac{C^2}{|\gamma(x)(s)-\gamma(z)(\tau)||\gamma(y)(t)-\gamma(z)(\tau)|}\mathrm{d}\tau  \mathrm{d}z \label{temp4mmh76}
\end{align}
is square-integrable.

Now, we use the conformal map $\Phi: \mathbb{D}\rightarrow \Omega$ and our assumption that $\partial \Omega$ is Dini-smooth which implies (see \cite[Theorem 3.5]{Pommerenke1992BoundaryMaps}) that $\Phi'$ extends continuously to $\overline{\mathbb{D}}$ and satisfies $\Phi'(z)\neq 0$ for all $z\in \overline{\mathbb{D}}$. We see 
\begin{align}
    |\Phi^{-1}(u)-\Phi^{-1}(v)|= |\int_u^v (\Phi^{-1})'(z)dz|\leq \sup |(\Phi^{-1})'(z)|\text{len(path $u\rightarrow v$)}|\leq C'|u-v|\label{temp43rebcv}
\end{align}
for $u,v \in \Omega$, where we used that $|(\Phi^{-1})'(z)|=|1/\Phi'(\Phi(z))|$ is bounded in the last inequality. Hence 
\begin{align}
    |z-w|/C'\leq |\Phi(z)-\Phi(w)|\label{temp65ytaa}
\end{align}for $z,w\in \overline{\mathbb{D}}$. Make the change of variable $z=\Phi(\xi)$ in the integral in \eqref{temp4mmh76}, we see 
\begin{align}
|K^*K^{r,r'}((x,s),(y,t))|&\leq C^2\int_{\overline{\mathbb{D}}}\mathrm{d}\xi |\Phi'(\xi)|^2\int_0^1\mathrm{d}\tau  \frac{1}{|\gamma(x)(s)-\gamma(\Phi(\xi))(\tau)||\gamma(y)(t)-\gamma(\Phi(\xi))(\tau)|}\\
&\leq C^2C'^2C''^2\int_{\overline{\mathbb{D}}}\mathrm{d}\xi \int_0^1\mathrm{d}\tau\frac{1}{|a-\tilde\gamma(\xi)(\tau)||b-\tilde\gamma(\xi)(\tau)|}\label{temp8uyjhm}
\end{align}
where in the previous inequality, we set $a=\Phi^{-1}(\gamma(x)(s)), b=\Phi^{-1}(\gamma(y)(t))$, $\tilde \gamma$ is the horizontal path in $\mathbb{D}$ such that $\Phi(\tilde\gamma(\xi)(\tau))=\gamma(\Phi(\xi))(\tau)$, and we used the bounds $|\Phi'|\leq C''$ and \eqref{temp65ytaa}. Use Fubini on the integral \eqref{temp8uyjhm} then observe that for a fixed $\tau$, $\tilde\gamma(\xi)(\tau), \xi\in \mathbb{D}$, is just a horizontal shift of the points in the set $\mathbb{D},$ i.e. $\tilde\gamma(\xi)(\tau)=\xi+f(\tau)$ for some function $f$. Using the well-known bound
\begin{align}
\int_K \mathrm{d}z\frac{1}{|z-x|}\frac{1}{|z-y|}\leq C(1+|\log|x-y||)
\end{align}
for any compact set $K$, we deduce 
\begin{align}
|K^*K^{r,r'}((x,s),(y,t))|\leq C'''(1+|\log|a-b||)\leq C''''(1+|\log|\gamma(x)(s)-\gamma(y)(t)||)
\end{align}
where we used \eqref{temp43rebcv} again. Inserting this bound into the left hand side of \eqref{HilbertschmidtK*K}, using Fubini to first integrate over either $x$ or $y$, we see the resulting bound is finite since both $\log$ and $\log^2$ are integrable.

Hence we have proved
 \eqref{K4thSchatten}.
\end{proof}


\appendix

\section{Control of the non-massive discrete Green function and its derivatives for nearby points.}\label{app:green:bound}
In this appendix, we recall classical bounds on the behaviour of the discrete non-massive Green function and its discrete derivatives near the diagonal.
These bounds can be found in~\cite{Ken02} (square lattice case), and could probably be extracted from the proof of~\cite{Li17} in the isoradial case. 
We provide proofs for completeness. 

Recall that when $\Omega$ and $\Primal$ have a straight boundary $L = \Omega \cap B$ with $B = (\beta^*,\eta)$, we denote by $\ph$ the reflection along $L$.
Recall that the discrete Green function $\dGreen$ can be extended to $\Primal_\ph = \Primal \cup (\infPrimal \cap B)$ by the discrete Schwarz principle.
This is also true for the continuous Green function, which can be extended by the continuous Schwarz principle to $\Omega_\ph = \Omega \cup B$.
\begin{lemma}\label{lem:green:bounded}
        Uniformly for $x_1$ in compact subsets of $\Omega$, $x_2 \in \Primal$,
            $$
		      \dGreen(x_1,x_2) = O\left(1+|\log|x_1-x_2||\right).
            $$ 
        If $\Omega$ and $\Primal$ have a straight boundary near $\beta^*$ (say on the ball $B = B(\beta^*, \eta)$, so that $L = \partial\Omega \cap B$ is straight, recall that $\Gamma_\ph = \Gamma \cup \ph( \Gamma \cap B)$ and $\Omega_\ph = \Omega \cup B$, where $\ph$ is the reflection across the straight line $L$ defined in Section \ref{subsec:straight}.
        Recall that in this case, $\dGreen$ is naturally extended to $\Primal_\ph$.
        Then, uniformly for $x_1$ in compact subsets of $\Omega_\ph$, $x_2 \in \Primal_\ph$,
            $$
		      \dGreen(x_1,x_2) = O\left(1+|\log(|x_1-x_2| \wedge |\ph(x_1)-x_2|)|\right).
            $$ 
    \end{lemma}
    \begin{proof}
        Let $\cC \subset \Omega$ be a compact subset. 
        Recall that $\dfreeGreen$ denotes the non-massive Green function on $\infPrimal$, which we introduced in Definition~\eqref{def:dfreeGreen}. 
        For $x_1, x_2 \in \Primal$, let
        \begin{equation}\label{eq:G=Gfree-H}
            H(x_1,\cdot) = \dfreeGreen(x_1,\cdot) - \dGreen(x_1,\cdot).
        \end{equation} 
        By definition of $\dfreeGreen$ and $\dGreen$, $H(x_1, \cdot)$ is harmonic in $\Primal$ and coincides with $\dfreeGreen(x_1,\cdot)$ on $\partialout \Primal$. 
        On the one hand, Equation \eqref{eq:green:asymptotic} implies that $H(x_1,\cdot) = O(1)$ uniformly for $x_1 \in \Primal \cap \cC$, $x_2 \in \partialout \Primal$, since $\dist(\cC, \partialout \Primal) > 0$.
        Hence by the maximum principle this remains true on all $\Primal$. 
        On the other hand, Equation \eqref{eq:green:asymptotic} implies that $\dfreeGreen(x_1,x_2) = O(1+|\log|x_1-x_2||)$, uniformly for $x_1,x_2 \in \infPrimal$.
        The result follows.

        When $\Omega$ and $\Primal$ have a straight boundary $L$, there is a classical improvement of this argument due to Kenyon~\cite{Ken02} (see the proof of his Corollary~19 for the square lattice case).
        Let $\tcC \subset \Omega_\ph$ be a compact subset. 
        Recall that  $\infPrimal \cap B(\beta^*,\eta)$ is symmetric along $L$, and that $\ph$ denotes the symmetry along $L$. 

        Since $\tcC \setminus B(\beta^*,\eta)$ is a compact subset of $\Omega$, we can apply the first point to obtain that $\dGreen(x_1,x_2) = O(1+|\log|x_1-x_2||)$ uniformly for $x_1 \in \tcC \setminus B(\beta^*,\eta)$ and $x_2 \in \Primal$. 
        Since $\dGreen(x_1, x_2) = \dGreen(x_1, \ph(x_2))$ for $x_2 \in B(\beta^*,\eta)$, we obtain that $\dGreen(x_1,x_2) = O(1+|\log(|x_1-x_2|\wedge |x_1 - \ph(x_2)||)$ uniformly for $x_1 \in \tcC \setminus B(\beta^*,\eta)$ and $x_2 \in \Primal \cup B(\beta^*,\eta)$.

        For $x_1 \in \tcC \cap B(\beta^*,\eta)$, $x_2 \in \Primal$, let 
        \begin{equation}\label{eq:G=Gfree-phGfree-H}
            \tH(x_1,x_2) = \dfreeGreen(x_1,x_2)-\dfreeGreen(\ph(x_1),x_2) - \dGreen(x_1,x_2).
        \end{equation}
        By definition of $\dGreen$ and $\dfreeGreen$, $\tH(x_1,\cdot)$ is harmonic in $\Primal$ and coincides with $\dfreeGreen(x_1,\cdot)-\dfreeGreen(\ph(x_1),\cdot)$ on $\partialout \Primal$.
        By symmetry of $\infPrimal \cap B(\beta^*,\eta)$ and locality of the discrete free Green function (see Remark~\ref{rem:locality} and before), $\tH$ vanishes on $\partialout \Primal \cap  L$. 
    Hence, Equation \eqref{eq:green:asymptotic} implies that $H(x_1,x_2)$ is $O(1)$ uniformly for $x_1 \in \Primal \cap \tcC$ and $x_2 \in \partialout \Primal$ since $\dist(\tcC, \partial \Omega \setminus L)>0$. 
    By the maximum principle this remains true for all $x_2 \in \Primal$. 
    On the other hand, Equation~\eqref{eq:green:asymptotic} implies that $\dfreeGreen(x_1,x_2) = O(1+\log|x_1-x_2|)$ uniformly for $x_1, x_2 \in \RR^2$,
    which concludes the proof.
    \end{proof}

The same argument can be applied to bound the derivative of the discrete function with Dirichlet boundary conditions. We first need a bound on the derivative of the discrete free Green function, which is a direct consequence of the asymptotic~\eqref{eq:green:asymptotic}:
\begin{lem}\label{lem:bound:dpartial:dfreegreen}
    Uniformly in $x_1, w_2, \eps$ and the isoradial graph $\infPrimal$,
    $$
        \dfreeGreen(x_1,x_2^+)-\dfreeGreen(x_1,x_2^-) = O\left(\frac{\eps}{|x_1-w_2|}\right)
    $$
\end{lem}
\begin{rmk}
    Note that this is much better than what we could obtain using Lemma~\ref{lem:lipschitz} with $r = |x_1-w_2|/2$ and Equation~\eqref{eq:green:asymptotic}: 
$$
    \dfreeGreen(x_1,x_2^+)-\dfreeGreen(x_1,x_2^-) = O\left(\frac{\eps \log(|x_1-w_2|)}{|x_1-w_2|}\right)
$$
\end{rmk}

We are now in position to bound the derivative of the discrete non-massive Green function.
\begin{lemma}\label{lem:bound:nearby:derivative}
    Uniformly for $x_1, w_2$ in compact subsets of $\Omega$,
    $$
        \dpartialtwo\dGreen(x_1,w_2) = O\left(\frac{1}{|x_1-w_2|}\right).
    $$
    If $\Omega$ and $\Primal$ have a straight boundary $L$, uniformly for $x_1, w_2$ in compact subsets of $\Omega_\ph$,
    $$
        \dpartialtwo\dGreen(x_1,w_2) = O\left(\frac{1}{|x_1-w_2| \wedge |\ph(x_1)-w_2|}\right).
    $$
\end{lemma}
\begin{proof}
    Let $\cC \subset \Omega$ be a compact subset. Taking the discrete derivative of Equation~\eqref{eq:G=Gfree-H}, we obtain that 
    $$
        \dpartialtwo G_\#(x_1,w_2) = \dpartialtwo\dfreeGreen(x_1,w_2)-\dpartialtwo H(x_1,w_2).
    $$
    where $H(x_1,\cdot): \Primal \to \RR$ is harmonic in $\Primal$ and coincides with $\dfreeGreen(x_1,\cdot)$ on $\partialout \Primal$. 
    As before, by Equation~\eqref{eq:green:asymptotic}, $H(x_1,x_2)$ is uniformly bounded for $x_1 \in \Primal \cap \cC, x_2 \in \partialout \Primal$ so by the maximum principle, it is uniformly bounded for $x_1 \in \cC \cap \Primal$, $x_2 \in \Primal$. 
    By Lemma~\ref{lem:lipschitz}, uniformly for $x_1 \in \cC$, $w_2 \in \cC$
    $$
        \dpartialtwo H(x_1,w_2) = O(1).
    $$
    Together with Lemma~\ref{lem:bound:dpartial:dfreegreen}, this proves the lemma.

    Assume now that $\Omega$ and $\Primal$ have a straight boundary $L$ and $\tcC \subset \Omega_\ph$ is a compact subset.
    Let also $\tcC \subset \mathring{\cD} \subset \cD \subset \Omega_\ph$ be a compact subset. 
     
    If $x_1 \in \tcC \cap B(\beta^*,\eta)$, we apply the same improvement as in Lemma~\ref{lem:green:bounded}. 
    Taking the discrete derivative of Equation~\eqref{eq:G=Gfree-phGfree-H}, we obtain for all $x_1, w_2 \in \Omega_\ph$,
    $$
        \dpartialtwo\dGreen(x_1,w_2) = \dpartialtwo\dfreeGreen(x_1,w_2)-\dpartialtwo\dfreeGreen(\ph(x_1),w_2) - \dpartialtwo \tH(x_1,w_2).
    $$
    By symmetry of $\infPrimal \cap B(\beta^*,\eta)$ and locality of the discrete free Green function (see Section~\ref{subsec:def:green:function}), $\tH$ vanishes on $L$. 
    Together with Equation~\eqref{eq:green:asymptotic}, it implies that $\tH$ is uniformly bounded for $x_1 \in \Primal \cap \tcC$, $x_2 \in \partialout \Primal$.
    By the maximum principle, $\tH$ is uniformly bounded for $x_1 \in \Primal \cap \tcC$, $x_2 \in \Primal$.
    Since $\tH(x_1, \cdot)$ vanishes on $L$, it can be extended to $\Primal_\ph$ by the discrete Schwarz reflection principle.
    The extended function is uniformly bounded for $x_1 \in \tcC \cap B(\beta^*,\eta)$, $x_2 \in \Primal_\ph$. 
    Hence by Lemma~\ref{lem:lipschitz}, $\dpartialtwo \tH(x_1,w_2)$ is uniformly bounded for $x_1 \in \tcC \cap B(\beta^*,\eta), w_2 \in \tcC$. 
    Together with Lemma~\ref{lem:bound:dpartial:dfreegreen}, this concludes the proof when $x_1 \in \tcC \cap B(\beta^*,\eta), x_2 \in \tcC$.

    There is still a small technical detail to fix, that is when $x_1 \in \tcC \setminus B(\beta^*,\eta)$, $x_2 \in \tcC \cap B(\beta^*,\eta)$.
    Let $\tcC \setminus B(\beta^*,\eta) \subset \mathring{\cC} \subset \cC \subset \Omega$ be a compact subset. 
    If $x_1 \in \tcC \setminus B(\beta^*,\eta) \subset \cC, w_2 \in \cC$, the result holds by the first point. 
    If $x_1 \in \tcC \setminus B(\beta^*,\eta), w_2 \in \tcC \setminus \cC$, the result holds by the second point and the maximum principle since $x_1 \to \dGreen(x_1, w_2)$ is harmonic in $\tcC \setminus B(\beta^*,\eta)$ and $\dist(\partial (\tcC \setminus B(\beta^*,\eta)), \tcC \setminus \cC) >0$.
    \end{proof}

\begin{lemma}\label{lem:bound:nearby:second:derivative}
    Uniformly for $w_1,w_2$ in compact subsets of $\Omega$, 
    $$
        \dpartialone \dpartialtwo \dGreen(w_1,w_2) = O\left(\frac{1}{|w_1-w_2|^2}\right).
    $$
    If $\Omega$ and $\Primal$ have a straight boundary $L$, uniformly for $w_1,w_2$ in compact subsets of $\Omega_\ph$, 
    $$
        \dpartialone \dpartialtwo \dGreen(w_1,w_2) = O\left(\frac{1}{|w_1-w_2|^2 \wedge |\ph(w_1)-w_2|^2}\right).
    $$
\end{lemma}
\begin{proof}
    Let $\cC \subset \mathring{\cD} \subset \cD \subset \Omega$ be two compact subsets, $d = \dist(\cC, \cD^c) > 0$. 
    Let $w_1, w_2 \in \cC$.
    The function $x_1 \to \dpartialtwo \dGreen(x_1,w_2)$ is harmonic on $B(w_1, d/2 \wedge |x_1-w_2|/2)$, and it is $O(1/|w_1-w_2|)$ uniformly for $w_1, w_2 \in \cC$, $x_1 \in B(w_1, d/2 \wedge |x_1-w_2|/2)$, by Lemma~\ref{lem:bound:nearby:derivative}. 
    Applying Lemma~\ref{lem:lipschitz} (which works in particular when $m = 0$) to the ball $B(w_1, r+R)$ with $r=d/4 \wedge |x_1-w_2|/4, R = 2r$ gives the lemma.

    If $\Omega$ and $\Primal$ have a straight boundary $L$, let $\cC \subset \mathring{\cD} \subset \cD \subset \Omega_\ph$ be compact subsets. Let $w_1, w_2$ and 
    $d = \dist(\cC, \cD^c)$.
    The function $x_1 \to \dpartialtwo \dGreen(x_1,w_2)$ can be extended by the discrete Schwarz principle to $\Primal_\ph$. 
    The extended function is harmonic on $B(w_1, d/2 \wedge |w_1-w_2|/2 \wedge |\ph(w_1)-w_2|/2)$ and it is $O(1/(|w_1-w_2|\wedge |\ph(w_1)-w_2|)$ uniformly for $w_1, w_2 \in \cC$ and $x_1 \in B(w_1, d/2 \wedge |w_1-w_2|/2 \wedge |\ph(w_1)-w_2|/2)$ by Lemma~\ref{lem:bound:nearby:derivative}. 
    Applying Lemma~\ref{lem:lipschitz} (in the $m=0$ case is just Corollary 2.9 in \cite{ChelkakSmirnov}) to the ball $B(w_1, r+R)$ with $r=d/4 \wedge |w_1-w_2|/4 \wedge |\ph(w_1)-w_2|/4, R = 2r$ gives the lemma.
\end{proof}

\section{Elementary technical results on functions of two variables}

In this appendix, we prove an elementary result on functions of two variables.
\begin{lemma}\label{lem:lip:two:var}
    Let $\cC_1$ be a compact set of $\RR^2$.
    Let $C_1>0$ and $\phi, \psi: \RR \to \RR$ be two continuous increasing functions such that $\phi(0) = \psi(0) = 0$.
    Let $f : \cC_1 \to \RR$ be a function satisfying
    $$
        \begin{aligned}
        \forall (x_1,x_2) &\in \cC_1, |f(x_1,x_2)| \leq C_1\\
        \forall (x_1,x_2), (x_1',x_2) &\in \cC_1, |f(x_1,x_2)-f(x_1',x_2)| \leq C_1\phi(|x_1-x_1'|)\\
        \forall (x_1,x_2), (x_1,x_2') &\in \cC_1, |f(x_1,x_2)-f(x_1,x_2')| \leq C_1\psi(|x_2-x_2'|).
        \end{aligned}
    $$
    Then, for all compact subset $\cC_2 \subset \mathring{\cC_1}$ there exists $C_2 >0$ such that
    $$
        \forall (x_1,x_2), (x_1',x_2') \in \cC_2, |f(x_1,x_2)-f(x_1',x_2')| \leq C_2\phi(|x_1-x_1'|+\psi(|x_2-x_2'|)).
    $$
\end{lemma}

\begin{proof}
    Let $d = \dist(\cC^2, \cC_1^c)$. 
    For all $(x_1,x_2), (x_1',x_2') \in \cC_2$ such that $|x_1-x_1'| \geq d/2$, then
    $$
        |f(x_1,x_2)-f(x_1',x_2')| \leq 2C_1 \leq \frac{2C_1}{\phi(d/2)} \phi(|x_1-x_1'|)
    $$
    For all $(x_1,x_2), (x_1',x_2') \in \cC_2$ such that $|x_1-x_1'| \leq d/2$, $(x_1',x_2) \in \cC_1$ so
    $$
        \begin{aligned}
            |f(x_1,x_2)-f(x_1',x_2')| 
            &\leq |f(x_1,x_2)-f(x_1',x_2)| + |f(x_1,x_2)-f(x_1',x_2)|\\
            &\leq C_1 \phi(|x_1-x_1'|) + C_1 \psi(|x_2-x_2'|).
        \end{aligned}
    $$
    Hence the lemma holds with $C_2 = \max(C_1, 2C_1/\phi(d/2))$.
\end{proof}

\section{Heuristic for identifying the limiting height field with the sine-Gordon model}

\subsection{Coleman's correspondence and sine-Gordon model with  external field}

\label{SS:Coleman}

The goal of this subsection is to give a heuristic argument that the correlations describing the limiting height function in Theorem \ref{thm:height} can be identified with the sine-Gordon model via Coleman's transform, for a generic domain. Note that we only prove this in the case of unit disc, using the construction of the correlation functions the sine-Gordon field given in \cite{PVW}.

Under $\PP_{\mathrm{SG}(\alpha)}$, $\ph$ has formal density
\begin{align}\label{phaseshiftSG}
\exp\Big[-\int_{\Omega} 2\partial \ph(x) \bar\partial \ph(x)+z_0\big \langle e^{-\ii\sqrt{4\pi}\ph},\tilde \alpha(x)\big\rangle  \ \mathrm{d}A(x)\Big] ,
\end{align}
which we rewrite when $z_0 = -2$ as
\begin{align} \label{eq:preColeman}
    \exp\Big[ - 2 \int_\Omega \partial \ph(x) \bar \partial \ph(x) -\int_{\Omega} \tilde \alpha \ e^{\ii\sqrt{4\pi }\ph}+ \overline{\tilde\alpha} \ e^{-\ii\sqrt{4\pi }\ph} \ \mathrm{d}A(x)\Big] 
\end{align}


In \cite{ColemanSGMTM}, Coleman observed the formal equivalence of the sine-Gordon field with the massive Thirring model. We recall from \eqref{eq:fermion_density} that in this model the ``spins'' are fermionic variables   $(\psi(z))_{z\in \Omega}$ where for each $z$, $\psi(z) = (\psi_1(z), \psi_2(z))$ is a vector of size two, whose formal density is given by 
 \begin{equation}\label{eq:fermion_density2}
 \exp\Big[-\int_{\Omega}\mathrm{d}A(x) \ \big( \bar\psi \slashed \partial \psi +\tilde \alpha \bar\psi_1\psi_1+\overline{\tilde \alpha} \bar\psi_2\psi_2 \big) \Big] .
 \end{equation}
Once again we refer to \cite{BenfattoFalcoMastropietro_Thirring} for a rigorous construction. Note the surprising feature that in this correspondence, the nonlinear potential $\cos (\sqrt{4\pi \ph)}$ is transformed into the above quadratic form, which boils down to simply adding a mass term to the field and is thus in a certain sense `trivial'.

\medskip To explain his idea, Coleman started from the observation that the massless GFF and the massless free fermion ``measure'' are equivalent in the sense that the following observables have the same correlations:
\begin{align}\label{temp5tfhh}
\!:\!e^{\ii\sqrt{4\pi}\varphi}\!:\! \ \leftrightarrow A\!:\!\bar\psi_1\psi_1\!:\!\\
\!:\!e^{-\ii\sqrt{4\pi}\varphi}\!:\! \ \leftrightarrow A\!:\!\bar\psi_2\psi_2\!:\!\\
-\ii\partial \varphi \leftrightarrow B\!:\!\bar\psi_2\psi_1\!:\!\label{temp7yvb}\\
\ii\bar\partial \varphi  \leftrightarrow B\!:\!\bar\psi_1\psi_2\!:\!\label{temp7yvv}
\end{align}
where $A = 4\pi e^{-\gamma/2}$ (with $\gamma$ being the Euler--Mascheroni constant) and $B = \sqrt{\pi}$. (In this massless case, this boils down to some simple Cauchy--Vandermonde determinant identities; see \cite[(2.69)]{BauerschmidtWebb}). In the classical case of the sine-Gordon model, $\tilde \alpha $ is a constant and real-valued. If we replace each term in \eqref{eq:preColeman} by its value under this dictionary in this case, this \emph{suggests} that the sine-Gordon model with real, constant  electromagnetic field should indeed be equivalent to the Fermionic measure given by \eqref{eq:fermion_density} still with $\tilde \alpha$ a real constant. Coleman's original idea to verify this involved an asymptotic expansion of both measures when $z_0$ (or equivalently $\alpha$) tends to 0. This has since then been proven rigorously under various assumptions on $\Omega, \alpha$ at or below the free fermion point (the Coleman correspondence is in fact more general than what we state here). For a more detailed discussion of this, at least in the full plane, we invite the reader to consult \cite{BauerschmidtWebb}.




Going back to the case of interest where $\tilde \alpha$ is not constant, we can still use the same equivalence. In that case we find that the law of $\ph$ under $\PP_{\mathrm{SG}(\alpha)}$ (see \eqref{phaseshiftSG})
is equivalent to the \emph{complex-massive} Dirac field with formal density \eqref{eq:fermion_density}.


\begin{rmk}
   The relation \eqref{temp7yvv} has an extra minus sign compared to \cite{BauerschmidtWebb} stemming from a minor sign error in equation (2.11) of Lemma 2.2 in  that paper.
 This has been amended in the arXiv version of that paper.
\end{rmk}

\begin{remark}
    As an aside, in the case that $\alpha(x)=\alpha$ is constant, but still complex,  observe that $\slashed\partial_{\alpha}\slashed\partial_{-\bar\alpha}=(\Delta-|\alpha|^2)I$,  hence the associated Green function is 
\begin{align}
\slashed\partial_{\alpha}^{-1}(w,z)=\slashed\partial_{-\bar\alpha} (\Delta-|\alpha|^2)^{-1}(w,z)I.
\end{align}
Furthermore, in the case of the full plane, we have $(\Delta-|\alpha|^2)^{-1}(w,z)=-\frac{1}{2\pi}K_0(|\alpha||w-z|)$, so
\begin{align}
\slashed\partial_{\alpha}^{-1}(w,z)=\frac{1}{2\pi}\begin{pmatrix}
    \overline{ \alpha}  \ K_0(|\alpha||w-z|) & -2\bar\partial_w K_0(|\alpha||w-z|)\\
    -2\partial_w K_0(|\alpha||w-z|)&\alpha \ K_0(|\alpha||w-z|)
\end{pmatrix}.
\end{align}
We thus obtain an explicit expression for the associated Green function in this case. 
\end{remark}

Using Coleman's correspondence as heuristic, namely \eqref{temp7yvb} and \eqref{temp7yvv} above, it is natural to expect that the derivatives of the sine-Gordon model with electromagnetic field have correlations given by
\begin{align}
\mathbb{E}_{\mathrm{SG}(\alpha)} \Big[\prod_{i=1}^n\ii^{-1+2s_i}\partial^{(s_i)}\ph(x_i)\Big]&=B^n\Big\langle \prod_{i=1}^n\!:\!\bar\psi_{2-s_i}(x_i)\psi_{1+s_i}(x_i)\!:\!\Big \rangle. \label{mSG-Coleman}
\end{align}
Using standard computations for Grassmann integrals (see, e.g., Lemma A.2 in \cite{BauerschmidtWebb}), we arrive at
\begin{align}
\mathbb{E}_{\mathrm{SG}(\alpha)}\Big[\prod_{i=1}^n\ii^{-1+2s_i}\partial^{s_i}\ph(x_i)\Big] &=B^n\det \Big((\slashed \partial_\alpha)_{2-s_i,1+s_j}^{-1}(x_i,x_j)\Big)_{i\neq j},
\end{align}
when the $x_i\in \Omega $ are pairwise distinct. 

\subsection{Effect of centering and gauge change}

Here we explain how Conjecture \eqref{eq:conjBHS2} can be rephrased in terms of Coleman's transform. It was conjectured in \cite{BHS} that the limiting height function field $h$ has a law given by 
\begin{equation}\label{mSG}
\mathbb{P}_{\mathrm{SG}(\alpha)}(\mathrm{d} h) \propto \exp \left[ z_0 \int_\Omega  \langle e^{- \ii \sqrt{\beta} h(x)} ,\ii \alpha(x) \rangle \mathrm{d}x \right] \mathbb{P}^{\mathrm{GFF}\#} (\mathrm {d} h),
\end{equation}
where $\beta = 4\pi$, the constant $z_0 \in \mathbb{R}$ was not identified and might be lattice dependent, and $\mathbb{P}^{\mathrm{GFF}\#} (\mathrm {d} h)$ denotes the law of the Gaussian free field on $\Omega$, normalised in such a way that $\mathbb{E}^{\mathrm{GFF}\#} (h(x) h(y)) = (-2\pi)^{-1} \log |x - y | + O(1)$ as $|y-x | \to 0$. The law $\PP_{\mathrm{mSG}}$ defined \eqref{mSG} is what we call the \textbf{sine-Gordon model with external field}. The external field in question refers to the vector field $\alpha(x)$ which acts like an external magnetic field, tilting the direction of $e^{\ii \sqrt{4\pi} h}$.

To explain how Theorem \ref{thm:height} is consistent with this conjecture, it is important to also specify the boundary conditions of the field in $\mathbb{P}^{\mathrm{GFF}\#}$, which are only implicit in \cite{BHS}, and which (explicitly) are given by the so-called winding boundary conditions (see, e.g., (2.7) in \cite{BLR_DimersGeometry}). Thus, when $\Omega$ is a smooth domain with boundary parametrized by a simple smooth curve $\gamma:[a,b] \to \partial \Omega$ which winds anticlockwise along $\partial \Omega$, and $x = \gamma(a) \in \partial D$ is a reference point (corresponding to the limit of the removed corner in the Temperleyan construction) we can define a function on the boundary by 
$$
u (\gamma(t)) = \frac1{\sqrt{4\pi}}u_0 ( \gamma(t)) ; \text{ where } u_0(\gamma(t)) = W_{\mathrm{int}} (\gamma([a,t]) ),$$
the \emph{intrinsic winding} of the boundary between $x$ and $\gamma(t)$ (see Section 2.1 in \cite{BLR_DimersGeometry}). The constant $(1/\sqrt{4\pi})$ in the expression for $u$ above (in front of $u_0$) is significant, and is not present in \cite{BLR_DimersGeometry}. This reflects the fact that the Gaussian free field in \cite{BLR_DimersGeometry} differs from the one here by a factor of $\sqrt{2\pi}$ (owing to different normalisations of the Green function). The extra factor of $\sqrt{2}$ is $1/\chi$ for $\chi = 1/\sqrt{2}$, the imaginary geometry constant associated with $\kappa = 2$: using notations of \cite{BLR_DimersGeometry}, the limit of the uncentered height function is $(1/\chi) h^{\mathrm{GFF}} + u_0 = (1/\chi) (h^{\mathrm{GFF}} + \frac1{\sqrt{2}} u_0)) $, where $h^{\mathrm{GFF}}$ is a Dirichlet Gaussian free field normalised according to the conventions of \cite{BLR_DimersGeometry}.

It is elementary to check that the function $u$ is invariant under orientation preserving reparametrization of $\gamma$ and thus depends only on $\Omega$ and $x$. The harmonic extension of this function to $\Omega$ is then denoted by $u_{(\Omega, x)}$. Thus, if $\Omega = \DD$ and $x = 1$, $u_{(\DD, 1)} (z) = (4\pi)^{-1/2}\arg(z)$, where $\arg$ is the determination of the argument which has a jump of $-2\pi$ when crossing the positive real line from bottom to top. On the other hand, when $\Omega = \HH$ is the upper half-plane and $x = 0$, say, then $u_{(\HH, 0)} (z) = 0$. To isolate the effect of the boundary conditions we thus write
$$
h = \phi + u
$$
where $h$ has the law $\mathbb{P}^{\mathrm{GFF}\#}$, $u = u_{(\Omega, x)}$, and $\phi$ is a Gaussian free field with Dirichlet boundary conditions (and with covariance normalised similar to above, i.e., $\EE ( \tilde h(x) \tilde h(y) ) = (-2\pi)^{-1} \log |x- y | + O(1)$). We will write $\tilde{\mathbb{P}}^{ \mathrm{GFF}\#}$ for the law of $\tilde h$.

Now observe that, since $\langle z, w\rangle = \Re (z\bar w)$, we can write
\begin{align*}
\langle e^{- \ii \sqrt{\beta} h(x)} , \ii\alpha(x) \rangle 
 = \Re \left( e^{- \ii \sqrt{\beta} (\phi(x) + u(x))} \overline{\ii\alpha} (x) \right)
= \langle e^{- \ii \sqrt{\beta} \phi(x) }, \ii \tilde \alpha (x) \rangle,
\end{align*}
where
\begin{equation}
\tilde \alpha (x) = \alpha(x) e^{\ii \sqrt{\beta} u(x)}
\end{equation}
is a new vector field which corresponds to the electromagnetic field of $\ph$.
Indeed, under $\mathbb{P}_{\mathrm{SG}(\alpha)}$, since $\phi  = h - u$ we have that $\phi$ has law given by 
$$
\exp \left[  z_0 \int_\Omega \langle e^{ - \ii \sqrt{\beta} \phi }, \ii \tilde \alpha (x) \rangle \mathrm{d} x\right] \PP^{\tilde{\mathrm{GFF}}\#} ( \mathrm{d} \phi) 
$$
Since $u$ is a deterministic function, the correlations of the centered height function $h - \EE_{\mathrm{SG}(\alpha)}(h)$ under $\PP_{\mathrm{SG}(\alpha)}$ are identical to those of the centered field $\phi - \EE_{\mathrm{SG}(\alpha)}(\phi)$.

\bibliographystyle{alpha}
\bibliography{references}

\end{document}